\documentclass[12pt]{amsart}

\usepackage[T1]{fontenc}
\usepackage{lmodern}

\usepackage{amsmath, amsthm, amsfonts, amssymb}
\usepackage{mathtools}
\usepackage[margin=0.9in]{geometry}

\usepackage{graphicx}
\usepackage{subcaption}
\usepackage{xcolor}
\usepackage{hyperref}

\title[Chebyshev polynomials]{Chebyshev polynomials in the complex plane and on the real line}

\author{Olof Rubin}
\address{Division of Probability, Mathematical Physics and Statistic, KTH Royal institute of Technology, Stockholm, Sweden}
\email{orubin@kth.se}



\date{February 2026}

\newcommand{\C}{\mathbb{C}}
\newcommand{\D}{\mathbb{D}}

\newcommand{\N}{\mathbb{N}}
\newcommand{\R}{\mathbb{R}}
\newcommand{\bbT}{\mathbb{T}}
\newcommand{\RS}{\overline{\mathbb{C}}}

\newcommand{\cE}{\mathcal{E}}
\newcommand{\cM}{\mathcal{M}}

\newcommand{\cW}{\mathcal{W}}

\newcommand{\sfE}{\mathsf{E}}

\newcommand{\capacity}{\operatorname{Cap}}
\newcommand{\supp}{\operatorname{supp}}
\newcommand{\sgn}{\operatorname{sign}}
\newcommand{\parreauwidom}{\mathcal{PW}}

\renewcommand{\Re}{\operatorname{Re}}
\renewcommand{\Im}{\operatorname{Im}}

\newtheorem{theorem}{Theorem}[section]
\newtheorem{lemma}[theorem]{Lemma}

\newtheorem{corollary}[theorem]{Corollary}

\newtheorem{problem}[theorem]{Problem}

\theoremstyle{remark}
\newtheorem{remark}[theorem]{Remark}

\begin{document}

\begin{abstract}
We present a survey of central developments in the theory of Chebyshev polynomials, 
introduced by P.~L.~Chebyshev and later extended to the complex plane by G.~Faber. 
Our primary focus is their defining extremal property: among all polynomials with a prescribed leading coefficient, 
they minimize the supremum norm on a given compact set. 
Although we do not present new results, we provide --- in selected cases --- new proofs of known theorems and compile a collection of open problems.
\end{abstract}

\maketitle

\section{Introduction}
In 1854 Pafnuty Chebyshev \cite{chebyshev54} introduced the problem of ``best approximation'', which amounts to the following:

\begin{problem}
	Given a continuous function $f:[-1,1]\rightarrow \R$ and a natural number $n\in \N$, determine parameters $a_0,\dotsc,a_{n-1}$ such that
	\[\max_{x\in [-1,1]}|f(x)-a_0-a_1x-\cdots-a_{n-1}x^{n-1}|\]
	is minimal.
	\label{prob:approximating_continuous_functions}
\end{problem}
In modern terms, this minimal deviation represents the distance between the function~$f$ and the space of polynomials of degree at most~$n-1$, measured with respect to the supremum norm on the interval $[-1,1]$. As a consequence of Weierstra{\ss}' approximation theorem, proven in \cite{weierstrass85}, we know that this minimal deviation tends to zero as $n\rightarrow\infty$ for any fixed continuous function $f$. A construction of Bernstein, detailed in \cite{bernstein12}, provides such a sequence explicitly in terms of the given function $f$. Chebyshev, however, had a different focus: his aim was to determine the actual minimizer rather than merely a sufficiently close approximant.
Curiously, Chebyshev's work predates Weierstra\ss' by more than thirty years.
Given the inherent complexity of the general formulation of Problem~\ref{prob:approximating_continuous_functions}, Chebyshev simplifies the conditions, motivated by a Taylor series consideration, and arrives at the following reduced problem.
\begin{problem}
	Given a natural number $n\in \N$, determine parameters $a_0,\dotsc,a_{n-1}$ such that
	\[\max_{x\in [-1,1]}|x^n-a_0-a_1x-\cdots-a_{n-1}x^{n-1}|\]
	is minimal.
	\label{prob:classical_chebyshev_polynomial}
\end{problem}
He further shows that the (as it turns out) unique solution to Problem \ref{prob:classical_chebyshev_polynomial} is given by the coefficients of the polynomial (all square root terms cancel)
\[a_0+\cdots+a_{n-1}x^{n-1} = x^n-\frac{1}{2^{n}}\left[\Bigl(x+\sqrt{x^2-1}\Bigr)^n+\Bigl(x-\sqrt{x^2-1}\Bigr)^n\right].\]
Suitably rephrased,
\begin{equation}
	T_n(x) = \frac{1}{2^{n}}\left[\Bigl(x+\sqrt{x^2-1}\Bigr)^n+\Bigl(x-\sqrt{x^2-1}\Bigr)^n\right]
	\label{eq:chebyshev_first_kind}
\end{equation}
is the monic polynomial of degree~$n$, which minimizes the supremum norm on $[-1,1]$. The use of the letter $T$ to denote these polynomials stems from the French transliteration \emph{Tsch\'{e}bycheff}. Today, the polynomial \eqref{eq:chebyshev_first_kind} is referred to as the \emph{Chebyshev polynomial of the first kind} of degree~$n$. An alternative representation is obtained by writing $x = \cos \theta\in [-1,1]$ in which case
\[x\pm\sqrt{x^2-1} = \cos \theta\pm i\sin\theta.\]
This leads to the well-known formula
\begin{equation}
	T_n(x) =  \frac{e^{in\theta}+e^{-in\theta}}{2^n} = 2^{1-n}\cos n\theta.
	\label{eq:chebyshev_first_kind_trigonometry_formula}
\end{equation}

In Section~\ref{sec:cheb_background} we will show that this formula indeed defines an algebraic polynomial in the variable~\(x\).
From \eqref{eq:chebyshev_first_kind_trigonometry_formula} it follows that the polynomials \(T_n\) attain the values
\(2^{1-n}\) and~\(-2^{1-n}\) alternately at \(n+1\) consecutive points in the interval \([-1,1]\).
This provides the first instance---though certainly not the last---of the phenomenon that extremal values tend to alternate
for minimal polynomials. As it turns out, this alternation property is in fact characteristic of best approximations with
respect to the supremum norm.

%\begin{figure*}
%\centering
%\includegraphics[width=0.3\textwidth]{Figures/Pafnuty_Lvovich_Chebyshev.jpg}
%\caption{Pafnuty Chebyshev (1821-1894).}
%\label{fig:chebyshev}
%\end{figure*}

While Chebyshev's intentions in \cite{chebyshev54} were to perform a theoretical study of a mechanical problem related to Watt's straight-line mechanism, his mathematical investigations are considerably deepened in \cite{chebyshev59}. There he introduces the following problem:

\begin{problem}
	Given a continuous function $f:[-1,1]\rightarrow \R$, a polynomial~$P$ which is strictly positive on $[-1,1]$, and a natural number~$n\in \N$, determine parameters~$a_0,\dotsc,a_{n-1}$ such that
	\[\max_{x\in [-1,1]}\frac{\left|f(x)-a_0-a_1x-\cdots-a_{n-1}x^{n-1}\right|}{P(x)}\]
	is minimal.
	\label{prob:weighted_chebyshev_approximation}
\end{problem}
Again, of special interest to Chebyshev is the reduction to the case \(f(x)=x^{n}\).
In this setting, he determines the exact solution of the problem, a result to which we return in Section~\ref{subs:markov_bernstein}.
A detailed account of Chebyshev's contributions to approximation theory can be found in \cite{goncharov00}.

Following Chebyshev's groundbreaking work, these problems were extended in a variety of directions.
In particular, we mention the works of Markov~\cite{markov84,markov48}, Borel~\cite{borel05}, Faber~\cite{faber20}, Achieser~\cite{achieser28,achieser56}, Bernstein~\cite{bernstein30,bernstein31}, and Widom~\cite{widom69}.
For a historical overview of the early development of the subject by Chebyshev's students, we refer the reader to \cite{achieser87}.

While the polynomial \(T_n\) defined by \eqref{eq:chebyshev_first_kind_trigonometry_formula} exhibits many remarkable properties (see, for instance, \cite{rivlin90}), the primary aim of this text is to survey both classical and modern results on Chebyshev polynomials in the complex plane from the perspective 
that they minimize the supremum norm among all polynomials with the same leading coefficient. 
In particular, we study the complex generalization introduced by Faber \cite{faber20}.
These polynomials are also referred to as Chebyshev polynomials and include the classical Chebyshev polynomials on an interval as special cases.

Our considerations are exclusively in one variable. We refer the reader to \cite{alpan-bayraktar-levenberg26,zakharyuta75,zakharyuta12} and references therein for more on the multivariate case.

Chebyshev polynomials constitute only one instance of best approximation with respect to the supremum norm.
Nevertheless, many of their characteristic properties persist in the solutions of the more general
Problems~\ref{prob:approximating_continuous_functions} and~\ref{prob:weighted_chebyshev_approximation},
as well as in subsequent generalizations.

For the general theory of approximation, we invite the reader to consult \cite{achieser56,cheney66,iske18,lorentz86,rivlin90}.
We emphasize, however, that Chebyshev polynomials are also of independent interest,
motivated, for instance, by their fundamental connections to spectral theory, 
numerical analysis, and potential theory.

\section{Fundamental properties}
\label{sec:cheb_background}
This section covers classical theory. Here, we will settle questions concerning existence and uniqueness of solutions to the problems initially suggested by Chebyshev and their subsequent generalizations, where a monic minimizer is sought. For a recent account of weighted Chebyshev polynomials in the complex plane, we refer the reader to \cite{novello-schiefermayr-zinchenko21}. 

\subsection{An extension of the concept, existence and uniqueness}
Following the notation of \cite{christiansen-simon-zinchenko-I,totik-yuditskii15,widom69}, we let $\sfE\subset \C$ denote a compact subset of the complex plane and $w:\sfE\rightarrow[0,\infty)$ a bounded function on $\sfE$. For any given natural number $n\in \N$ we introduce the quantity
\begin{equation}
	t_n(\sfE,w)\coloneqq \inf_{a_k\in \C}\sup_{z\in \sfE}\Bigl(w(z)\bigl|z^n+a_{n-1}z^{n-1}+\cdots+a_0\bigr|\Bigr).
	\label{eq:chebyshev_norm}
\end{equation}
Notice that we have replaced the minus signs in Problem \ref{prob:weighted_chebyshev_approximation} with plus signs. While this is merely a change in perspective, it signifies that we are now focusing on minimal monic polynomials rather than approximating monomials using polynomials of lower degree. Recall that a function~$f:\sfE\rightarrow \R$ is upper semicontinuous if
\[\limsup_{\substack{\zeta\rightarrow z \\
\zeta\in \sfE}}f(\zeta)\leq f(z)\]
holds at any point~$z\in \sfE$. For future reference, a function \(g:\sfE\to\R\) is lower semicontinuous if \(-g\) is upper semicontinuous. By defining 
\[\hat{f}(z) = \lim_{r\rightarrow 0^+}\Bigl(\sup_{\substack{|z-\zeta|<r\\ \zeta\in \sfE}}f(\zeta)\Bigr)\]
we obtain what is called the upper semicontinuous regularization of $f$. The functions $f$ and $\hat{f}$ share global upper bounds meaning that $\|f\|_\sfE = \|\hat{f}\|_\sfE$ where $\|\cdot\|_\sfE$ denotes the uniform norm on $\sfE$. Importantly for us, if $g:\sfE\rightarrow\R$ is continuous then
\begin{equation}
	\|fg\|_\sfE = \|\hat{f}g\|_\sfE,
	\label{eq:max_absolute_value_upper_sem}
\end{equation}
see e.g. \cite[Lemma 1]{novello-schiefermayr-zinchenko21}. Hence \eqref{eq:max_absolute_value_upper_sem} implies that
\[t_n(\sfE,w) = t_n(\sfE,\hat{w})\]
and furthermore if $a_{0}^\ast,\dotsc,a_{n-1}^\ast$ is an optimal configuration for the minimization of \eqref{eq:chebyshev_norm} with respect to $w$ then this configuration is also optimal with respect to $\hat{w}$. Upper semicontinuous functions always attain their maximum on compact sets \cite[Theorem 2.1.2]{ransford95} and therefore
\[t_n(\sfE,\hat{w})\coloneqq \inf_{a_k\in \C}\max_{z\in \sfE}\Bigl(\hat{w}(z)\bigl|z^n+a_{n-1}z^{n-1}+\cdots+a_0\bigr|\Bigr).\]
Whenever necessary we may replace $w$ by its upper semicontinuous regularization. This doesn't change the value of $t_n(\mathsf{E},w)$ and any optimal configuration remains the same.

At the outset it is not clear that minimizing parameters of \eqref{eq:chebyshev_norm} exist. We begin by showing that this infimum is indeed a minimum. 

If $w$ is positive on at least~$n+1$ distinct points of $\mathsf{E}$, then \[\|wP\|_{\sfE}\coloneqq\sup_{z\in \sfE}\Bigl(w(z)\left|P(z)\right|\Bigr)\]
defines a norm on the finite-dimensional space of polynomials of degree at most $n$. Since norms on finite-dimensional spaces are equivalent (see e.g. \cite[Theorem III.3.1]{conway90}) we conclude the existence of a positive constant $C>1$ such that
\begin{align}
	\begin{split}
		C^{-1}\sup_{z\in \sfE}&\Bigl(w(z)\bigl|a_nz^n+a_{n-1}z^{n-1}+\cdots+a_0\bigr|\Bigr)\\ & \leq \sum_{j=0}^{n}|a_j|\leq C\sup_{z\in \sfE}\Bigl(w(z)\bigl|a_nz^n+a_{n-1}z^{n-1}+\cdots+a_0\bigr|\Bigr)
	\end{split}
	\label{eq:equivalent_norms}	
\end{align}
for any choice of coefficients $a_0,\dotsc,a_{n-1},a_n$. Choose sequences $\{a_j^{(k)}\}_k$ for $j=0,\dotsc,n-1$ such that
\begin{align*}
\lim_{k\rightarrow \infty }\sup_{z\in \sfE}&\Bigl(w(z)\bigl|z^n+a_{n-1}^{(k)}z^{n-1}+\cdots+a_0^{(k)}\bigr|\Bigr) = t_n(\sfE,w).
\end{align*}
It is clear from \eqref{eq:equivalent_norms} that each sequence $\{a_j^{(k)}\}_k$ is bounded. Bolzano-Weierstra{\ss}' theorem implies that there exists a subsequence $k_l$ with the property that $\{a_j^{(k_l)}\}$ has a limit as~\(l\to \infty\) for every $j = 0,\dotsc,n-1$. We introduce the limits $a_j^\ast$ for $j=0,\dotsc,n-1$ so that
\[\lim_{l\to \infty}a_j^{(k_l)} = a_j^\ast.\]
Again, \eqref{eq:equivalent_norms} implies that
\[\sup_{z\in \sfE}\left(w(z)\Bigl|(a_{n-1}^{(k_l)}-a_{n-1}^\ast)z^{n-1}+\cdots +(a_0^{(k_l)}-a_0^\ast)\Bigr|\right)\leq C\sum_{j=0}^{n-1}|a_j^{(k_l)}-a_j^\ast|\rightarrow 0\]
as $l\rightarrow \infty$ and therefore
\[\sup_{z\in \sfE}\Bigl(w(z)\bigl|z^n+a_{n-1}^\ast z^{n-1}+\cdots+a_0^\ast\bigr|\Bigr) = t_n(\sfE,w).\] 
This establishes the existence of a minimizer. We proceed by showing that such a minimizer is unique. The following lemma will be useful in this regard.

\begin{lemma}
Assume that $\mathsf{E} \subset \mathbb{C}$ is a compact set, and let $w : \mathsf{E} \to [0, \infty)$ be an upper semicontinuous function that is positive on at least distinct $n+1$ points. Let $a_0^\ast, \dotsc, a_{n-1}^\ast \in \mathbb{C}$ be such that
	\begin{equation}
		\max_{z\in \sfE}\Bigl(w(z)\bigl|z^n+a_{n-1}^\ast z^{n-1}+\cdots+a_0^\ast\bigr|\Bigr) = t_n(\sfE,w)
		\label{eq:lem_minimimizing_coeff}
	\end{equation}
	then there are (at least) $n+1$ distinct points $z_1,\dotsc,z_{n+1}$ in $\sfE$ such that
	\[\Bigl(w(z_j)\bigl|z_j^n+a_{n-1}^\ast z_j^{n-1}+\cdots+a_0^\ast\bigr|\Bigr)  = t_n(\sfE,w),\quad j=1,\dotsc,n+1.\]
	\label{lem:nbr_extremal_points}
\end{lemma}
\begin{proof}
We argue by contradiction. Let
\[
T(z)=z^n+a_{n-1}^\ast z^{n-1}+\cdots+a_0^\ast,
\]
where the coefficients \(a_0^\ast,\dots,a_{n-1}^\ast\) are chosen so that \(T\) satisfies
\eqref{eq:lem_minimimizing_coeff}.
%	We argue by contradiction. Let 
%	\[T(z) = z^n+a_{n-1}^\ast z^{n-1}+\cdots+a_0^\ast\]
%	where the coefficients are chosen as to satisfy \eqref{eq:lem_minimimizing_coeff}.
	 Assume, in order to derive a contradiction, that there are fewer than $n+1$ points in $\sfE$ where $w|T|$ attains the value $t_n(\sfE,w)$. Denote these points with $z_1,\dotsc,z_m$, so that $m<n+1$. 
	The polynomial
	\[Q(z)=\sum_{j=1}^{m}T(z_j)\prod_{\substack{k=1\\k\neq j}}^{m}\frac{z-z_k}{z_j-z_k}\]
	has degree at most \(m-1\leq n-1\) and satisfies
	$Q(z_j) = T(z_j)$ for $j=1,\dotsc,m$. We consider the perturbed polynomial
	\[T(z)-\varepsilon Q(z)\]
	where $\varepsilon>0$. This difference is a monic polynomial of degree $n$ since $Q$ has degree at most~$n-1$. The triangle inequality applied to the absolute value of the difference at a point~$z$ ensures us that 
	\begin{equation}
		w(z)|T(z)-\varepsilon Q(z)| \leq  (1-\varepsilon)w(z)|T(z)|+\varepsilon w(z)|T(z)-Q(z)|.
		\label{eq:triangle_inequality_perturbation}
	\end{equation}
	Since $T(z_j)-Q(z_j) = 0$ we can find an open neighborhood \(U\) of \(\{z_1,\dotsc,z_m\}\) in \(\sfE\) such that if \(z\in U\), then
	\[w(z)|T(z)-Q(z)|<\frac{t_n(\sfE,w)}{2}.\]
	On the one hand, we obtain from \eqref{eq:triangle_inequality_perturbation} that if \(z\in U\)
	\begin{equation}
		w(z)|T(z)-\varepsilon Q(z)|\leq \left(1-\frac{\varepsilon}{2}\right)t_n(\sfE,w).
		\label{eq:upper_bound_in_delta_domain}
	\end{equation}
	On the other hand there exists some $0<\rho<1$ such that if $z\in \sfE\setminus U$, then \[w(z)|T(z)|\leq (1-\rho)t_n(\sfE,w).\] Here is a subtle yet crucial point: the choice of $\rho$ does not need to account for $\varepsilon$, as it depends solely on $U$. We conclude that, uniformly for $z\notin U$,
	\begin{equation}
		w(z)|T(z)-\varepsilon Q(z)|\leq (1-\rho)t_n(\sfE,w)+\varepsilon \|wQ\|_{\sfE}.
		\label{eq:upper_bound_outside_delta_domain}
	\end{equation}
	Combining \eqref{eq:upper_bound_in_delta_domain} and \eqref{eq:upper_bound_outside_delta_domain} it is clear that by letting $\varepsilon>0$ be sufficiently small, we can  obtain the inequality
	\[\|w(T-\varepsilon Q)\|_{\sfE}<t_n(\sfE,w)\]
	but this contradicts the assumed minimality of $T$.
\end{proof}

With Lemma \ref{lem:nbr_extremal_points} at hand, we can now easily show that there is only one monic polynomial whose weighted deviation from zero is the smallest on a given compact set.

\begin{theorem}
	Let $\sfE\subset \C$ denote a compact set and $w:\sfE\rightarrow [0,\infty)$ a bounded function which is positive for at least $n+1$ points of $\sfE$. There exists a unique monic polynomial of the form
	\[T_n^{\sfE, w}(z) = z^n+a_{n-1}^\ast z^{n-1}+\cdots + a_0^\ast\]
	such that
	\[\|wT_n^{\sfE,w}\|_{\sfE} = t_n(\sfE,w).\] 
	\label{thm:chebyshev_existence_uniqueness}
\end{theorem}
This is the so-called weighted Chebyshev polynomial of degree $n$ with respect to the set $\sfE$ and the weight function $w$.

\begin{proof}
	Assume without loss of generality that $w$ is upper semicontinuous. The existence of a minimizer has already been established. To prove the uniqueness of the minimizer we assume that there are two monic polynomials of degree $n$, denoted $T^{(1)}$ and $T^{(2)}$,  satisfying
	\[
		\|wT^{(1)}\|_{\sfE} = t_n(\sfE,w) = \|wT^{(2)}\|_{\sfE}.
	\]
	Their average is denoted $T = \frac{1}{2}(T^{(1)}+T^{(2)})$. By the triangle inequality 
	\[\|wT\|_{\sfE}\leq \frac{1}{2}\|wT^{(1)}\|_{\sfE}+\frac{1}{2}\|wT^{(2)}\|_{\sfE}= t_n(\sfE,w).\]
	On the other hand, since $T$ is a monic polynomial of degree $n$, the reverse inequality \[\|wT\|_{\sfE}\geq t_n(\sfE,w)\] follows by definition of $t_n$. It turns out that the average polynomial $T$ is actually a minimizer. Lemma \ref{lem:nbr_extremal_points} implies the existence of $n+1$ distinct points $z_1,\dotsc,z_{n+1}$ such that
	\[w(z_j)|T(z_j)| = t_n(\sfE,w).\]
	Since
	\begin{align*}
		w(z_j)|T(z_j)| &= w(z_j)\frac{|T^{(1)}(z_j)+T^{(2)}(z_j)|}{2}\\&\leq \frac{1}{2}w(z_j)|T^{(1)}(z_j)|+\frac{1}{2}w(z_j)|T^{(2)}(z_j)| \leq t_n(\sfE,w)
	\end{align*}
	equality holds throughout. But this is only possible if $\arg T^{(1)}(z_j) = \arg T^{(2)}(z_j)$ and 
	\[|T^{(1)}(z_j)| = t_n(\sfE,w) = |T^{(2)}(z_j)|\]
	for all $j = 1,\dotsc,n+1$.
	
	As a consequence, the difference $T^{(1)}-T^{(2)}$ is a polynomial of degree at most $n-1$ that vanishes at the $n+1$ points $z_1,\dotsc,z_{n+1}$. This can occur if and only if the difference is constantly equal to $0$. Therefore $T^{(1)} = T^{(2)}$ and we have established the existence of a unique minimizer, henceforth denoted $T_n^{\sfE,w}$.
\end{proof}

If the risk of confusing the reader is low and the set $\sfE$ associated with the weight function~$w$ is clearly implied from the weight function, we will simply use the notation $T_n^w$ for the corresponding weighted Chebyshev polynomial. Alternatively, in the case where the weight is given by $w = 1$ we will use the notation $T_n^{\sfE}$. The notation $\D$ will be used for the open unit disk and $\bbT$ for the unit circle. We will further use $\RS = \C\cup\{\infty\}$ to denote the Riemann sphere. Chebyshev polynomials corresponding to compact sets in the complex plane are only known explicitly for a very narrow class of sets and weights. To provide at least one example, we show that
\begin{equation}
	T_n^{\bbT}(z) = z^n.
	\label{eq:cheb_unit_cirlce}
\end{equation}
Let \(P(z)=z^n+a_{n-1}z^{n-1}+\cdots+a_0\) be an arbitrary monic polynomial. Then
\[
\|P\|_{\bbT}
= \max_{|z|=1}\left|z^n\bigl(1+a_{n-1}z^{-1}+\cdots+a_0z^{-n}\bigr)\right|
= \max_{|z|=1}\left|1+a_{n-1}z+\cdots+a_0z^{n}\right|.
\]
Since the latter is the maximum modulus of an analytic function on the unit disk whose value at
\(z=0\) equals \(1\), the maximum modulus principle implies that
\[
\|P\|_{\bbT} \ge 1.
\]
As the polynomial \(z^n\) attains this lower bound, Theorem~\ref{thm:chebyshev_existence_uniqueness}
shows that \(T_n^{\bbT}(z)=z^n\).

This example also demonstrates that the set of extremal points of a Chebyshev polynomial may be infinite.
In contrast, this cannot occur in the unweighted setting when \(\sfE\subset\R\), as we shall show in
Theorem~\ref{thm:alternation_theorem} below. To this end, we will pursue an approach based on a dual
characterization of Chebyshev polynomials.

%Let $P(z) = z^n+a_{n-1}z^{n-1}+\cdots +a_{0}$ be any monic polynomial. Then 
%\[\|P\|_{\bbT} = \max_{|z| = 1}\left|z^n\left(1+a_{n-1}z^{-1}+\cdots+a_0z^{-n}\right)\right| = \max_{|z| = 1}\left|1+a_{n-1}z+\cdots+a_0z^{n}\right|\geq 1\]
%by the maximum modulus principle applied to \(z=0\). Since the polynomial $z^n$ saturates this lower bound we find by Theorem \ref{thm:chebyshev_existence_uniqueness} that $T_n^{\bbT}(z) = z^n$. This example further shows that the number of extremal points of a Chebyshev polynomial can be infinite. This is not the case however in the unweighted setting if the set $\sfE$ is a real subset as we will show in Theorem \ref{thm:alternation_theorem} below. We will pursue and approach based on a dual characterization of Chebyshev polynomials.

%\subsection{Optimal prediction measures and alternation}
%We will proceed by showing that Chebyshev polynomials in the complex plane have orthogonality properties. Given a positive finite measure \(\mu\), supported on a compact set containing at least \(n+1\) points the corresponding monic orthogonal polynomial of degree \(n\) is the extremal polynomial that minimizes
%\[\inf_{a_k}\int \Bigl|z^n+\sum_{k=0}^{n-1}a_kz^k\Bigr|^2d\mu(z).\]
%In other words, the monic polynomial with smallest \(L^2(d\mu)\)-norm. Let \(\cM(\sfE)\) denote the collection of all probability measures supported on \(\sfE\). The following result relates Chebyshev polynomials to the study of orthogonal polynomials.

\subsection{Orthogonality and alternation}
Let \(\mu\) be a positive finite measure supported on a compact set
containing at least \(n+1\) points. Then the corresponding monic orthogonal polynomial of degree \(n\)
is the extremal polynomial that minimizes
\[
\int \Bigl|z^n+\sum_{k=0}^{n-1} a_k z^k\Bigr|^2\,\mathrm{d}\mu(z)
\]
among all choices of coefficients \(a_0,\dots,a_{n-1}\).
Equivalently, it is the monic polynomial of degree \(n\) with minimal \(L^2(\mathrm{d}\mu)\)-norm. The Chebyshev polynomials of the first kind defined in~\eqref{eq:chebyshev_first_kind_trigonometry_formula} are orthogonal with respect to \(d\mu(x) = \frac{1}{\sqrt{1-x^2}}dx\), see e.g. \cite{bernstein30}. An extension of this orthogonality relation exists for Chebyshev polynomials on unions of intervals as is shown in~\cite[\S 3]{sodin-yuditskii93}.
% Since our investigations into orthogonal polynomials will be very brief, we refer the reader to \cite{van-assche97} for a proper introduction to this subject (where the title of this text draws inspiration from).
Since our discussion of orthogonal polynomials will be brief, we refer the reader to \cite{van-assche97} for a comprehensive introduction to the subject. The title of the present text is inspired by that work.

We proceed by showing that Chebyshev polynomials in the complex plane always admit a natural orthogonality interpretation.  Let \(\cM(\sfE)\) denote the collection of all probability measures supported on \(\sfE\).
The following result relates Chebyshev polynomials to the theory of orthogonal polynomials.

\begin{theorem}
	Let \(\sfE\subset \C\) be a compact set and \(w:\sfE\rightarrow [0,\infty)\) an upper semicontinuous weight function supported on at least \(n+1\) distinct points. Then
	\begin{equation}
	\max_{\mu\in \cM(\sfE)}\inf_{a_k}\int\Bigl|z^n+\sum_{k=0}^{n-1}a_kz^{k}\Bigr|^2w(z)^2d\mu(z) = t_n(\sfE,w)^2.
	\label{eq:OPM}	
	\end{equation}
	Furthermore, \(T_n^{\sfE,w}\) is the monic orthogonal polynomial w.r.t \(w(z)^2d\mu^\ast\)  for any extremal measure \(\mu^\ast\). Such a measure can be chosen to be supported on \(n+1\) points if \(\sfE\subset \R\) and at most \(2n+1\) points otherwise.
	\label{thm:OPM}
\end{theorem}
Recently, the study of \eqref{eq:OPM} has garnered attention, see e.g. \cite{bos-levenberg-ortega-cerda21, buchecker-eichinger-zinchenko25, charpentier-levenberg-wielonsky25}. In these articles, extremal measures corresponding to \eqref{eq:OPM} are called \textit{optimal prediction measures}. We will demonstrate that this concept is actually equivalent to the theory of \textit{extremal signatures} from~\cite{rivlin-shapiro61}.

First of all, note that if \(\mu\in \cM(\sfE)\), then
\[\inf_{a_k}\int\Bigl|z^n+\sum_{k=0}^{n-1}a_kz^{k}\Bigr|^2w(z)^2d\mu(z)\leq \int \bigl|T_n^{\sfE,w}(z)\bigr|^2w(z)^2d\mu(z)\leq t_n(\sfE,w)^2.\]
Consequently, in order to prove Theorem \ref{thm:OPM} we only need to demonstrate that there exists a measure for which equality holds in \eqref{eq:OPM}. This will be a consequence of an adaptation of \cite[Theorem 1]{rivlin-shapiro61}, see also \cite[Theorem 2.5]{rivlin90} and \cite[Theorem 8]{novello-schiefermayr-zinchenko21}. We will use that an upper semicontinuous function \(f:\sfE\rightarrow \C\) has a compact set of extremal points
\[\operatorname{Ext}(f,\sfE) \coloneqq \{z\in \sfE: |f(z)| = \|f\|_\sfE\}.\] 

\begin{theorem}[Rivlin \& Shapiro \cite{rivlin-shapiro61}]
	Let \(\sfE\subset \C\) be compact and \(w:\sfE\rightarrow [0,\infty)\) be upper semicontinuous supported on at least distinct \(n+1\) points. A monic polynomial \(P\) of degree~\(n\) is equal to \(T_n^{\sfE,w}\) if and only if there are \(m\geq n+1\) distinct points \(z_1,\dotsc,z_m\in \operatorname{Ext}(wP,\sfE)\) and positive numbers \(\lambda_1\dotsc,\lambda_m\) such that
 	\begin{equation}
 		\sum_{j=1}^{m}\lambda_j\overline{P(z_j)}z_j^k = 0,\quad k=0,\,1, \dotsc,n-1.\label{eq:rivlin_shapiro_orthogonality}
 	\end{equation}
 	If \(\sfE\subset \R\) then we can choose \(m= n+1\). Otherwise, we may impose \(m\leq 2n+1\).
 	\label{thm:rivlin_shapiro}
\end{theorem}
Before proving Theorem \ref{thm:rivlin_shapiro} let us illustrate how Theorem~\ref{thm:OPM} follows from Theorem~\ref{thm:rivlin_shapiro}. For this reason, assume that \(z_1,\dotsc,z_m\) and \(\lambda_1,\dotsc,\lambda_m\) have the properties specified in Theorem~\ref{thm:rivlin_shapiro}. Note that \(w(z_j)\neq 0\) for \(j=1,\dotsc,m\). Define
\[\alpha_j = \frac{\lambda_j}{w(z_j)^2}\Bigg/\sum_{l=1}^{m}\frac{\lambda_l}{w(z_l)^2},\]
and
\[\mu = \sum_{j=1}^{m}\alpha_j\delta_{z_j}.\]
Clearly \(\mu\) is a probability measure. Furthermore, by \eqref{eq:rivlin_shapiro_orthogonality} we gather that if \(0\leq k \leq n-1\), then
\[\int_{\sfE}\overline{T_n^{\sfE,w}(z)}z^kw(z)^2d\mu(z) = \sum_{j=1}^{m}\lambda_j\overline{T_n^{\sfE,w}(z_j)}z_j^k\Bigg/\sum_{l=1}^{m}\frac{\lambda_l}{w(z_l)^2} = 0.\]
Therefore, \(T_n^{\sfE,w}\) is a monic orthogonal polynomial with respect to \(w(z)^2d\mu\) and as a consequence
\[\int |T_n^{\sfE,w}(z)|^2w(z)^2d\mu(z) = \inf_{a_k}\int\Bigl|z^n+\sum_{k=0}^{n-1}a_kz^{k}\Bigr|^2w(z)^2d\mu(z).\] 
Finally, since \(z_j\in \operatorname{Ext}(wT_n^{\sfE,w},\sfE)\) we gather that \(|T_n^{\sfE,w}(z_j)|w(z_j) = t_n(\sfE,w)\) which implies that
\[\int |T_n^{\sfE,w}(z)|^2w(z)^2d\mu(z) = \sum_{j=1}^{m}\alpha_j|T_n^{\sfE,w}(z_j)|^2w(z_j)^2 = \sum_{j=1}^{m}\alpha_j t_n(\sfE,w)^2 = t_n(\sfE,w)^2.	\]

We will also require the following lemma in the proof of Theorem \ref{thm:rivlin_shapiro}.

\begin{lemma}
	Let $\sfE\subset \R$ be compact and $w:\sfE\rightarrow [0,\infty)$ be a bounded function which is non-zero for at least $n+1$ points. The coefficients of $T_n^{\sfE,w}$ are real.
	\label{lem:cheb_on_real}
\end{lemma}
\begin{proof}
	The inequality
	\[|w(x)\Re(T(x))| = |\Re(w(x)T(x))|\leq |w(x)T(x)|\]
	holds for any polynomial $T$ and $x\in \sfE$. Since $\Re(T(x))$ is a polynomial with real coefficients, the result follows from the uniqueness of the minimizer guaranteed by Theorem~\ref{thm:chebyshev_existence_uniqueness}.
\end{proof}

\begin{proof}[Proof of Theorem \ref{thm:rivlin_shapiro}]
	We follow the approach in \cite{rivlin90}. Let 
	\[K \coloneqq \left\{\left(\overline{T_n^{\sfE,w}(z)},\overline{T_n^{\sfE,w}(z)}z,\dotsc, \overline{T_n^{\sfE,w}(z)}z^{n-1}):z\in \operatorname{Ext}(wT_n^{\sfE,w},\sfE\right)\right\}.\]
	To begin let us show that \(0\in \operatorname{cvh}(K)\) - where \(\operatorname{cvh}\) denotes the convex hull. If that is not the case, then the separating hyperplane theorem, see e.g. \cite{boyd-vandenberghe04}, implies that we can find~\(c_0,\dotsc,c_{n}\in \C\) such that
	\[\Re\left(c_{0}+c_1\zeta_1+\cdots +c_{n}\zeta_n\right)\geq 0,\quad (\zeta_1,\dotsc,\zeta_{n})\in K,\]
	while
	\[-\tau \coloneqq \Re\left(c_{0}+c_1\cdot 0+\cdots +c_{n}\cdot 0\right) = \Re c_0<0.\]
	Let us define
	\[P(z) = \sum_{k=0}^{n-1}c_{k+1}z^k\]
	so that \(P\) is a polynomial of degree at most \(n-1\). It is clear that if \(z\in \operatorname{Ext}(wT_n^{\sfE,w},\sfE)\), then
	\[\Re\left(\overline{T_n^{\sfE,w}(z)}P(z)\right) = \Re\left(\sum_{k=0}^{n-1}c_{k+1}\overline{T_n^{\sfE,w}(z)}z^k\right)\geq \tau.\]
	Note that \(T_n^{\sfE,w}-\epsilon P\) is a monic polynomial of degree \(n\) for every \(\epsilon>0\). We will now show that by choosing \(\epsilon>0\) sufficiently small we can ensure that \(\|w(T_n^{\sfE,w}-\epsilon P)\|_{\sfE}<t_n(\sfE,w)\) thus deriving a contradiction since \(t_n(\sfE,w)\) is the minimal quantity.
	
	For \(z\in \operatorname{Ext}(wT_n^{\sfE,w},\sfE)\) we have that \(w(z) = t_n(\sfE,w)/|T_n^{\sfE,w}(z)|\) which shows that the restriction of \(w\) to the compact set \(\operatorname{Ext}(wT_n^{\sfE,w},\sfE)\) is continuous, see also \cite[Lemma 2]{novello-schiefermayr-zinchenko21}. Consequently, since \(\operatorname{Ext}(wT_n^{\sfE,w},\sfE)\) is compact, there exists some lower bound\[w(z)\geq1/M>0\] for \(z\in \operatorname{Ext}(wT_n^{\sfE,w},\sfE)\) and therefore
	\begin{align*}
		w(z)^2\left(\left|T_n^{\sfE,w}(z)-\epsilon P(z)\right|\right)^2 & = w(z)^2\left(|T_n^{\sfE,w}(z)|^2-2\epsilon \Re\left(\overline{T_n^{\sfE,w}(z)}P(z)\right)+\epsilon^2|P(z)|^2\right)\\
		&\leq t_n(\sfE,w)^2-2\epsilon \tau /M+\epsilon^2 \|wP\|_{\sfE}^2.
	\end{align*} 
	By choosing \(\epsilon>0\) sufficiently small we can achieve the inequality
	\[\|w(T_n^{\sfE,w}-\epsilon P)\|_{\operatorname{Ext}(wT_n^{\sfE,w},\sfE)}<t_n(\sfE,w),\]
	and by upper semi-continuity this may be extended to an open neighborhood \(U\) of~\(\operatorname{Ext}(wT_n^{\sfE,w},\sfE)\) with respect to \(\sfE\).
	Note that \(B = \sfE\setminus U\) is compact and since upper semicontinuous functions attain maximum values on compacts we gather that there exists some \(0<\rho<1\) such that
	\[|w(z)T_n^{\sfE,w}(z)|\leq (1-\rho)t_n(\sfE,w),\quad z\in B.\]
	As a consequence, if \(z\in B\), then
	\begin{align*}
		w(z)\left|T_n^{\sfE,w}(z)-\epsilon P(z)\right| & \leq (1-\rho)t_n(\sfE,w)+\epsilon \|wP\|_{\sfE}.
	\end{align*} 
	Here, it is also clear that by choosing \(\epsilon>0\) sufficiently small, we can achieve
	\[\|w(T_n^{\sfE,w}-\epsilon P)\|_{B}<t_n(\sfE,w).\]
	This shows that under our current assumption, we can choose \(\epsilon>0\) small enough so that
	\[\|w(T_n^{\sfE,w}-\epsilon P)\|_{\sfE}<t_n(\sfE,w)\]
	but this is a contradiction. As a consequence, \(0\in \operatorname{cvh}(K)\) must hold.
	
	Since \(0\in \operatorname{cvh}(K)\) and since \(K\) can be identified with a subset of \(\R^{2n}\) we find by Carath\'{e}odory's theorem, see e.g. \cite[Theorem 1.1.4]{schneider13}, that there are \(m\leq 2n+1\) positive numbers \(\lambda_1,\dotsc,\lambda_m\) summing to \(1\) and associated points \(z_1,\dotsc,z_m\in \operatorname{Ext}(wT_n^{\sfE,w},\sfE)\) such that
	\begin{equation}
		\sum_{k=1}^{m}\lambda_k\overline{T_n^{\sfE,w}(z)}z_k^j = 0,\quad j =0,1,\dotsc,m.
		\label{eq:caratheodory_sum_chebyshev}
	\end{equation}
	We need to establish that \(m \geq n+1\). If it was the case that \(m\leq n\), then
	\[Q(z) = \sum_{j=1}^{m}T_n^{\sfE,w}(z_j)\prod_{k=1\\k\neq j}^{m}\frac{z-z_k}{z_j-z_k}\] 
	is of degree at most \(n-1\) and by \eqref{eq:caratheodory_sum_chebyshev}
	\[0 = \sum_{k=1}^{m}\lambda_k\overline{T_n^{\sfE,w}(z_k)}Q(z_k) = \sum_{k=1}^{m}\lambda_k |T_n^{\sfE,w}(z)|^2>0,\]
	which clearly is impossible. We may therefore conclude that \(n+1\leq m \leq 2n+1\). 
	
	If \(\sfE\subset \R\) then \(K\) is a subset of \(\R^n\) by Lemma \ref{lem:cheb_on_real}. In this case, the application of Carath\'{e}odory's theorem means that we may choose \(m =  n+1\).
	
	Conversely, assume that \(P\) is a monic polynomial of degree \(n\) such that \eqref{eq:rivlin_shapiro_orthogonality} holds. By defining
	\[\alpha_j = \frac{\lambda_j}{w(z_j)^2}\Bigg/\sum_{l=1}^{m}\frac{\lambda_l}{w(z_l)^2},\]
	and
	\[\mu = \sum_{j=1}^{m}\alpha_j\delta_{z_j}.\]
	we find that
	\[t_n(\sfE,w)^2\leq \|wP\|_{\sfE}^2\leq \int |P(z)|^2w(z)^2d\mu(z)\leq \int |T_n^{\sfE,w}(z)|^2w(z)^2d\mu(z)\leq t_n(\sfE,w)^2.\]
	But this shows that \[\|wP\|_{\sfE} = t_n(\sfE,w)\] and consequently \(P = T_n^{\sfE,w}\) by Theorem \ref{thm:chebyshev_existence_uniqueness}.
\end{proof}

This interpretation of Chebyshev polynomials as \(L^2\)-extremals has many nice qualities and we will use it to prove the classical alternation theorem. As we already saw, without proof, the Chebyshev polynomials on an interval have the representation
\[T_n^{[-1,1]}(x) = 2^{1-n}\cos(n\theta)\]
where $x = \cos\theta$. As such if
\[x_{n-j} = \cos\left(\frac{\pi j}{n}\right)\]
for $j = 0,\ldots,n$ then $-1 = x_0<x_1<\cdots <x_n = 1$ and
\[T_{n}^{[-1,1]}(x_j) = 2^{1-n}(-1)^{n-j}.\]
The value of the Chebyshev polynomial $T_n^{[-1,1]}$ alternates between $\pm2^{1-n}$ at $n+1$ consecutive points. This property is in fact characterizing for best approximations in the real setting, however, this realization took a long time to develop. Although Chebyshev described this phenomenon the first person to fully study this characterization was Kirchberger \cite{kirchberger02}, in 1902. The first complete proof was published by Borel in \cite{borel05}. Achieser states in \cite[p. 7]{achieser87} that Markov gave a proof in a series of lectures around 1905 that first appeared in print in 1948 \cite{markov48}. We will consider the general form as expressed by Novello et al. in \cite{novello-schiefermayr-zinchenko21}, but our proof will be based on Theorem \ref{thm:rivlin_shapiro}.
%\begin{theorem}[Borel (1905) \cite{borel05}, Markov \cite{markov48}]
%	Let $\sfE\subset \R$ be compact and $w:\sfE\rightarrow [0,\infty)$ be an upper semicontinuous function which is non-zero for at least $n+1$ points of $\sfE$. A monic polynomial $T$ of degree $n$ coincides with the Chebyshev polynomial $T_n^{\sfE,w}$ if and only if there are $n+1$ points in $\sfE$ denoted by $x_0<x_1<\cdots <x_{n}$ such that
%	\begin{equation}
%		w(x_j)T(x_j) = (-1)^{n-j}\|wT\|_{\sfE}.
%		\label{eq:alternation_condition}
%	\end{equation}
%	\label{thm:alternation_theorem}
%\end{theorem}

\begin{theorem}[Borel (1905) \cite{borel05}, Markov \cite{markov48}]
	Let \(\sfE\subset\R\) be compact, and let \(w:\sfE\to[0,\infty)\) be an upper semicontinuous function
	that is positive on \(n+1\) distinct points of \(\sfE\).
	A monic polynomial \(T\) of degree \(n\) coincides with the Chebyshev polynomial \(T_n^{\sfE,w}\)
	if and only if there exist \(n+1\) points
	\begin{equation}
		x_1<x_2<\cdots<x_{n+1}\qquad \text{in }\sfE
		\label{eq:alternation_set}	
	\end{equation}
	such that
	\begin{equation}
		w(x_j)\,T(x_j)=(-1)^{\,n+1-j}\,\|wT\|_{\sfE},
		\qquad j=1,\dots,n+1.
		\label{eq:alternation_condition}
	\end{equation}
	\label{thm:alternation_theorem}
\end{theorem}

\begin{proof}
%	Let us begin by showing that there are points such as \eqref{eq:alternation_set} for which \(T_n^{\sfE,w}\) satisfies \eqref{eq:alternation_condition}. 
	We begin by showing that there exist points \(x_1<\cdots<x_{n+1}\) in \(\sfE\) such that~\(T_n^{\sfE,w}\) satisfies \eqref{eq:alternation_condition}.
	By referring to Theorem \ref{thm:rivlin_shapiro} we know that there are positive numbers~\(\lambda_1,\dotsc,\lambda_{n+1}\) and points \(x_1,\dotsc,x_{n+1}\) in \(\operatorname{Ext}(wT_n^{\sfE,w},\sfE)\) such that
	\[\sum_{j=1}^{n+1}\lambda_j T_n^{\sfE,w}(x_j)P(x_j) = 0\] 
	for any polynomial \(P\) of degree \(n-1\).
	Assume that there is some index \(k\) such that
	\[\sgn T_n^{\sfE,w}(x_k) = \sgn T_n^{\sfE,w}(x_{k+1}).\]
	By defining,
	\[Q(x) = \prod_{j<k}(x-x_j)\prod_{j>k+1}(x-x_j),\]
	we find that \(Q\) is a polynomial of degree at most \(n-1\) and \(\sgn Q(x_k) = \sgn Q(x_{k+1})\). From~\eqref{eq:rivlin_shapiro_orthogonality} we gather that
	\[0 = \sum_{j=1}^{n+1}\lambda_jT_n^{\sfE,w}(x_j)Q(x_j) = \lambda_k T_n^{\sfE,w}(x_k)Q(x_k)+\lambda_{k+1} T_n^{\sfE,w}(x_{k+1})Q(x_{k+1}),\]
	however, this last term can not be \(0\) since both summands have the same sign and neither is \(0\).
	
	Conversely, if \(T\) is a monic polynomial satisfying \eqref{eq:alternation_condition}, then if \(\|wT\|_{\sfE}>t_n(\sfE,w)\) we find that
	\[\sgn w(x_j)(T(x_j)-T_n^{\sfE,w}(x_j)) = (-1)^{n+1-j},\quad j=1,\dotsc,n+1.\]
	But this is a contradiction since \(T-T_n^{\sfE,w}\) is a polynomial of degree at most \(n-1\) and it has~\(n+1\) sign changes, and consequently, at least \(n\) zeros. We conclude that \(\|wT\|_{\sfE} = t_n(\sfE,w)\) and consequently \(T = T_n^{\sfE,w}\).
\end{proof}

It readily follows from Theorem \ref{thm:alternation_theorem} that if $x = \cos\theta$ then
\[T_n^{[-1,1]}(x) = 2^{1-n}\cos n\theta.\]
Indeed, using Euler's formula as well as the notation $\lfloor x \rfloor$ to mean the largest integer less than $x\in \R$ we obtain
\begin{align*}
	\cos n \theta &= \Re \Big[\cos \theta+i\sin \theta\Big]^n = \Re \sum_{j=0}^{n}{{n}\choose{j}}i^{j}\cos(\theta)^{n-j}\sin(\theta)^j \\
	&=\sum_{j=0}^{\lfloor n/2\rfloor}{{n}\choose{2j}}(-1)^j\cos(\theta)^{n-2j}\sin(\theta)^{2j}=\sum_{j=0}^{\lfloor n/2\rfloor}{{n}\choose{2j}}(-1)^j\cos(\theta)^{n-2j}(1-\cos^2\theta)^j.
\end{align*}
We see that $\cos n\theta = P(\cos\theta)$ where $P$ is an algebraic polynomial of degree $n$. The leading coefficient of $P$ is given by
\[\sum_{j=0}^{\lfloor n/2\rfloor}{{n}\choose{2j}} = 2^{n-1}\]
and consequently $2^{1-n}P(x)$ is a monic polynomial of degree \(n\) with an alternating set consisting of $n+1$ points. Consequently, Theorem \ref{thm:alternation_theorem} implies that $T_n^{[-1,1]}(x) = 2^{1-n}P(x).$

In order to provide further related and explicit examples we introduce the family of Jacobi weights for $\alpha\geq 0$ and $\beta \geq 0$
\begin{equation}
	w^{(\alpha,\beta)}(x) = (1-x)^\alpha(1+x)^\beta, \quad x\in [-1,1].
	\label{eq:jacobi_weight}
\end{equation}
Using the transformation $x = \cos\theta$ together with Theorem \ref{thm:alternation_theorem} it is possible to conclude, in a similar way as above, that
\begin{align*}
	T_{n}^{w^{(1/2,1/2)}}(x) & = 2^{-n}\frac{\sin (n+1)\theta}{\sin \theta}, \\
	T_{n}^{w^{(0,1/2)}}(x) & = 2^{-n}\frac{\cos (n+\tfrac{1}{2})\theta}{\cos \tfrac{1}{2}\theta}, \\
	T_n^{w^{(1/2,0)}}(x) & = 2^{-n}\frac{\sin(n+\tfrac{1}{2})\theta}{\sin \tfrac{1}{2}\theta},
\end{align*}
for $x\in [-1,1]$, see e.g. \cite{mason93}. These are the normalized Chebyshev polynomials of the second, third, and fourth kind. In \cite{christiansen-rubin25}, the Chebyshev polynomials corresponding to Jacobi weights are studied, in particular monotonicity properties of $n\mapsto 2^nt_n([-1,1],w^{(\alpha,\beta)})$. This complements work in \cite{bernstein31}.

These examples illustrate that Theorem \ref{thm:alternation_theorem}, also called the ``Alternation theorem'', can be used to determine Chebyshev polynomials of real sets explicitly, something we will see further examples of shortly. Its use can further provide insight into the asymptotic behavior of Chebyshev polynomials corresponding to real sets, see e.g. \cite{bernstein30,bernstein31,christiansen-simon-zinchenko-I}. For a general account of alternation properties on the real line we refer the reader to \cite{sodin-yuditskii93}. The fact that the Alternation theorem does not extend in the complex setting is one reason for the fact that Chebyshev polynomials corresponding to complex sets are less understood than their real counterparts. 

\subsection{Weighted Chebyshev polynomials on an interval}
\label{subs:markov_bernstein}
We begin this section by considering Markov's generalization of Problem \ref{prob:weighted_chebyshev_approximation}, in the case where $f(x) = x^n$. We will detail the solution he provided, as discussed in \cite{markov84}. Our goal is to ultimately derive a general result on weighted Chebyshev polynomials, as established by Bernstein in \cite{bernstein31}. To achieve this, we will provide a detailed account of the result by Markov. This approach allows us to bypass the analysis of the asymptotic formulas for orthogonal polynomials discussed by Bernstein in \cite{bernstein30}. Consequently, our proof is shorter than Bernstein's, though it does not address the asymptotically alternating properties of orthogonal polynomials associated with weights on $[-1,1]$. 
 
    Let $a_k\in\overline{\C}\setminus [-1,1]$ and form the polynomial
    \[P = \prod_{k=1}^{2m}\left(1-\frac{x}{a_k}\right)\]
    which we impose to satisfy \(P(x)>0\) for \(x\in [-1,1]\). The associated weight function~$w:[-1,1]\rightarrow (0,\infty)$ is defined by
	\begin{equation}
          w(x) = 1\Big /\sqrt{P(x)}. 
          \label{eq:weight_markov}
	\end{equation}
    The case of an odd number of factors can be handled by taking $a_{2m} = \infty$ and $\vert a_k \vert<\infty$ for~$k=1, 2, \ldots, 2m-1$. 
    Let $z\in\bbT$ and $\rho_k\in\D$ be defined implicitly through the equations
    \begin{equation}
    	       x = \frac{1}{2}\left(z+\frac{1}{z}\right) \; \mbox{ and } \; 
       a_k  = \frac{1}{2}\left(\rho_k+\frac{1}{\rho_k}\right) \; \mbox{ for } \; k=1, 2, \ldots, 2m.
       \label{eq:parameters_markov}
    \end{equation}
The following result is due to Markov in \cite{markov84}. However, the explicit representation for~$wT_n^{w}$ that we use can be found in \cite[Appendix A]{achieser56} where the proof is left to the reader. We include a proof here.

%\begin{figure}
%\centering
%\includegraphics[width=0.3\textwidth]{Figures/Andrei_Markov.jpg}
%\caption{Andrei A. Markov (1856-1922).}
%\label{fig:markov}
%\end{figure}
    \begin{theorem}[Markov (1884) \cite{markov84}]
        \label{thm:markov_weighted_chebyshev}
        Let $w:[-1,1]\rightarrow (0,\infty)$ and $\{a_k:k = 1,\dotsc,2m\}$ be as in \eqref{eq:weight_markov} and let $\{\rho_k:k=1,\dotsc,2m\}$ be defined through \eqref{eq:parameters_markov}. For positive integers~$n>m$, we have for $x = \frac{1}{2}(z+\frac{1}{z})$ the identity 
        \begin{equation}
        \label{eq:markov_explicit_chebyshev}
           w(x)T_n^{w}(x) = 2^{-n}\prod_{k=1}^{2m}\sqrt{1+\rho_k^2}
           \left(z^{m-n} \prod_{k=1}^{2m}\sqrt{\frac{1-z\rho_k}{z-\rho_k}}+
           z^{n-m}\prod_{k=1}^{2m}\sqrt{\frac{z-\rho_k}{1-z\rho_k}}\,\right)
        \end{equation}
        and
        \begin{equation}
            t_n([-1,1],w) = 2^{1-n}\exp\left(\frac{1}{\pi}\int_{-1}^{1}\frac{\log w(x)}{\sqrt{1-x^2}}dx\right).
            \label{eq:achieser_norm}
        \end{equation}

    \end{theorem}
    We remark that computation of integrals of the form
    \[\frac{1}{\pi}\int_{-1}^{1}\frac{\log w(x)}{\sqrt{1-x^2}}dx\]
    can be handled using the machinery of potential theory. In particular, we will heavily rely on the following lemma for computations. We postpone the proof until Section \ref{subseq:pot_theory}.
    
    \begin{lemma}
			For any $z\in \C$, we have
			\begin{equation}
			\label{int}
				\frac{1}{\pi}\int_{-1}^{1}\frac{\log|x-z|}{\sqrt{1-x^2}}dx = \log\frac{\left|z+\sqrt{z^2-1}\right|}{2},
			\end{equation}
			where $z+\sqrt{z^2-1}$ maps $\C\setminus[-1,1]$ conformally onto $\C\setminus\overline{\D}$. For $z\in [-1,1]$ the integral is constantly equal to $-\log 2$.
			\label{lem:log_integral_equilibrium_measure}
		\end{lemma}
\begin{proof}[Proof of Theorem \ref{thm:markov_weighted_chebyshev}]
        The proof rests upon showing that the function exhibited in \eqref{eq:markov_explicit_chebyshev} defines the product of a monic polynomial in $x$ of degree $n$ with the weight function $w$. We further show that this function possesses the specified alternating qualities made precise in Theorem~\ref{thm:alternation_theorem}. To begin, we consider a branch of the function
        \[\Psi(z) = \prod_{k=1}^{2m}\sqrt{z-\rho_k}.\]
        The branch can be specified if we let \(\Psi\) be defined on a neighborhood of \(|z|\geq 1\) with~$\Psi(z) = z^m(1+o(1))$ as $z\rightarrow \infty$. We introduce the function
        \begin{align}
            \begin{split}
            \Phi_n(z) & = \frac{1}{2}\left(z^{2m-n}\frac{\Psi(1/z)}{\Psi(z)}+z^{n-2m}\frac{\Psi(z)}{\Psi(1/z)}\right)\Big /w\left(\frac{z+z^{-1}}{2}\right)\\& = \frac{1}{2}\left(z^{m-n}\prod_{k=1}^{2m}\sqrt{\frac{1-z\rho_k}{z-\rho_k}}+z^{n-m}\prod_{k=1}^{2m}\sqrt{\frac{z-\rho_k}{1-z\rho_k}}\right)\Big /w\left(\frac{z+z^{-1}}{2}\right)
            \end{split}
            \label{eq:extremal_func}
        \end{align}    
        and claim that $\Phi_n$ is a polynomial in the variable $x$ related to \(z\) through \eqref{eq:parameters_markov}. To see this note that
        \begin{align*}
            w\left(\frac{z+z^{-1}}{2}\right) = z^{m}\prod_{k=1}^{2m}\left(\frac{(z-\rho_k)(1-\rho_kz)}{1+\rho_k^2}\right)^{-1/2}.
        \end{align*}
        Substituting this expression into \eqref{eq:extremal_func} yields
        \[\Phi_n(z) = \frac{1}{2}\frac{1}{\prod_{k=1}^{2m}\sqrt{1+\rho_k^2}}\left(z^{-n}\prod_{k=1}^{2m}(1-z\rho_k)+z^{n-2m}\prod_{k=1}^{2m}(z-\rho_k)\right).\]
        Represented this way, it is clear that $\Phi_n$ is a rational function in $z$ which is analytic away from $0$ and $\infty$.
        
        From the definition in \eqref{eq:extremal_func} we see that $\Phi_n(z) = \Phi_n(1/z)$ implying that $\Phi_n(x+\sqrt{x^2-1})$ has well-defined real-valued limit values as the complex variable $x$ approaches $[-1,1]$ from either side with respect to the complex plane. Schwarz reflection principle \cite[Theorem IX.1.1]{conway78} implies that $\Phi_n(x+\sqrt{x^2-1})$ is entire.
        
        By letting $z\rightarrow \infty$ it is clear that $\Phi_n(z) / z^n$ has the finite limit 
        \[2^{-1}\Big/\prod_{k=1}^{2m}\sqrt{1+\rho_k^2}\]
        Moving our considerations back to the variable $x$, we find that $\Phi_n(x+\sqrt{x^2-1})$ must be a polynomial of degree $n$ in $x$ with leading coefficient 
        \[2^{n-1}\Big/\prod_{k=1}^{2m}\sqrt{1+\rho_k^2}.\]
        The polynomial 
        \begin{equation}
            2^{-n}\prod_{k=1}^{2m}\sqrt{1+\rho_k^2}\left(z^{m-n}\prod_{k=1}^{2m}\sqrt{\frac{1-z\rho_k}{z-\rho_k}}+z^{n-m}\prod_{k=1}^{2m}\sqrt{\frac{z-\rho_k}{1-z\rho_k}}\right)\Big/ w(x)
            \label{eq:alteranting_poly}
        \end{equation}
        is necessarily monic in the variable $x$, and as we shall see, actually equal to $T_n^{w}$.
        Indeed, the only remaining task is to verify the alternating behavior of \eqref{eq:alteranting_poly} when multiplied with~$w(x)$.
        Note that for $|z|=1$
        \[\left|z^{n-m}\prod_{k=1}^{2m}\sqrt{\frac{z-\rho_k}{1-z\rho_k}}\right|\leq 1\]
        and hence the function defined in \eqref{eq:alteranting_poly} is upper bounded by
        \begin{equation}
            2^{1-n}\prod_{k=1}^{2m}\sqrt{1+\rho_k^2}
            \label{eq:max_val}
        \end{equation}
        whenever $x\in [-1,1]$. Since $w$ is a real function, any $a_k$ that has non-zero imaginary part must appear together with its complex conjugate. This ensures that \eqref{eq:max_val} is positive. Let~$z$ traverse the upper part of the unit circle from $1$ to $-1$. This corresponds to $x$ going from $1$ to $-1$. The maximal value from \eqref{eq:max_val} is attained precisely when
        \[\arg \left(z^{n-m}\prod_{k=1}^{2m}\sqrt{\frac{z-\rho_k}{1-z\rho_k}}\right) = 0\mod \pi.\]
        Let 
        \[f(z) = \left( z^{n-m}\prod_{k=1}^{2m}\sqrt{\frac{z-\rho_k}{1-z\rho_k}}\right)^2 = z^{2n-2m}\prod_{k=1}^{2m}\frac{z-\rho_k}{1-z\rho_k}.\]
        This function is holomorphic in the unit disk and has $2n$ zeros inside. As $z$ traverses the upper half-circle the image $f(z)$ will wrap around the origin $2n$ times, see the discussion following \cite[Theorem V.3.4]{conway78}. But this implies that as $z$ goes from $1$ to $-1$ along the upper half-circle, the value of
        \[\arg \left(z^{n-m}\prod_{k=1}^{2m}\sqrt{\frac{z-\rho_k}{1-z\rho_k}}\right)\]
        goes from $0$ to $n\pi$. Consequently, the function in \eqref{eq:alteranting_poly} has $n+1$ alternating points where it attains the value \eqref{eq:max_val} on $[-1,1]$ when multiplied with $w$. Theorem \ref{thm:alternation_theorem} implies that $T_n^{w}$ and \eqref{eq:alteranting_poly} coincide. Using Lemma \ref{lem:log_integral_equilibrium_measure} we conclude that
        \[\prod_{k=1}^{2m}\sqrt{1+\rho_k^2} = \exp\left(\frac{1}{\pi}\int_{-1}^{1}\frac{\log w(x)}{\sqrt{1-x^2}}dx\right)\]
        and the proof is complete.
    \end{proof}

The solution of Problem~\ref{prob:weighted_chebyshev_approximation} due to Chebyshev~\cite{chebyshev59}
in the special case \(f(x)=x^n\) can be recovered from Theorem~\ref{thm:markov_weighted_chebyshev}
by imposing the relations \(a_{2k}=a_{2k+1}\) for each \(k=1,\dotsc,m\).
The next major advance beyond Markov's results on weighted Chebyshev polynomials on \([-1,1]\)
is due to Bernstein, who between 1930 and 1931 published a two-part series of papers in which he
established precise asymptotic formulas for both orthogonal and Chebyshev polynomials associated
with general weight functions on \([-1,1]\).

This work marked a decisive shift in emphasis. Whereas Chebyshev and Markov had focused primarily
on obtaining explicit formulas for weighted Chebyshev polynomials, Bernstein instead investigated
their asymptotic behavior as the degree tends to infinity.
It should be noted, however, that already in 1919 Faber~\cite{faber20} had extended the notion of
Chebyshev polynomials to the complex plane and carried out a detailed asymptotic analysis for
analytic Jordan domains. We shall not pursue these developments here, as they will be discussed
in depth in Section~\ref{sec:chebyshev_complex}.

Bernstein considered general weights \(w\) on \([-1,1]\) that are bounded away from zero and
Riemann integrable, and showed that the two quantities in \eqref{eq:achieser_norm} remain
asymptotically equivalent as \(n\to\infty\).
Remarkably, his results continue to hold even for weights that vanish on \([-1,1]\).
In general, however, the analysis of Chebyshev polynomials associated with vanishing weights is
considerably more delicate.
Indeed, the classical results of Chebyshev and Markov apply to weights that are reciprocals of
positive polynomials, a class which can only approximate vanishing weights in a very limited sense.
    
%\begin{figure}
%\centering
%\includegraphics[width=0.3\textwidth]{Figures/bernstein.jpg}
%\caption{Sergei N. Bernstein (1880-1968).}
%\label{fig:berstein}
%\end{figure}
\begin{theorem}[Bernstein (1931) \cite{bernstein31}]
			Suppose $\alpha_k\in [0,\infty)$ and $b_k\in [-1,1]$ for $k=0, 1, \ldots, m$.  Consider a weight function $w:[-1,1]\rightarrow [0,\infty)$ of the form
			\begin{equation}  \label{eq:bernstein_weight}
				w(x) = w_0(x)\prod_{k=0}^{m}|x-b_k|^{\alpha_k},
			\end{equation}
			where $w_0$ is Riemann integrable and satisfies $1/M \leq w_0(x) \leq M$ for some constant $M\geq 1$ and all points $x\in [-1,1]$. Then %the weighted Chebyshev polynomial
%			\[
%				T_n^{w}(x) = x^{n}+\sum_{k=0}^{n-1}a_kx^k
%			\]
%			which minimises
%			\[
%				\sup_{x\in [-1,1]}w(x)|T_n^w(x)|
%			\]
%			among all monic polynomials of degree $n$ satsifies
			\begin{equation}  \label{eq:bernstein_asymptotics}
				%\sup_{x\in [-1,1]}w(x)|T_n^w(x)|= 
				t_n([-1,1],w)\sim
				2^{1-n}\exp\left(\frac{1}{\pi}\int_{-1}^{1}
				\frac{\log w(x)}{\sqrt{1-x^2}}dx\right),\quad n\to \infty.
			\end{equation}
			\label{thm:Bernstein}
		\end{theorem}
		
		\begin{remark}
			Note that Theorem \ref{thm:Bernstein} remains valid even when \(\alpha_k<0\), however, in this case \(w\) is not a bounded function, and the corresponding extremal polynomial will have to have zero that compensates for the effect of the pole.
			
			This result has recently been extended to weights with more general types of zeros in~\cite[Theorem 2.5]{alpan-zinchenko26}.
		\end{remark}
		Given two functions $f,g:\N\rightarrow \C$, such that $g(n)\neq 0$ for sufficiently large $n$, we use the notation $f\sim g$ as $n\rightarrow\infty$ to mean 
		\[\lim_{n\rightarrow \infty}\frac{f(n)}{g(n)}=1.\]
		
	\begin{proof}
				  	The proof will be carried out in two steps. Assume initially that $w$ is a continuous function on $[-1,1]$ satisfying
	\begin{equation}
	    \label{eq:weight_bounds}
	    \frac{1}{M} \leq w(x) \leq M       
	\end{equation}
for some $M\geq 1$.  In this more general setting, we no longer have equality in \eqref{eq:achieser_norm} however these quantities are still asymptotically equivalent as $n\to\infty$.

The idea is to approximate $w$ by weight functions that are reciprocals of polynomials 
and then applying Theorem \ref{thm:markov_weighted_chebyshev}. If $Q_1$ and $Q_2$ are two polynomials of degree $m$ such that
\[0<\frac{1}{Q_{1}(x)}<w(x)<\frac{1}{Q_{2}(x)},\quad x\in [-1,1]\]
then it is easily seen by the minimality of the weighted Chebyshev polynomial that
\begin{equation}
	\|T_n^{1/Q_2}/Q_2\|_{[-1,1]}\geq \|T_n^{1/Q_2}w\|_{[-1,1]}\geq \|T_n^{w}w\|_{[-1,1]}\geq \|T_n^{w}/Q_1\|_{[-1,1]}\geq \|T_n^{1/Q_1}/Q_1\|_{[-1,1]}.
	\label{eq:monotonicity_cheb_weight}
\end{equation}
Theorem \ref{thm:markov_weighted_chebyshev} therefore implies that if $n>m$ 
\[2^{1-n}\exp\left(-\frac{1}{\pi}\int_{-1}^{1}\frac{\log Q_1(x)}{\sqrt{1-x^2}}dx\right)\leq \|T_n^ww\|_{[-1,1]}\leq 2^{1-n}\exp\left(-\frac{1}{\pi}\int_{-1}^{1}\frac{\log Q_2(x)}{\sqrt{1-x^2}}dx\right).\]
From this we conclude that
\begin{align}
	\begin{split}
		\exp\left(-\frac{1}{\pi}\int_{-1}^{1}\frac{\log Q_1(x)}{\sqrt{1-x^2}}dx\right)&\leq \liminf_{n\rightarrow\infty}	2^{n-1}\|T_n^ww\|_{[-1,1]}\\&\leq \limsup_{n\rightarrow \infty}2^{n-1}\|T_n^ww\|_{[-1,1]}\leq \exp\left(-\frac{1}{\pi}\int_{-1}^{1}\frac{\log Q_2(x)}{\sqrt{1-x^2}}dx\right).
	\end{split}
	\label{eq:limsupinf_polynomial_weights}
\end{align}
On the other hand we also see that
\begin{equation}
	\exp\left(-\frac{1}{\pi}\int_{-1}^{1}\frac{\log Q_1(x)}{\sqrt{1-x^2}}dx\right)\leq \exp\left(\frac{1}{\pi}\int_{-1}^{1}\frac{\log w(x)}{\sqrt{1-x^2}}dx\right)\leq \exp\left(-\frac{1}{\pi}\int_{-1}^{1}\frac{\log Q_2(x)}{\sqrt{1-x^2}}dx\right)
	\label{eq:squeezed_integrals}
\end{equation}
holds for any such configuration. Given $\epsilon\in(0,1/M)$ we can apply Weierstra\ss' Approximation Theorem to conclude that there are polynomials $Q_1$ and $Q_2$ such that
\[0<w(x)-\epsilon<\frac{1}{Q_1(x)}<w(x)<\frac{1}{Q_2(x)}<w(x)+\epsilon.\]
As $\epsilon\rightarrow 0^+$ the dominated convergence theorem implies that the expressions in \eqref{eq:squeezed_integrals} converges to 
\[\exp\left(\frac{1}{\pi}\int_{-1}^{1}\frac{\log w(x)}{\sqrt{1-x^2}}dx\right).\]
We therefore conclude, using \eqref{eq:limsupinf_polynomial_weights}, that
\[\lim_{n\rightarrow \infty}2^{n-1}\|T_n^ww\|_{[-1,1]} = \exp\left(\frac{1}{\pi}\int_{-1}^{1}\frac{\log w(x)}{\sqrt{1-x^2}}dx\right)\]
which concludes the argument when $w$ is continuous. We refer the reader to \cite{christiansen-eichinger-rubin23} for details on how to lift this argument to Riemann integrable functions. The argument is essentially the same since such functions can be approximated from above and below in the $L^1$ norm by step functions which in turn can be approximated by continuous functions.
	
	The final step in proving Theorem \ref{thm:Bernstein} is to allow for zeros at certain points of $[-1,1]$.  We will adopt the approach outlined in \cite{bernstein31} and restrict our examination to the introduction of a single zero $b_1 = b$ within the interval $[-1, 1]$ using the weight $|x-b|^{\alpha}$.  That is, we consider weights of the form
	\begin{equation}  \label{w}
	   w(x)=w_0(x) \vert x-b \vert^\alpha,
	\end{equation}	
where $w_0$ is Riemann integrable and satisfies $1/M \leq w_0(x) \leq M$ for some constant $M\geq 1$.  More zeros can be added by repeated use of the argument we are about to explain.  

   Assume first that $\alpha=m$ is a positive integer. For $\epsilon>0$, consider the weight function
		\begin{equation}
			w^\epsilon(x) = w_0(x)|x-b-i\epsilon|^m.
		\end{equation}
		This weight is non-vanishing on $[-1,1]$ and fulfills all the previous requirements for \eqref{eq:bernstein_asymptotics} to be valid.  By use of Lemma \ref{lem:log_integral_equilibrium_measure}, we hence find that
	\begin{align}
	\begin{split}
		 \| w^\epsilon T_n^{w^\epsilon}\|_{[-1, 1]}
         %2^{1-n}\exp\left\{\frac{1}{\pi}\int_{-1}^{1}\frac{\log w^\delta(x)}{\sqrt{1-x^2}}dx\right\}(1+o(1)) \\
	& \sim 2^{1-n} \exp\left(\frac{1}{\pi}\int_{-1}^{1}\frac{\log w_0(x)}{\sqrt{1-x^2}}dx\right)
	 \exp\left(\frac{1}{\pi}\int_{-1}^{1}\frac{\log |x-b-i\epsilon|^{m}}{\sqrt{1-x^2}}dx\right)\\  
	& \sim \|w_0 T_n^{w_0}\|_{[-1,1]}
	\Biggl(\frac{\left|(b+i\epsilon)+\sqrt{(b+i\epsilon)^2-1}\right|}{2}\Biggr)^m
	\end{split}
            \label{eq:asymptotics_upper}
		\end{align}
as $n\to\infty$.  Note that the second factor in \eqref{eq:asymptotics_upper}
       % \[\left(\frac{\left|(b+i\delta)+\sqrt{(b+i\delta)^2-1}\right|}{2}\right)^m\]
        converges to $2^{-m}$ as $\epsilon\rightarrow 0$.  Since $w^\epsilon>w$ on~$[-1,1]$, we deduce as in \eqref{eq:monotonicity_cheb_weight} that 
		\begin{align}
            \begin{split}
            \|w^\epsilon T_n^{w^\epsilon}\|_{[-1,1]}& \geq \|wT_n^{w^\epsilon}\|_{[-1,1]}
            \geq \|wT_n^w\|_{[-1,1]}  \\ 
            &\geq \|w_0T_{n+m}^{w_0}\|_{[-1,1]} 
            \sim 2^{-m}\|w_0 T_{n}^{w_0}\|_{[-1,1]},    
            \end{split}
   \label{eq:upper_lower_bound}
		\end{align}
		as $n\rightarrow \infty$. The combination of \eqref{eq:asymptotics_upper} and \eqref{eq:upper_lower_bound} now implies that
	\begin{equation}
		\|wT_n^w\|_{[-1,1]} \sim 2^{-m}\|w_0 T_n^{w_0}\|_{[-1,1]}
		\sim 2^{1-n}\exp\left(\frac{1}{\pi}\int_{-1}^{1}\frac{\log w(x)}{\sqrt{1-x^2}}dx\right)
		\end{equation}
as $n\to\infty$, and this proves that \eqref{eq:bernstein_asymptotics} is valid for weights of the form \eqref{w} with $\alpha=m\in\N$. 

We next consider the case of negative integer powers.  If $\alpha=-m$ is a negative integer, then $T_{n+m}^{w}$ has a zero of order $m$ at $x = b$ and hence
		\begin{equation}
			w(x) T_{n+m}^{w}(x) = w_0(x) T_n^{w_0}(x)\frac{(x-b)^m}{|x-b|^m}.
		\end{equation}
Therefore,
		\begin{align*}
		\begin{split}
		\|wT_{n+m}^{w}\|_{[-1,1]} &\sim 
		  2^{1-n}\exp\left(\frac{1}{\pi}\int_{-1}^{1}\frac{\log w_0(x)}{\sqrt{1-x^2}}dx\right)\\
		  & \sim 2^{1-n-m}\exp\left(\frac{1}{\pi}\int_{-1}^{1}\frac{\log w(x)}{\sqrt{1-x^2}}dx\right)
		\end{split}
		\end{align*}
as $n\to\infty$, and this proves the validity of \eqref{eq:bernstein_asymptotics} for such weights as well. 
		
		%To extend these asymptotic formulae to include the case of weights whose representation is given by $w_0(x)|x-b|^\alpha$ 
		To handle arbitrary exponents $\alpha\in \R$, it suffices to examine the case of $\alpha\in (0,1)$. The weight function $w_0$ can namely incorporate zeros with integer exponents on $[-1,1]$ as we have already addressed the asymptotics in this particular case. Let $\epsilon>0$ and form the weight functions
		\begin{align*}
			w_l^{\epsilon}(x)&=\begin{cases}
				w(x),& |x-b|\geq\epsilon \\
				w_0(x)|x-b|,& |x-b|<\epsilon,
			\end{cases}  \\
			w_u^{\epsilon}(x)&=\begin{cases}
				w(x),& |x-b|\geq\epsilon \\
				w_0(x),& |x-b|<\epsilon.
			\end{cases}
		\end{align*}
Note that \eqref{eq:bernstein_asymptotics} applies to both $w_l^{\epsilon}$ and $w_u^{\epsilon}$. Regardless of the value of $0<\epsilon<1$, we have the inequalities
		\begin{equation*}
			w_l^{\epsilon}(x)\leq w(x)\leq w_u^{\epsilon}(x)
		\end{equation*}
and hence %\eqref{eq:monotonicity_cheb_weight} implies that
		\begin{equation}
		   \|w_l^{\epsilon}T_n^{w_l^{\epsilon}}\|_{[-1,1]}
		   \leq \|w T_n^{w}\|_{[-1,1]}
		   \leq \|w_u^{\epsilon}T_n^{w_u^{\epsilon}}\|_{[-1,1]}.\label{eq:weighted_chebyshev_inequalities}
		   \end{equation}	
Since we already know the limiting values of the upper and lower bounds we consequently have

\begin{align}
	\begin{split}
		\exp\left(\frac{1}{\pi}\int_{-1}^{1}\frac{\log  w_l^{\epsilon}(x)}{\sqrt{1-x^2}}dx\right)&\leq \liminf_{n\rightarrow\infty}	2^{n-1}\|T_n^ww\|_{[-1,1]}\\&\leq \limsup_{n\rightarrow \infty}2^{n-1}\|T_n^ww\|_{[-1,1]}\leq \exp\left(\frac{1}{\pi}\int_{-1}^{1}\frac{\log w_u^{\epsilon}(x)}{\sqrt{1-x^2}}dx\right).
	\end{split}
	%\label{eq:limsupinf_polynomial_weights}
\end{align}
By dominated convergence, 
	\begin{equation*}
		\lim_{\epsilon\rightarrow 0} \int_{-1}^{1}\frac{\log w_l^{\epsilon}(x)}{\sqrt{1-x^2}}dx = 
		\int_{-1}^{1}\frac{\log w(x)}{\sqrt{1-x^2}}dx =
		\lim_{\epsilon\rightarrow 0} \int_{-1}^{1}\frac{\log w_u^{\epsilon}(x)}{\sqrt{1-x^2}}dx,
	\end{equation*}
	and therefore we can conclude that
%	 	Equation \eqref{eq:weighted_chebyshev_inequalities} together with the asymptotic formulae for $w_l^\delta$ and $w_u^\delta$ gives us that for a fixed $\delta>0$
%so for any given $r>0$, we can choose $\epsilon>0$ small enough that
%	 \begin{equation}
%	     1-r \leq \frac{\exp\left\{\frac{1}{\pi}\int_{-1}^{1}
%	     \frac{\log w_l^{\epsilon}(x)}{\sqrt{1-x^2}}dx\right\}}
%	     {\exp\left\{\frac{1}{\pi}\int_{-1}^{1}\frac{\log w(x)}{\sqrt{1-x^2}}dx\right\}}
%	     \leq 1 \leq
%	     \frac{\exp\left\{\frac{1}{\pi}\int_{-1}^{1}
%	     \frac{\log w_u^{\epsilon}(x)}{\sqrt{1-x^2}}dx\right\}}
%	     {\exp\left\{\frac{1}{\pi}\int_{-1}^{1}\frac{\log w(x)}{\sqrt{1-x^2}}dx\right\}}
%	     \leq 1+r.
%	 \end{equation}
%This in turn implies that
%	\begin{multline}
%	    \qquad 2^{1-n}\exp\left\{\frac{1}{\pi}\int_{-1}^{1}
%	    \frac{\log w(x)}{\sqrt{1-x^2}}dx\right\} (1-r) \bigl(1+o(1)\bigr) \\
%	    \leq \|wT_n^w\|_{[-1,1]}
%	    \leq 2^{1-n}\exp\left\{\frac{1}{\pi}\int_{-1}^{1}
%	    \frac{\log w(x)}{\sqrt{1-x^2}}dx\right\}(1+r)\bigl(1+o(1)\bigr)  \qquad
%	 \end{multline}
%as $n\rightarrow \infty$. Since $r>0$ is arbitrary, we conclude that
	 \begin{equation}
	    \|wT_n^w\|_{[-1,1]} \sim 
	    2^{1-n}\exp\left(\frac{1}{\pi}\int_{-1}^{1}
	    \frac{\log w(x)}{\sqrt{1-x^2}}dx\right)
	 \end{equation}
as $n\rightarrow \infty$. In other words, \eqref{eq:bernstein_asymptotics} applies.  The addition of more weights of the form $|x-b_k|^{\alpha_k}$ can be carried out by repeated use of the argument explained above.

	\end{proof}

	Apart from Bernstein's formula in Theorem \ref{thm:Bernstein}, few asymptotic results have been established for Chebyshev polynomials associated with weights which vanish on parts of the interval $[-1,1]$. If $w>0$ holds almost everywhere then the $n$th root asymptotics are known. In this case  \[\lim_{n\rightarrow \infty}\left(T_n^{[-1,1],w}(z)\right)^{1/n} = \frac{z+\sqrt{z^2-1}}{2}\]
	holds for $z\notin [-1,1]$, see \cite{fekete-walsh54}. The precise asymptotical behavior -- so-called Szeg\H{o}--Widom asymptotics -- of $T_n^{[-1,1],w}$ on $\C\setminus [-1,1]$ were determined in \cite{lubinsky-saff87} in the case where $w$ is strictly positive. In \cite{kroo-peherstorfer08,kroo14} the asymptotical behavior of $T_n^{[-1,1],w}$ on $[-1,1]$ was determined for positive smooth weight functions. This was done using Chebyshev's result for weights $w$ which are given as reciprocals of polynomials that are strictly positive on $[-1,1]$ together with a polynomial approximation argument. We will later see that Theorem \ref{thm:Bernstein} has applications to the analysis of Chebyshev polynomials corresponding to sets in the complex plane.
	
    Following in Bernstein's footsteps, Achieser considered in 1933 the case of weighted Chebyshev polynomials with respect to disjoint intervals of the form $\sfE(a,b) = [-1,-a]\cup [b,1]$ with~$0<a,b<1$. In the particular case where $b = a$, we write $\sfE(a,b) = \sfE(a)$.  

%\begin{figure}
%\centering
%\includegraphics[width=0.3\textwidth]{Figures/Naum_Akhiezer.jpg}
%\caption{Naum Achieser (1901-1980).}
%\label{fig:akhiezer}
%\end{figure}
    \begin{theorem}[Achieser (1933) \cite{achieser28}]
    	For any $n\in \N$
    	\begin{equation}
    		t_{2n}(\sfE(a))= 2^{1-2n}(1-a^2)^n. \\
    	\end{equation}	
    	As $n\rightarrow \infty$
    	\begin{align*}
    		t_{2n+1}(\sfE(a))  \sim 2^{-2n}(1-a^2)^{n+\frac{1}{2}}\sqrt{\frac{1+a}{1-a}}.
    	\end{align*}
    	\label{thm:achieser_disjoint_intervals}
    \end{theorem}
    \vspace{-.5cm}
   	Getting a bit ahead of our narrative, the number \(\tau(a) = \frac{1}{2}\sqrt{1-a^2}\) is the so-called \textit{transfinite diameter} or \textit{logarithmic capacity} of \(\sfE(a)\). As we see from this example
   	\[\lim_{n\to\infty}t_n(\sfE(a))^{1/n} = \tau(a),\]
   	and in fact as the forthcoming Theorem \ref{thm:ffs-theorem} shows, \(t_n(\sfE)^{1/n}\) has a limit as \(n\to\infty\) for any compact set \(\sfE\). Theorem \ref{thm:achieser_disjoint_intervals} implies that
   	\begin{equation}
   		\frac{t_{2n}(\sfE(a))}{\tau(a)^{2n}} = 2,\quad \lim_{n\to\infty}\frac{t_{2n+1}(\sfE(a))}{\tau(a)^{2n+1}} =2 \cdot \sqrt{\frac{1+a}{1-a}}.
   		\label{eq:achieser_limit_pts}
   	\end{equation}
   	
   	Achieser provides the full asymptotic formula for any choice of $0<a,b<1$ including the possible effects of a weight function. This generalizes Bernstein's formula and we refer the reader to \cite[Appendix E]{achieser56} for details. The emerging pattern is that in the generic case when~$a$ and $b$ differ, by letting 
   	\[\tau(a,b) = \lim_{n\to\infty}t_n(\sfE(a,b))^{1/n}\]
   	the sequence $\{t_n(\sfE(a,b))/\tau(a,b)^n\}$ has a full interval of limit points rather than just two points as given in \eqref{eq:achieser_limit_pts}. A recent proof of Theorem \ref{thm:achieser_disjoint_intervals} using elliptic functions is given in \cite{schiefermayr19}, we will provide a novel proof based on Theorem \ref{thm:Bernstein}. Theorem \ref{thm:achieser_disjoint_intervals} points out a recurring phenomenon in the world of Chebyshev polynomials associated to compact sets which have several components. The limit behavior of $t_n(\sfE,w)$ may differ along different subsequences. This was studied in further detail by Widom in \cite{widom69}.

	\begin{proof}[Proof of Theorem \ref{thm:achieser_disjoint_intervals}]
		Due to symmetry of $\sfE(a)$ and uniqueness of the corresponding Chebyshev polynomial we obtain $T_{n}^{\sfE(a)}(-x) = (-1)^nT_{n}^{\sfE(a)}(x)$. Equivalently, $T_{2n}^{\sfE(a)}$ is an even polynomial and $T_{2n+1}^{\sfE(a)}$ is an odd polynomial. We write
\begin{align*}
T_{2n}^{\sfE(a)}(x) & = x^{2n}+\sum_{k=0}^{n-1}a_kx^{2k} = P_{n}^{\sfE(a)}(x^2),\\
T_{2n+1}^{\sfE(a)}(x) &= x^{2n+1}+\sum_{k=0}^{n-1}b_kx^{2k+1} = xQ_n^{\sfE(a)}(x^2)	
\end{align*}

where $Q_n^{\sfE(a)}$ and $P_n^{\sfE(a)}$ are $n$th degree monic polynomials. By changing the variable from $x$ to $t = x^2$ we find that $Q_n^{\sfE(a)}$ and $P_n^{\sfE(a)}$ are the $n$th degree monic minimizer of the expressions
\[
\max_{t\in [a^2,1]}\left|w(t)\left(t^n+\alpha_{n-1}t^{n-1}+\cdots +\alpha_{0} \right)\right|
\]
with $w(t) = \sqrt{t}$ and $w(t) =1$ respectively. By Theorem \ref{thm:alternation_theorem} one immediately concludes that~$P_n^{\sfE(a)}$ can be explicitly represented by
\[P_n^{\sfE(a)}(t) = T_n^{[-1,1]}\left(\frac{2t-1-a^2}{1-a^2}\right)\left(\frac{1-a^2}{2}\right)^n.\]
We may therefore conclude that
\[t_{2n}(\sfE(a)) = 2^{1-2n}(1-a^2)^n.\]

To determine the odd norms we consider the change of variables $\xi = \frac{2}{1-a^2}\left(t-\frac{1+a^2}{2}\right)\Leftrightarrow \frac{1-a^2}{2}\xi+\frac{1+a^2}{2}$. This yields
\[t_{2n+1}(\sfE(a))= \min_{\beta_0,\dotsc, \beta_{n-1}}\max_{\xi\in [-1,1]}\left(\frac{1-a^2}{2}\right)^n\left|\sqrt{\frac{1-a^2}{2}\xi+\frac{1+a^2}{2}}\left(\xi^n+\sum_{k=0}^{n-1}\beta_k\xi^k\right)\right|.\]
Theorem \ref{thm:Bernstein} provides the precise asymptotical formula
\[t_{2n+1}(\sfE(a)) \sim  2^{1-n}\left(\frac{1-a^2}{2}\right)^n\exp\left(\frac{1}{\pi}\int_{-1}^{1}\frac{\log \sqrt{\frac{1-a^2}{2}\xi+\frac{1+a^2}{2}}}{\sqrt{1-\xi^2}}d\xi\right)\]
as $n\rightarrow \infty$. The integral is effectively computed using Lemma \ref{lem:log_integral_equilibrium_measure} and we find that
\[\exp\left(\frac{1}{\pi}\int_{-1}^{1}\frac{\log \sqrt{\frac{1-a^2}{2}\xi+\frac{1+a^2}{2}}}{\sqrt{1-\xi^2}}d\xi\right) = \frac{1+a}{2}.\]
In conclusion,
\[t_{2n+1}(\sfE(a)) \sim  2^{1-n}\left(\frac{1-a^2}{2}\right)^n\frac{1+a}{2}.\qedhere\]
	\end{proof}
	
Theorem \ref{thm:achieser_disjoint_intervals} has applications to the study of Chebyshev polynomials corresponding to arc of the unit circle. We will return to this in Section \ref{subs:cheb_arc}.

\section{Chebyshev polynomials in the complex plane}

	\label{sec:chebyshev_complex}

	The first broadening of the concept of Chebyshev polynomials to the setting of the complex plane is due to Faber in \cite{faber20}. This foundational work, published in 1919, predates the contributions of Bernstein and Achieser. We begin by exploring Faber's results and the subsequent developments by Suetin and Widom, employing the framework of conformal mappings. Later, we consider the applications of potential theory to this area, providing the general results established by Fekete and Szeg\H{o} among others. Familiarity with the material in \cite{ransford95} is assumed for references to potential theory.
	
	We mention that Faber's generalization of Chebyshev polynomials and their residual counterparts \cite{christiansen-simon-zinchenko-V} are of interest within the study of convergence estimates for Krylov subspace methods which are used to analyze spectral properties of a matrix as well as solving linear systems, see e.g. \cite{beckermann-kuijlaars02,toh-trefethen98}.

\subsection{Chebyshev polynomials relative to a Jordan curve}	
\label{subs:cheb_curve}
Faber begins his investigation in \cite{faber20} by letting $T_n^{\sfE}$ be defined as the monic minimizer of degree $n$ with respect to the supremum norm on a compact set $\sfE$ in $\C$. He then goes on with mentioning some simple geometric cases where these polynomials can be explicitly determined. Without proof, he states the following result, although he claims that it is ``just as obvious'' as the determination of \eqref{eq:cheb_unit_cirlce}. We choose to include a proof here since it illustrates a reasoning that is central to estimates of Chebyshev polynomials in the complex plane, an argument that will be reused several times throughout.
%\begin{figure}
%\centering
%\includegraphics[width=0.3\textwidth]{Figures/faber_1.jpg}
%\caption{Georg Faber (1877-1966).}
%\label{fig:faber}
%\end{figure}
	\begin{theorem}[Faber (1920) \cite{faber20}]
		Let $P(z) = z^m+a_{m-1}z^{m-1}+\cdots+a_0$. If
		\[\sfE(r) = \{z: |P(z)| = r^m\}\]
		for $r>0$, then
		\[T_{nm}^{\sfE(r)}(z) = P(z)^n\]
		for any $n\in \N$.
		\label{thm:chebyshev_on_lemnsicate}
	\end{theorem}
	The set $\sfE(r)$ is called a lemniscate. If $r>0$ is sufficiently small then $\sfE(r)$ will contain as many components as the number of distinct zeros of $P$. For large enough $r$ on the other hand, $\sfE(r)$ will consist of one component. The notation \(\sfE(r)\) will be consistent when we later use this to denote the corresponding equipotential curves of \(\sfE\).
	\begin{proof}
		%As alluded to by Faber, the proof follows along the same lines as the proof used to show \eqref{eq:cheb_unit_cirlce}. Indeed let 
		Let $Q$ be any monic polynomial of degree $nm$ and form the quotient
		\[\varphi(z) = \frac{Q(z)}{P(z)^n}.\]
		The function \(\varphi\) is analytic on $\{z: |P(z)|>r^m\}$ since all zeros of $P$ lie inside $\{z:|P(z)|<r^m\}$. Analyticity extends to the point at infinity since 
		\[\lim_{z\rightarrow \infty}\varphi(z) = 1.\]
		By the maximum principle applied to the unbounded component we find that
		\[1\leq \|\varphi\|_{\sfE(r)} = \frac{\|Q\|_{\sfE(r)}}{r^{nm}}.\]
		We conclude that $\|Q\|_{\sfE(r)}\geq r^{nm}$ holds for any monic polynomial $Q$ of degree $nm$. Since $\|P^n\|_{\sfE(r)} = r^{nm}$ we conclude that \(P^n\) is a minimizer, and since there is only one minimizer, this implies that $T_{nm}^{\sfE(r)} = P(z)^n$.
	\end{proof}
	Faber does not address the issue of determining Chebyshev polynomials for degrees other than multiples of the degree of the generating polynomial (in this case $nm$). In fact, the analysis of the remaining degrees can be quite involved, see e.g. \cite{bergman-rubin24, peherstorfer-steinbauer01}. One of the main points of Faber's article is to show that the classical Chebyshev polynomials $T_n^{[-1,1]}$ are also Chebyshev polynomials in the extended sense to ellipses in the complex plane with focii at~$\pm 1$. We will return to this in Section \ref{subs:equipotential}
	
	Perhaps of even greater influence on subsequent research in approximation theory, Faber demonstrated that a  class of polynomials, introduced by him in \cite{faber03} and now known as Faber polynomials, can be used to construct sequences of polynomials that, in certain cases, asymptotically achieve the same minimal norm as Chebyshev polynomials. We proceed with explaining this chain of ideas. For this purpose, let $\sfE$ denote a compact set and let $\Omega_\sfE$ denote the unbounded component of \(\overline{\C}\setminus\sfE.\) The maximum principle implies that $T_n^{\sfE} = T_n^{\partial\Omega_{\sfE}}$, and therefore it is of no importance in what follows whether we consider Chebyshev polynomials on a set or its corresponding outer boundary. With the additional assumption that $\Omega_{\sfE}$ is simply connected with respect to \(\RS\), the Riemann mapping theorem implies that there exists a conformal map $\Phi:\Omega_\sfE\rightarrow \RS\setminus \overline{\D}$ satisfying
	\begin{equation}
		\Phi(\infty) = \infty,\quad \Phi'(\infty)\coloneqq\lim_{z\rightarrow \infty}\frac{\Phi(z)}{z}>0,
		\label{eq:canonical_conformal_map}
	\end{equation}
	see e.g. \cite[Theorem VII.4.2]{conway78}. It follows that $\Phi$ has the Laurent series expansion
	\begin{equation}
		\Phi(z) = \Phi'(\infty)z+a_0+a_{-1}z^{-1}+\cdots
		\label{eq:conformal_map_laurent_expansion}
	\end{equation}
	at infinity. The argument used to prove Theorem \ref{thm:chebyshev_on_lemnsicate} has the following adaptation. Let $Q$ be any monic polynomial of degree $n\in \N$. Then
	\begin{equation}
		\varphi(z) = \Phi'(\infty)^n\frac{Q(z)}{\Phi(z)^n}
		\label{eq:maximum_modulus_func}
	\end{equation}
	is analytic on $\Omega_\sfE$ and $\varphi(\infty) = 1$. Since $|\Phi(z)|\rightarrow 1$ as $z\rightarrow \partial\Omega_\sfE$, the maximum modulus theorem implies that
	\[1\leq \|\varphi\|_{\sfE}= \Phi'(\infty)^n\|Q\|_{\sfE}\]
	and hence we obtain the following.
	
	\begin{theorem}[Faber 1920 \cite{faber20}]
	Let \(\sfE\) denote a Jordan curve in the complex plane with associated exterior map \(\Phi\) satisfying \eqref{eq:conformal_map_laurent_expansion}. Then,
	\begin{equation}
		\Phi'(\infty)^n t_n(\sfE)\geq 1.
		\label{eq:conformal_map_chebyshev_lower_bound}
	\end{equation}	
	\end{theorem}

%	In order to show that \(\Phi'(\infty)^nt_n(\sfE)\sim 1\) we therefore need to construct a family of polynomials that asymptotically saturate this lower bound. In order that equality holds in \eqref{eq:conformal_map_chebyshev_lower_bound} it is clear from \eqref{eq:maximum_modulus_func} that \(T_n^{\sfE,w}\) and \(\Phi(z)^n\) must be proportional. This can hold if and only if \(\Phi(z)^n\) is a polynomial which rarely holds. To remedy this one may consider the Faber polynomial of degree $n$ corresponding to $\sfE$ is denoted with $F_n^{\sfE}$ and is defined as the polynomial part of \(\Phi(z)^n\), in other words, it is the polynomial satisfying
%	\begin{equation}
%		F_n^{\sfE}(z) = \Phi(z)^n+O(z^{-1})
%		\label{eq:faber_polynomial_definition}
%	\end{equation}
%	as $z\rightarrow \infty$. 
To show that \(\Phi'(\infty)^n t_n(\sfE)\sim 1\), it therefore suffices to construct a family of
polynomials that asymptotically saturate this lower bound.
For equality to hold in \eqref{eq:conformal_map_chebyshev_lower_bound}, it follows from~\eqref{eq:maximum_modulus_func} that \(T_n^{\sfE,w}\) and \(\Phi(z)^n\) must be proportional.
This can occur if and only if \(\Phi(z)^n\) is itself a polynomial, which rarely happens.

To overcome this difficulty, one considers the Faber polynomial of degree \(n\) associated with
\(\sfE\), denoted by \(F_n^{\sfE}\), defined as the polynomial part of \(\Phi(z)^n\).
More precisely, \(F_n^{\sfE}\) is the unique polynomial satisfying
\begin{equation}
	F_n^{\sfE}(z)=\Phi(z)^n+O(z^{-1}), \qquad z\to\infty.
	\label{eq:faber_polynomial_definition}
\end{equation}

It is clear that $F_n^{\sfE}$ defined this way is an \(n\)th degree polynomial with leading coefficient~\(\Phi'(\infty)^n\) and consequently
	\begin{equation}
		\Phi'(\infty)^nt_n(\sfE)\leq \|F_n^{\sfE}\|_{\sfE}.
		\label{eq:faber_upper_bnd_cheb}
	\end{equation}
	Furthermore, \(F_n^{\sfE}-\Phi(z)^n\) is analytic in \(\Omega_\sfE\) and vanishes at \(\infty\). Therefore, if \(z\) and \(\sfE\) are contained inside the ball of radius \(R>0\) centered at the origin, then
	\[\left|\int_{|\zeta| = R}\frac{F_n^{\sfE}(\zeta)-\Phi(\zeta)^n}{\zeta-z}d\zeta\right|\leq \frac{2\pi R}{R-|z|}\max_{|\zeta| = R}|F_n^{\sfE}(\zeta)-\Phi(\zeta)^n|\to 0,\quad R\to \infty.\]
	However, since \(F_n^{\sfE}\) is a polynomial, Cauchys integral formula implies that
	\[\int_{|\zeta| = R}\frac{F_n^{\sfE}(\zeta)}{\zeta-z}d\zeta = F_n^{\sfE}(z).\]
	We obtain the representation
	\[F_n^{\sfE}(z) = \frac{1}{2\pi i}\int_{|\zeta| = R}\frac{\Phi(\zeta)^n}{\zeta-z}d\zeta.\]
	Alternatively, using the notation \[\sfE(r) = \{z: |\Phi(z)| =r \},\quad r>1\] to mean the level curves of \(\Phi\) parametrized in the positive direction, we find that if \(R>1\) and \(z\) lies interior to \(\sfE(R)\) then
	\begin{equation}
		F_n^{\sfE}(z) = \frac{1}{2\pi i}\int_{\sfE(R)}\frac{\Phi(\zeta)^n}{\zeta-z}d\zeta.
		\label{eq:faber_cauchy_representation}
	\end{equation}
	Letting \(\Psi:\overline{\C}\setminus \overline \D\to \Omega_\sfE\) denote the inverse of \(\Phi\), then the change of variables \(\zeta = \Psi(s)\) transforms \eqref{eq:faber_cauchy_representation} to
	\begin{equation*}
		F_n^{\sfE}(z) = \frac{1}{2\pi i}\int_{|s| = R}\frac{s^n\Psi'(s)}{\Psi(s)-z}ds.
	\end{equation*}
	As a consequence, we observe that the the analytic function \(s\mapsto \Psi'(s)/(\Psi(s)-z)\) has the Laurent series expansion 
	\begin{equation}
		\frac{\Psi'(s)}{\Psi(s)-z} = \sum_{k=0}^{\infty}F_k^{\sfE}(z)s^{-k-1},\quad s\to \infty.
		\label{eq:faber_laurent_series}
	\end{equation}

	If further, \(z\) lies within the domain of definition of \(\Phi\) with \(r<|\Phi(z)|<R\), then Cauchy's integral formula implies that
	\begin{align}
	\begin{split}
	\Phi(z)^n & = \frac{1}{2\pi i}\int_{\sfE(R)}\frac{\Phi(\zeta)^n}{\zeta-z}d\zeta-\frac{1}{2\pi i}\int_{\sfE(r)}\frac{\Phi(\zeta)^n}{\zeta-z}d\zeta \\
	& = F_n^{\sfE}(z)-\frac{1}{2\pi i}\int_{\sfE(r)}\frac{\Phi(\zeta)^n}{\zeta-z}d\zeta.
	\end{split}
	\label{eq:faber_on_bdry}
	\end{align}
	see also \cite[Chapter 2]{smirnov-lebedev68}.
	Faber, in \cite{faber20}, provides the following argument to show that if the boundary of $\sfE$ is smooth enough then $\Phi'(\infty)^{n}t_n(\sfE)\sim 1$ as $n\rightarrow \infty$. 
	
	Assume that $\sfE$ is the closure of an analytic Jordan domain or, equivalently formulated,~$\Phi$ extends analytically to some neighborhood of $\partial\Omega_{\sfE}$. Then there exists some \(r\in (0,1)\) such that \eqref{eq:faber_on_bdry} holds with \(z\in \sfE\), and we obtain
	\[F_n^{\sfE}(z) = \Phi(z)^n+\frac{1}{2\pi i}\int_{\sfE(r)}\frac{\Phi(\zeta)^n}{\zeta-z}d\zeta.\]
	We immediately conclude that if $z\in \sfE = \sfE(1)$ then there exists some $C>0$ -- independent of \(n\) and \(z\in \sfE\) -- such that
	\begin{equation}
		|F_n^{\sfE}(z)|\leq 1+Cr^n.
		\label{eq:faber_upper_est_analytic}
	\end{equation}
	By combining \eqref{eq:faber_upper_est_analytic} with \eqref{eq:conformal_map_chebyshev_lower_bound} and \eqref{eq:faber_upper_bnd_cheb} we find that
	\[1\leq \Phi'(\infty)^nt_n(\sfE)\leq 1+Cr^n\]
	and since $0<r<1$ we obtain the first half of the following theorem.
	
	\begin{theorem}[Faber (1920) \cite{faber20}]
		Let $\sfE$ denote the closure of an analytic Jordan domain with exterior conformal map $\Phi:\Omega_\sfE\rightarrow \RS\setminus \overline{\D}$ as in \eqref{eq:conformal_map_laurent_expansion}. Then there is some \(r\in (0,1)\) such that
		\label{thm:faber_asymptotics}
		\begin{equation}
			t_n(\sfE)\Phi'(\infty)^n = 1+O(r^n), \quad n\to\infty.
			\label{eq:faber_widom_asymptotics}
		\end{equation} 
		Furthermore,
		\begin{equation}
			T_n^{\sfE}(z)\Phi'(\infty)^n\Phi(z)^{-n}=1+o(1),\quad n\to \infty
			\label{eq:faber_szego_widom_asymptotics}
		\end{equation}
		uniformly on closed subsets of $\Omega_\sfE$.
	\end{theorem}
	\begin{proof}
		We are left to prove that the left hand side of \eqref{eq:faber_szego_widom_asymptotics} converges locally uniformly to $1$. For this purpose, note that the functions
		\[\varphi_n(z) = T_n^{\sfE}(z)\Phi'(\infty)^n\Phi(z)^{-n}\]
		are analytic on $\Omega_\sfE$ and attain the value $1$ at infinity. Since we further have that $\|\varphi_n\|_{\sfE}\leq 1+Cr^n$ for some $C>0$ and $0<r<1$, Montel's theorem (see e.g. \cite[Theorem VII.2.9]{conway78}) implies that $\{\varphi_n\}$ is a normal family in $\Omega_\sfE$. Since any convergent subsequences of $\{\varphi_n\}$ must converge to the constant function $1$ locally uniformly on $\Omega_\sfE$ we conclude that so does the full sequence.
	\end{proof}
	
The assumption in Theorem \ref{thm:faber_asymptotics}, that the bounding curve is analytic can be weakened while still maintaining the asymptotic saturation in \eqref{eq:faber_widom_asymptotics}. This requires significantly more work and was first investigated by Suetin in \cite{suetin64}, see also \cite{suetin74, suetin84}. We will follow the approach used in a recent preprint \cite{minadiaz-rubin-wennman25} which combines ideas of Pommerenke from \cite{pommerenke64, pommerenke65} with those of Pritsker \cite{pritsker02}. Our aim is to provide a new proof of the following fundamental result in the field of extremal polynomials.

\begin{theorem}[Suetin (1964) \cite{suetin64}]
	Let \(\sfE\) denote a Jordan curve of class \(C^{1+\alpha}\) where \(0<\alpha <1\). Then,
	\[F_n^{\sfE}(z) = \Phi(z)^n+O\left(\frac{\log n}{n^\alpha}\right),\quad z\in \overline \Omega_{\sfE}.\]
	\label{thm:suetin}
\end{theorem}
\vspace{-.5cm}
That \(\sfE\) is a Jordan curve of class \(C^{1+\alpha}\) means that any arc-length parametrization of \(\sfE\) is continuously differentiable, and the derivative satisfies an \(\alpha\)-H\"{o}lder condition. By the Kellogg-Warschawski theorem, see e.g. \cite[Theorem 4.3]{garnett-marschall05} this is the same as saying that the associated Riemann map \(\Phi:\overline \Omega_\sfE\to \overline{\C}\setminus\D\) is of class \(C^{1+\alpha}\). 

The implications to the study of Chebyshev polynomials is that Theorem \ref{thm:faber_asymptotics} holds under the weaker assumption that the boundary is of class \(C^{1+\alpha}\). Indeed, since \(|\Phi(z)| = 1\) if \(z\in \sfE\), \eqref{eq:faber_upper_bnd_cheb} implies that
\[1\leq t_n(\sfE)\Phi'(\infty)^n\leq \|F_n^{\sfE}\|_{\sfE}=1+O\left(\frac{\log n}{n^\alpha}\right),\]
providing the estimate
\[t_n(\sfE)\Phi'(\infty)^n = 1+O\left(\frac{\log n}{n^\alpha}\right),\quad n\to \infty.\]

We will prove the equivalent formulation of Theorem \ref{thm:suetin} that
\[F_n^{\sfE}(\Psi(e^{i\theta}))-e^{in\theta} = O\left(\frac{\log n}{n^\alpha}\right),\quad n\to \infty.\]
Our approach is based on an integral representation due to Pommerenke of the Faber polynomials whose proof we now recount. If \(g\) is analytic in \(\overline \C\setminus \overline \D\), extending continuously to~\(|z| = 1\) with \(\Re g(\infty) = 0\), then
\begin{equation}
	g(s) = i\int_{0}^{2\pi}\frac{s+e^{it}}{s-e^{it}}\Im g(e^{it})\frac{dt}{2\pi},\quad |z|>1.
	\label{eq:poisson_representation}
\end{equation}
Let us fix \(\theta \in [0,2\pi)\) and recall that \(\Psi = \Phi^{-1}\). By appropriately choosing a branch of the logarithm we can ensure that
\[g(s) = \log \left(\frac{\Psi(s)-\Psi(e^{i\theta})}{\Psi'(\infty)s}\right)\]
extends analytically at \(\infty\), where it also vanishes. Consequently, if \(|s|>R>1\), then by applying \eqref{eq:poisson_representation} to \(g_R(sR^{-1})\) where \(g_R(s) = g(Rs)\) we obtain
\begin{align*}
\log \left(\frac{\Psi(s)-\Psi(e^{i\theta})}{\Psi'(\infty)s}\right) = i\int_{0}^{2\pi}\frac{sR^{-1}+e^{it}}{sR^{-1}-e^{it}}\Bigl(\arg \left(\Psi(Re^{it})-\Psi(e^{i\theta})\right)-t\Bigr)\frac{dt}{2\pi}.
\end{align*}
Following \cite[Lemma 1]{pommerenke65}, we denote
\begin{equation}
	v_\theta(t) = \arg \bigl(\Psi(e^{it})-\Psi(e^{i\theta})\bigr),\quad t\neq \theta \mod 2\pi
	\label{eq:def_v_theta}
\end{equation}
and observe \(\arg \left(\Psi(Re^{it})-\Psi(e^{i\theta})\right)\rightarrow v_\theta(t)\) as \(R\to 1\). Since \(\sfE\) is smooth, the function~\(z\mapsto \arg\left(z-\Psi(e^{i\theta})\right)\) is bounded on \(\Omega_{\sfE}\) and therefore we can use the \emph{Dominated convergence theorem}, when passing \(R\to 1\), to obtain
\begin{align}
\begin{split}
\log \left(\frac{\Psi(s)-\Psi(e^{i\theta})}{\Psi'(\infty)s}\right) &= i\int_{0}^{2\pi}\frac{1+e^{it}s^{-1}}{1-e^{it}s^{-1}}\bigl(v_\theta(t)-t\bigr)\frac{dt}{2\pi}\\&=i\int_{0}^{2\pi}\bigl(v_\theta(t)-t\bigr)\frac{dt}{2\pi}+\frac{i}{\pi}\int_{0}^{2\pi}\frac{e^{it}s^{-1}}{1-e^{it}s^{-1}}\bigl(v_\theta(t)-t\bigr)dt.
\label{eq:pommernke_intermediate_step}	
\end{split}
\end{align}
%Since the left hand side vanishes at \(\infty\), we find by letting \(s\to \infty\) that
As \(s\to\infty\) both the left hand-side and the rightmost term of \eqref{eq:pommernke_intermediate_step} vanishes and consequently
\[\int_{0}^{2\pi}\bigl(v_\theta(t)-t\bigr)\frac{dt}{2\pi} = 0.\]
We have therefore shown that 
\begin{align*}
	\log \left(\frac{\Psi(s)-\Psi(e^{i\theta})}{\Psi'(\infty)s}\right)&=\frac{i}{\pi}\int_{0}^{2\pi}\frac{e^{it}s^{-1}}{1-e^{it}s^{-1}}\bigl(v_\theta(t)-t\bigr)dt \\&= \frac{-1}{\pi}\int_{0}^{2\pi}\frac{d}{dt}\Bigl(\log\left(1-e^{it}s^{-1}\right)\Bigr)\bigl(v_\theta(t)-t\bigr)dt.	
\end{align*}
As we will later see, \(v_\theta\) has bounded variation, and therefore we can apply partial integration, see e.g. \cite[Theorem 3.36]{folland99}, obtaining that
\[\log \left(\frac{\Psi(s)-\Psi(e^{i\theta})}{\Psi'(\infty)s}\right)=\frac{1}{\pi}\int_{0}^{2\pi}\log(1-e^{it}s^{-1})dv_\theta(t)-\int_{0}^{2\pi}\log(1-e^{it}s^{-1})dt.\]
The last term is equal to zero, see e.g. \cite[Exercise 2.2.2]{ransford95}, rendering the formula
\[\log \left(\frac{\Psi(s)-\Psi(e^{i\theta})}{\Psi'(\infty)s}\right)=\frac{1}{\pi}\int_{0}^{2\pi}\log(1-e^{it}s^{-1})dv_\theta(t).\]
Differentiating both sides with respect to \(s\) and multiplying both sides with the same variable we obtain
\[\frac{s\Psi'(s)}{\Psi(s)-\Psi(e^{i\theta})}-1 = \frac{1}{\pi}\int_{0}^{2\pi}\frac{-s^{-1}e^{it}}{1-e^{it}s^{-1}}dv_\theta(t) = \sum_{k=1}^{\infty}s^{-k}\frac{1}{\pi}\int_{0}^{2\pi}e^{itk}dv_\theta(t).\]
Evaluating \eqref{eq:faber_laurent_series} at \(z=\Psi(e^{i\theta})\) gives us
\[\frac{s\Psi'(s)}{\Psi(s)-\Psi(e^{i\theta})} = \sum_{k=0}^{\infty}F_k^{\sfE}(\Psi(e^{i\theta}))s^{-k} = 1+\sum_{k=1}^{\infty}s^{-k}\frac{1}{\pi}\int_{0}^{2\pi}e^{itk}dv_\theta(t).\]
By equating the coefficient in front of \(s^{-k}\) we have thus established
\begin{equation}
	F_n^{\sfE}(\Psi(e^{i\theta})) = \frac{1}{\pi}\int_{0}^{2\pi}e^{int}dv_\theta(t),\quad \theta\in [0,2\pi),
	\label{eq:faber_pommerenke_representation}
\end{equation}
see also \cite{pommerenke65}.

In order to proceed with proving Theorem \ref{thm:suetin} we will estimate the right-hand side of \eqref{eq:faber_pommerenke_representation}. The measure \(dv_\theta\) admits the following decomposition,
\begin{align*}
	\frac{1}{\pi}dv_\theta(t) & = \delta_\theta(t)+\frac{1}{\pi}\Im \left(\frac{d}{dt}\log\bigl(\Psi(e^{it})-\Psi(e^{i\theta})\bigr)\right)dt \\&= \delta_\theta(t)+\frac{1}{\pi}\Im\left(\frac{ie^{it}\Psi'(e^{it})}{\Psi(e^{it})-\Psi(e^{i\theta})}\right)dt.	
\end{align*}
Inserting this into \eqref{eq:faber_pommerenke_representation} we obtain
\begin{equation}
	F_n^{\sfE}(\Psi(e^{i\theta}))=e^{in\theta} +\int_{0}^{2\pi}e^{int}\frac{1}{\pi}\Im\left(\frac{ie^{it}\Psi'(e^{it})}{\Psi(e^{it})-\Psi(e^{i\theta})}\right)dt. 
	\label{eq:pommerenke_representation_of_faber}
\end{equation}
Denoting \(\kappa_\theta(t) = \Psi(e^{it})-\Psi(e^{i\theta})\) we may write
\[\frac{ie^{it}\Psi'(e^{it})}{\Psi(e^{it})-\Psi(e^{i\theta})} = \frac{\kappa_\theta'(t)}{\kappa_\theta(t)}.\]
As we already stated, the Kellogg-Warschawski theorem (see e.g. \cite[Theorem 4.3]{garnett-marschall05}) implies that \(t\mapsto\Psi(e^{it})\) is of class \(C^{1+\alpha}\), and as a consequence \(\kappa_\theta'\) satisfies a global \(\alpha\)-H\"{o}lder condition. In other words, there exists a constant \(C>0\), independent of \(\theta\), such that
\begin{equation}
	|\kappa_\theta'(s)-\kappa_\theta'(t)|\leq C|s-t|^\alpha.
	\label{eq:kappa_der_holder}
\end{equation} Therefore,
\begin{align}
	\left|\kappa_\theta(t) -(t-\theta)\kappa_\theta'(\theta)\right|&= \left|\int_{\theta}^{t}\bigl(\kappa_\theta'(s)-\kappa_\theta'(\theta)\bigr)ds\right|\leq \int_{\theta}^{t}C|s-\theta|^\alpha ds \notag\\&= \frac{C}{\alpha+1}(t-\theta)^{\alpha+1} = O(|t-\theta|^{\alpha+1})	,
	\label{eq:kappa_der_under_diff_holder}
\end{align}
where the bounding constant is independent of \(\theta\). Here we implicitly assumed that \(t>\theta\), but the same formula follows if \(t<\theta\). It is important in the forthcoming arguments that any occurrence of a ``big-O'' is \(\theta\)-independent. We will heavily rely on the fact that \(|\kappa_\theta'|\) is continuous and bounded away from \(0\).

We let \(\delta\in (0,\pi)\) be fixed for the moment. Due to the periodicity of \(\kappa_\theta\) and its derivative we observe that
\begin{align*}
	F_n^{\sfE}(\Psi(e^{i\theta}))-e^{in\theta} &= \frac{1}{\pi}\int_{0}^{2\pi}e^{int}\Im\left(\frac{\kappa_\theta'(t)}{\kappa_\theta(t)}\right)dt = \frac{1}{\pi}\int_{\theta-\delta}^{\theta+2\pi-\delta}e^{int}\Im\left(\frac{\kappa_\theta'(t)}{\kappa_\theta(t)}\right)dt\\&=\frac{1}{\pi}\int_{\theta-\delta}^{\theta+\delta}e^{int}\Im\left(\frac{\kappa_\theta'(t)}{\kappa_\theta(t)}\right)dt+\frac{1}{\pi}\int_{\theta+\delta}^{\theta+2\pi-\delta}e^{int}\Im\left(\frac{\kappa_\theta'(t)}{\kappa_\theta(t)}\right)dt.
\end{align*}
We will now demonstrate how \(\delta\to 0\) can be chosen so that the right-hand side is \(O(n^{-\alpha}\log n)\), with a uniform estimate independent of \(\theta\). To handle the integral near the singularity at \(\theta\), we first note that by \eqref{eq:kappa_der_under_diff_holder}
\begin{equation}
	\left|\Im \left(\frac{\kappa_\theta'(t)}{\kappa_\theta(t)}\right)\right| = \left|\Im\left(\frac{\kappa_\theta'(\theta)}{(t-\theta)\kappa_\theta'(\theta)}\Bigl(1+O(|t-\theta|^\alpha)\Bigr)\right)\right|=O(|t-\theta|^{\alpha-1}).
	\label{eq:holder_bnd_im_part}
\end{equation}
Therefore,
\[\left|\frac{1}{\pi}\int_{\theta-\delta}^{\theta+\delta}e^{int}\Im \left(\frac{\kappa_\theta'(t)}{\kappa_\theta(t)}\right)dt\right|\leq C_1\int_{0}^{\delta}u^{\alpha-1}du = \frac{C_1}{\alpha}\delta^\alpha,\]
for some constant \(C_1\). This shows that
\begin{equation}
	F_n^{\sfE}(\Psi(e^{i\theta}))-e^{in\theta} =O(\delta^\alpha)+\frac{1}{\pi}\int_{\theta+\delta}^{\theta+2\pi-\delta}e^{int}\Im\left(\frac{\kappa_\theta'(t)}{\kappa_\theta(t)}\right)dt.
	\label{eq:faber_first_bnd}
\end{equation}
Note that this step also verifies that \(t\mapsto v_\theta(t)\) is a function of bounded variation since we have shown that
\[\frac{d}{dt}v_\theta(t),\quad t\neq \theta\] 
is absolutely integrable.

To handle the behavior in the far-away region we employ the following result from Fourier analysis, see e.g. \cite[(2.3.7) p. 36]{edwards79}. 

\begin{lemma}
	Let \(f\) be a locally integrable, \(2\pi\)-periodic function on \(\R\). Then
\[\left|\int_{0}^{2\pi}e^{int}f(t)dt\right|\leq \frac{1}{2}\int_{0}^{2\pi}|f(t)-f(t+\tfrac{\pi}{n})|dt,\quad n\in \mathbb{N}.\]
\label{lem:edwards_fourier}
\end{lemma}
\begin{proof}[Proof of Theorem \ref{thm:suetin}]
To apply Lemma \ref{lem:edwards_fourier} in our setting, we first introduce the function \(h_\theta\) as the \(2\pi\)-periodic extension of 
\[\begin{cases}
	\Im\left(\frac{\kappa_\theta'(t)}{\kappa_\theta(t)}\right),& \theta+\delta\leq t \leq 2\pi+\theta-\delta \\
	0, & \text{otherwise in }[\theta,\theta+2\pi).
\end{cases}\]
We assume that \(\frac{\pi}{n}\leq \delta\) and define
\begin{align*}
	&I_1 = [\theta+\delta-\pi/n,\theta+\delta], 
	&I_2  = [\theta+\delta,\theta+\pi], \\
	&I_3  = [\theta+\pi,2\pi+\theta-\delta-\pi/n], 
	&I_4  = [2\pi+\theta-\delta-\pi/n, 2\pi +\theta-\delta].
\end{align*}
Lemma \ref{lem:edwards_fourier} implies that
\begin{align*}
	&\left|\frac{1}{\pi}\int_{\theta+\delta}^{2\pi+\theta-\delta}e^{int}\Im \left(\frac{\kappa_\theta'(t)}{\kappa_\theta(t)}\right)\right|dt = \left|\frac{1}{\pi}\int_{\theta}^{2\pi+\theta}e^{int}h_\theta(t)dt\right|	\leq \frac{1}{2\pi}\int_{\theta}^{2\pi+\theta}|h_\theta(t)-h_\theta(t+\tfrac{\pi}{n})|dt\\&\leq \int_{I_1}|h_\theta(t+\tfrac{\pi}{n})|dt+\int_{I_2}|h_\theta(t)-h_\theta(t+\tfrac{\pi}{n})|dt+\int_{I_3}|h_\theta(t)-h_\theta(t+\tfrac{\pi}{n})|dt+\int_{I_4}|h_\theta(t)|dt.
\end{align*}
We will only consider the first two integrals since the remaining two can be handled by completely analogous arguments. The only difference is that the point \(\theta\) is replaced by~\(2\pi +\theta\). Since \(\kappa_\theta(t)\) is \(2\pi\)-periodic this has no effect on the outcome.

First of all, to handle the integration over \(I_1\) (and with a slight modification also \(I_4\)) we note that by \eqref{eq:holder_bnd_im_part} there is some constant \(C_2>0\) independent of \(\theta\) such that
\begin{equation}
	\int_{I_1}|h_\theta(t+\tfrac{\pi}{n})|dt=\frac{1}{\pi}\int_{\theta+\delta}^{\theta+\delta+\pi/n}\left|\Im \left(\frac{\kappa_\theta'(t)}{\kappa_\theta(t)}\right)\right|dt\leq C_2\int_{\delta}^{\delta+\pi/n}u^{\alpha-1}du\leq \frac{C_2\pi}{n}\delta^{\alpha-1}.
	\label{eq:far_away_bnd_0}
\end{equation}
In the final inequality we used the fact that \(u\mapsto u^{\alpha-1}\) is decreasing for \(u>0\), since \(0<\alpha<1\). Secondly, for \(t\in I_2\)
\begin{align*}
	h_\theta(t)-h_\theta(t+\tfrac{\pi}{n})&=\Im\left(\frac{\kappa_\theta'(t)}{\kappa_\theta(t)}-\frac{\kappa_\theta'(t+\tfrac\pi n)}{\kappa_\theta(t+\tfrac \pi n)}\right)\\& = \Im\left(\frac{\kappa_\theta'(t)-\kappa_\theta'(t+\tfrac\pi n)}{\kappa_\theta(t)}\right)-\Im\left(\kappa_\theta'(t+\tfrac\pi n)\left(\frac{1}{\kappa_\theta(t)}-\frac{1}{\kappa_\theta(t+\tfrac\pi n)}\right)\right).	
\end{align*}
We obtain that
\begin{align}
	\begin{split}
		\frac{1}{\pi}\int_{I_2}|h_\theta(t)-h_\theta(t+\tfrac\pi n)|dt\leq \frac{1}{\pi}\int_{I_2}&\left|\Im\left(\frac{\kappa_\theta'(t)-\kappa_\theta'(t+\tfrac\pi n)}{\kappa_\theta(t)}\right)\right|\\+&\left|\Im\left(\kappa_\theta'(t+\tfrac\pi n)\left(\frac{1}{\kappa_\theta(t)}-\frac{1}{\kappa_\theta(t+\tfrac\pi n)}\right)\right)\right|dt.	
	\end{split}
	\label{eq:new_eq}
\end{align}
From \eqref{eq:kappa_der_holder} and \eqref{eq:kappa_der_under_diff_holder} we gather that
\[|\kappa_\theta'(t)-\kappa_\theta'(t+\tfrac\pi n)|\leq O(n^{-\alpha}),\quad \text{and}\quad \kappa_\theta(t) = (t-\theta)\kappa_\theta'(t)\Bigl(1+O(|t-\theta|^{\alpha})\Bigr)\]
and therefore
\[\left|\Im\left(\frac{\kappa_\theta'(t)-\kappa_\theta'(t+\tfrac\pi n)}{\kappa_\theta(t)}\right)\right|\leq \frac{C_3}{n^\alpha}\frac{1}{|t-\theta|}.\]
This has the effect that
\begin{equation}
	\frac{1}{\pi}\int_{\theta+\delta}^{\theta+\pi}\left|\Im\left(\frac{\kappa_\theta'(t)-\kappa_\theta'(t+\tfrac\pi n)}{\kappa_\theta(t)}\right)\right|dt\leq \frac{C_3}{\pi n^\alpha}\log\frac{\pi}{\delta}.
	\label{eq:far_away_bnd_1}
\end{equation}
We see that if \(\delta = \pi /n\) then this term is of the form \(O(n^{-\alpha}\log n)\). All we have left to estimate is 
\[\frac{1}{\pi}\int_{\theta+\delta}^{\theta+\pi}\left|\Im\left(\kappa_\theta'(t+\tfrac\pi n)\left(\frac{1}{\kappa_\theta(t)}-\frac{1}{\kappa_\theta(t+\tfrac\pi n)}\right)\right)\right|dt.\]
%\begin{align*}
%	&|\kappa_\theta(t+\tfrac{\pi}{n})-\kappa_\theta(t)-\tfrac{\pi}{n}\kappa_\theta'(t)| \leq \int_{t}^{t+\tfrac\pi n}\left|\kappa_\theta'(s)-\kappa_\theta'(t)\right|ds\leq C\int_{0}^{\pi/n}u^\alpha du=\frac{C\pi^{\alpha+1}}{\alpha+1}n^{-1-\alpha},\\
%	&\frac{1}{\kappa_\theta(t)} = \frac{1}{(t-\theta)\kappa_\theta'(\theta)}\Bigl(1+O(|t-\theta|^\alpha)\Bigr),\\
%	&\frac{1}{\kappa_\theta(t+\tfrac{\pi}{n})} = \frac{1}{(t+\tfrac{\pi}{n}-\theta)\kappa_\theta'(\theta)}\left(1+O(|t+\tfrac{\pi}{n}-\theta|^\alpha)\right)	 = \frac{1}{(t+\tfrac{\pi}{n}-\theta)\kappa_\theta'(\theta)}\Bigl(1+n^{-1}O(|t-\theta|^{\alpha-1})\Bigr), \\
%	&\kappa_\theta'(t)^2 =\left(\kappa_\theta'(\theta)+O(|t-\theta|^\alpha)\right)^2 = \kappa_\theta'(\theta)^2+O(|t-\theta|^\alpha),
%\end{align*}
We collect the following estimates, valid uniformly for $t\in I_2$:
\begin{equation}\label{eq:kappa_expansions}
\begin{aligned}
\kappa_\theta(t+\tfrac{\pi}{n})-\kappa_\theta(t)
&= \tfrac{\pi}{n}\kappa_\theta'(t)+O(n^{-1-\alpha}),\\[0.3em]
\frac{1}{\kappa_\theta(t)}
&= \frac{1}{(t-\theta)\kappa_\theta'(\theta)}
   \bigl(1+O(|t-\theta|^\alpha)\bigr),\\[0.3em]
\frac{1}{\kappa_\theta(t+\tfrac{\pi}{n})}
&= \frac{1}{(t+\tfrac{\pi}{n}-\theta)\kappa_\theta'(\theta)}
   \Bigl(1+n^{-1}O(|t-\theta|^{\alpha-1})\Bigr),\\[0.3em]
\kappa_\theta'(t)^2
&= \kappa_\theta'(\theta)^2+O(|t-\theta|^\alpha).
\end{aligned}
\end{equation}
Using \eqref{eq:kappa_expansions}, we write
\begin{align}
&\Im\!\left(
  \kappa_\theta'(t+\tfrac{\pi}{n})
  \Bigl(\frac{1}{\kappa_\theta(t)}-\frac{1}{\kappa_\theta(t+\tfrac{\pi}{n})}\Bigr)
\right) \notag\\
&\quad=
\Im\!\left(
  \kappa_\theta'(t+\tfrac{\pi}{n})
  \frac{\kappa_\theta(t+\tfrac{\pi}{n})-\kappa_\theta(t)}
       {\kappa_\theta(t)\kappa_\theta(t+\tfrac{\pi}{n})}
\right).
\end{align}
Substituting the expansions in \eqref{eq:kappa_expansions} yields
\begin{align}
&\Im\!\left(
  \kappa_\theta'(t+\tfrac{\pi}{n})
  \Bigl(\frac{1}{\kappa_\theta(t)}-\frac{1}{\kappa_\theta(t+\tfrac{\pi}{n})}\Bigr)
\right) \notag\\
&\quad=
\frac{\pi}{n}
\Im\!\left(
\frac{
\kappa_\theta'(\theta)^2
+O(|t-\theta|^\alpha)+O(n^{-\alpha})
}{
(t-\theta)(t+\tfrac{\pi}{n}-\theta)\kappa_\theta'(\theta)^2
}
\right)
\Bigl(1+n^{-1}O(|t-\theta|^{\alpha-1})+O(|t-\theta|^\alpha)\Bigr) \notag\\
&\quad=
\frac{
O\!\left(|t-\theta|^\alpha
      +n^{-1}|t-\theta|^{\alpha-1}
      +n^{-\alpha}\right)
}{
n(t-\theta)(t+\tfrac{\pi}{n}-\theta)
}.
\label{eq:kernel_split_estimate}
\end{align}
Integrating over $I_2$ we obtain
\begin{align}
\int_{I_2}
\frac{|t-\theta|^{\alpha-1}}
     {(t-\theta)(t+\tfrac{\pi}{n}-\theta)}\,dt
&\le \int_{\delta}^{\pi} u^{\alpha-3}\,du
 \le \frac{1}{2-\alpha}\,\delta^{\alpha-2}, \notag\\[0.3em]
\int_{I_2}
\frac{|t-\theta|^{\alpha}}
     {(t-\theta)(t+\tfrac{\pi}{n}-\theta)}\,dt
&\le \int_{\delta}^{\pi} u^{\alpha-2}\,du
 \le \frac{1}{1-\alpha}\,\delta^{\alpha-1}, \notag\\[0.3em]
\int_{I_2}
\frac{1}
     {(t-\theta)(t+\tfrac{\pi}{n}-\theta)}\,dt
&\le \frac{n}{\pi}
 \log\!\left(1+\frac{\pi}{n\delta}\right).
\end{align}
Combining these bounds with \eqref{eq:kernel_split_estimate}, we find that
there exists a constant $C_4>0$, independent of $\theta$, such that
\begin{equation}
\int_{I_2}
\left|
\Im\!\left(
\kappa_\theta'(t+\tfrac{\pi}{n})
\Bigl(\frac{1}{\kappa_\theta(t)}-\frac{1}{\kappa_\theta(t+\tfrac{\pi}{n})}\Bigr)
\right)
\right|dt
\le
C_4\!\left(
\frac{\delta^\alpha}{n\delta}
+\frac{\delta^\alpha}{(n\delta)^2}
+\frac{\log\!\left(1+\frac{\pi}{n\delta}\right)}{n^\alpha}
\right).
\label{eq:far_away_bnd_2}
\end{equation}
Choosing $\delta=\pi/n$, and combining
\eqref{eq:new_eq}, \eqref{eq:far_away_bnd_1} and \eqref{eq:far_away_bnd_2}, we obtain
\begin{equation}
\frac{1}{\pi}\int_{I_2}|h_\theta(t)-h_\theta(t+\tfrac\pi n)|dt
=O(n^{-\alpha}\log n).
\label{eq:far_away_bnd_full}
\end{equation}
The same argument applies verbatim to the integration over $I_3$.
Combining \eqref{eq:faber_first_bnd}, \eqref{eq:far_away_bnd_0},
and \eqref{eq:far_away_bnd_full}, with $\delta=\pi/n$, completes the proof
of Theorem~\ref{thm:suetin}.
\end{proof}

%	\begin{figure}
%\centering
%\includegraphics[width=0.3\textwidth]{Figures/widom.png}
%\caption{Harold Widom (1932-2021). Photo by Renate Schmid.}
%\label{fig:widom}
%\end{figure}
	
Faber's construction was extended to analyze orthogonal and Chebyshev polynomials relative to unions of Jordan curves and arcs of class \(C^{2+\alpha}\) in a very influential paper from~1969, see~\cite{widom69}. 
A fundamental difficulty in this setting is that there is no single-valued exterior conformal map onto the exterior of the unit disk. 
Even after introducing a suitable multivalued analytic analogue of the conformal map, the lower bound \eqref{eq:conformal_map_chebyshev_lower_bound} must be reformulated appropriately.
Widom constructs a multivalued analytic function on the unbounded complement \(\Omega_{\sfE}\) that solves a certain extremal problem; we denote this function by \(f_n\), see e.g. \cite{christiansen-simon-zinchenko22}. 
With this function he establishes the improved lower bound
\[
\|f_n\|_{\sfE}\,\Phi'(\infty)^{-n} \le t_n(\sfE).
\]
More importantly, Widom constructs a family of modified Faber polynomials \(Q_n\) satisfying
\begin{equation}
	t_n(\sfE) \le \|Q_n\|_{\sfE} 
	\sim 
	\|f_n\|_{\sfE}\,\Phi'(\infty)^{-n},
	\qquad n\to\infty,
	\label{eq:widom_convergence_equation}
\end{equation}
see \cite[Theorem~8.3]{widom69}. 
As a consequence, he obtains the asymptotic relation
\[
\Phi'(\infty)^{n} t_n(\sfE) 
\sim 
\|f_n\|_{\sfE},
\qquad n\to\infty.
\]

%Widom introduces polynomials of the form
%	\[Q_n(z) \coloneqq\frac{1}{2\pi i}\int_{\sfE}\frac{\Phi(\zeta)^nf_n(\zeta)}{\zeta-z}d\zeta\]
%	where $\Phi$ is a generalization of the exterior conformal map of $\sfE$ and $f_n$ solves a minimal problem associated with the set $\sfE$. If $\sfE$ has one component then $f_n = 1$ but otherwise~$\|f_n\|_{\sfE}\geq 1$. Widom then shows, in a similar fashion to how we showed that \eqref{eq:conformal_map_chebyshev_lower_bound} holds, that
%	\[\|f_n\|_{\sfE}\Phi'(\infty)^{-n}\leq t_n(\sfE)\]
%	but more importantly, he also proves that
%	\begin{equation}
%		t_n(\sfE)\leq \|\Phi'(\infty)^{-n}Q_n\|_{\sfE}\sim \|f_n\|_{\sfE}\Phi'(\infty)^{-n}
%		\label{eq:widom_convergence_equation}
%	\end{equation}
%	as $n\rightarrow \infty$, see \cite[Theorem 8.3]{widom69}. As a consequence he concludes \[t_n(\sfE)\sim \|f_n\|_{\sfE}\Phi'(\infty)^{-n},\quad n\rightarrow \infty.\]
	
	This result also extends to the case of weight functions on the boundary. In order to describe one way to handle weights we limit ourselves to the case where the set has one component. The reason for this limitation is to avoid having to deal with multivalued analytic functions. Assume that $\sfE$ is a Jordan curve, and $w:\sfE\rightarrow (0,\infty)$ is a (positive) continuous function. The Dirichlet problem on the unbounded complement $\Omega_{\sfE}$ with boundary data~$\log w$ on~$\sfE$ has a unique solution, see e.g. \cite[Corollary 4.1.8]{ransford95}. We conclude that there exists a harmonic function $\omega$ on $\Omega_{\sfE}$ such that $\omega(z)\rightarrow \log w(\zeta)$ as $z\rightarrow \zeta\in \sfE$. Let $\tilde{\omega}$ denote the harmonic conjugate which vanishes at infinity. Then the \emph{Szeg\H{o} function}
	\begin{equation}
		W(z) = \exp\Big(\omega(z)+i\tilde\omega(z)\Big)
		\label{eq:widom_R_function}
	\end{equation}
	satisfies $W(\infty)>0$ and $|W(z)| = w(z)$ on $\sfE$. If $\Phi:\Omega_{\sfE}\rightarrow \overline{\C}\setminus \overline{\D}$ denotes the exterior conformal map and $Q$ is any monic polynomial of degree $n$ then by the maximum principle
	\[\max_{\zeta\in  \sfE}w(\zeta)\left|Q(\zeta)\right| = \max_{\zeta\in \sfE}\left|\frac{Q(\zeta)W(\zeta)}{\Phi(\zeta)^n}\right|\geq\lim_{z\rightarrow \infty}\left|\frac{Q(z)W(z)}{\Phi(z)^n}\right| = \Phi'(\infty)^{-n}W(\infty).  \]
	Therefore \begin{equation}t_n(\sfE,w)\Phi'(\infty)^{n}W(\infty)^{-1}\geq 1.\label{eq:szego_weight_lb}\end{equation} On the other hand Widom, in \cite[Theorem 8.3]{widom69}, shows the following. Note that we have restricted ourselves to the (much simpler) case of a single Jordan curve.
	
	\begin{theorem}[Widom (1969) \cite{widom69}]
		Let $\sfE$ be a Jordan curve of class $C^{1+\alpha}$ for some \(0<\alpha <1\) and let $w:\sfE\rightarrow [0,\infty)$ be an upper semicontinuous weight function such that
		\begin{equation}
			\int_{ \sfE}\log w(z)|dz|>-\infty.
			\label{eq:szego_condition}
		\end{equation}
		Then as $n\rightarrow \infty$
		\begin{equation}
			t_n(\sfE,w) \Phi'(\infty)^{n}W(\infty)^{-1}\sim 1
			\label{eq:widom_weighted_asymptotics}
		\end{equation} 
		and
		\begin{equation}
			T_n^{\sfE,w}(z) = \Phi'(\infty)^{-n}\Phi(z)^nW(\infty)\Big(W(z)^{-1}+o(1)\Big)
			\label{eq:szego_widom_weight}
		\end{equation}
		holds uniformly on closed subsets of $\Omega_{\sfE}$. 
		\label{thm:widom_weight}
	\end{theorem}
	\begin{proof}
		Let us give a proof in the case where \(w:\sfE\to[0,\infty)\) is continuous and positive. In this case, the function \(W(z)\) defined in \eqref{eq:widom_R_function} is continuous on \(\overline \Omega_{\sfE}\). Let us begin by temporarily fixing \(R>1\), and defining the new weight function
		\[w_R(z) \coloneqq \left|W\bigl(g_R(\Phi(z))\bigr)\right|,\quad z\in \sfE,\qquad g_R(s) = \Psi(Rs).\]
		The associated Szeg\H{o} function of \(w_R\) defined analogously as in \eqref{eq:widom_R_function} is precisely \(W(g_R(\Phi(z)))\)
		and therefore \eqref{eq:szego_weight_lb} implies that
		\[t_n(\sfE,w_R)\Phi'(\infty)^n\geq W(\infty).\]
 		The Laurent expansion
		\[\frac{1}{W(g_R(s) )}=W(\infty)^{-1}+\sum_{k=1}^{\infty}a_{k}s^{-k}\]
		converges uniformly on closed subsets of \(|s|>R^{-1}.\) Let \[S_m(s) \coloneqq W(\infty)^{-1}+\sum_{k=1}^{m}a_ks^{-k}.\] Given \(\epsilon>0\), we can choose \(m\) large enough so that
		\[\|1/w_R-S_m\circ \Phi\|_{\sfE} <\epsilon.\]
		If \(n>m\) then
		\[Q_{n,m}(z) \coloneqq W(\infty)^{-1}F_n^{\sfE}(z)+\sum_{k=1}^{m}a_kF_{n-k}^{\sfE}(z)\]
		has leading coefficient \(W(\infty)^{-1}\Phi'(\infty)^n\), and furthermore, there exists a constant \(C>0\) such that if \(z\in \sfE\)
		\begin{align*}
		|Q_{n,m}(z)|& \leq \left|W(\infty)^{-1}\Phi(z)^n+\sum_{k=1}^{m}a_k \Phi(z)^{n-k}\right|+C\cdot \frac{\log n}{n^\alpha} \\
		& \leq \left|S_m(\Phi(z))\right|+C\cdot \frac{\log n}{n^\alpha} \\
		& \leq \frac{1}{w_R(z)}+C\cdot \frac{\log n}{n^\alpha}+\epsilon.
		\end{align*}
		This implies that
		\begin{align*}
			\|w_RQ_{n,m}\|_{\sfE}&\leq 1+O\left(\frac{\log n}{n^\alpha}\right)+\epsilon\max_{z\in \sfE}|w_R(z)|.
		\end{align*}
		From the fact that
		\[W(\infty)^{-1}\Phi'(\infty)^nt_n(\sfE,w_R)\leq \|w_RQ_{n,m}\|_{\sfE}\]
		we obtain by letting \(n\to \infty\) that
		\begin{equation}
			\lim_{n\to\infty}W(\infty)^{-1}\Phi'(\infty)^nt_n(\sfE,w_R) = 1.
			\label{eq:widom_scaled_weight}
		\end{equation}
		This proves the theorem for the weight \(w_R\). To extend this to \(w\) we note that the function~\(W\) is continuous on \(|z|\geq 1\) and therefore 
		\[w_R\to w\] uniformly on \(\sfE\). Let again \(0<\epsilon<1\) be given. We can choose \(\delta>0\) small enough so that if~\(1<R<1+\delta\), then
		\[w_R(z)(1-\epsilon)\leq w(z)\leq w_R(z)(1+\epsilon),\quad z\in \sfE.\]
		But this has the effect that
		\begin{align*}
			t_n(\sfE,w_R)\leq \|w_RT_n^{\sfE,w}\|_{\sfE}\leq (1-\epsilon)t_n(\sfE,w)\leq (1-\epsilon)\|wT_n^{\sfE,w_R}\|_{\sfE}\leq (1-\epsilon^2)t_n(\sfE,w_R).
		\end{align*}
		Multiplying both sides by \(W(\infty)^{-1}\Phi'(\infty)^n/(1-\epsilon)\) and using \eqref{eq:widom_scaled_weight} we obtain that
		\begin{align*}
			(1-\epsilon)^{-1} &=(1-\epsilon)^{-1}\lim_{n\to\infty}W(\infty)^{-1}\Phi'(\infty)^nt_n(\sfE,w_R)\\&\leq \liminf_{n\to \infty}W(\infty)^{-1}\Phi'(\infty)^nt_n(\sfE,w) \\
			& \leq \limsup_{n\to \infty}W(\infty)^{-1}\Phi'(\infty)^nt_n(\sfE,w)\\&\leq (1+\epsilon)\lim_{n\to\infty}W(\infty)^{-1}\Phi'(\infty)^nt_n(\sfE,w_R) = 1+\epsilon.
		\end{align*}
		Since \(0<\epsilon<1\) was arbitrary this proves \eqref{eq:widom_weighted_asymptotics} for positive, continuous weight functions. 
		
		The result can now be lifted to the case of an upper semicontinuous weight function satisfying \eqref{eq:szego_condition} using standard arguments, see e.g. \cite{widom69}.
		
		To prove \eqref{eq:szego_widom_weight} we can argue analogously as in Theorem \ref{thm:faber_asymptotics}.
	\end{proof}

	Let us stress the fact that Widoms considerations are much deeper, allowing for several components of \(\sfE\). 
		
	A natural question is how this picture changes when the boundary regularity is relaxed.
If$\sfE$ is a piecewise smooth Jordan curve and $\Psi$ denotes the exterior conformal map,
then the function
\[
v_\theta(t)\coloneqq \lim_{R\to 1}\arg\bigl(\Psi(Re^{it})-\Psi(e^{i\theta})\bigr),
\qquad t\neq \theta \pmod{2\pi},
\]
defined in \eqref{eq:def_v_theta} exhibits a jump discontinuity at $t=\theta$.
The magnitude of this jump is equal to the exterior angle of $\sfE$ at the point
$\Psi(e^{i\theta})$.
	
	As shown by Pommerenke in \cite{pommerenke64}, if \(\sfE\) is a piecewise smooth Jordan curve with convex interior, then
	\[\lim_{n\to\infty}\max_{z\in \sfE}|F_n^{\sfE}(z)| = \text{the largest exterior angle of \(\sfE\)}.\]
	This implies that it is easy to come upp with examples of sets \(\sfE\) where
	\[\lim_{n\to\infty}\max_{z\in \sfE}|F_n^{\sfE}(z)|>1,\]
	for instance a square would have this property. In general, if $\sfE$ is a convex compact set then~\cite[Theorem 2]{kovari-pommerenke67} implies that
	\begin{equation}
		|F_n^{\sfE}\circ \Phi^{-1}(w)-w^n|\leq 1
		\label{eq:faber_kovari_pommerenke_convex}
	\end{equation}
	if $|w|\geq 1$. Furthermore this inequality is strict unless \(\sfE\) is a straight line segment. From this we may conclude the following.
	\begin{theorem}[K\"{o}vari \& Pommerenke (1967) \cite{kovari-pommerenke67}]
		Let $\sfE$ denote a compact convex set then
		\[\Phi'(\infty)^nt_n(\sfE)\leq 2.\]
		\label{thm:convex_set_norm_bounds}
		Furthermore, the inequality is strict unless \(\sfE\) is a straight line segment.
	\end{theorem}
	\begin{proof}
If \(\sfE\) is not a straight-line segment, then \(\partial \sfE\) is a Jordan curve. In that case, Carath\'{e}odory's theorem implies that \(\Phi\) extends to a homeomorphism from \(\partial \sfE\) to \(\bbT\). Together with \eqref{eq:faber_kovari_pommerenke_convex} this implies that 
	\[t_n(\sfE)\Phi'(\infty)^n\leq \max_{z\in \sfE}|F_n^{\sfE}(z)|\leq \max_{z\in \sfE}|F_n^{\sfE}(z)-\Phi(z)^n|+\max_{z\in \sfE}|\Phi(z)|^n< 2.\]
	If \(\sfE\) is a straight line segment, then we already know from \eqref{eq:straight_line_widom} that the theorem holds.
	\end{proof}
	By combining results from \cite{pritsker99, pritsker02} we further see that if \(\sfE\) is a piecewise Dini-smooth Jordan curve, then
	\begin{equation}
		\text{the largest exterior angle of \(\sfE\)}\leq \liminf_{n \to \infty}\max_{z\in \sfE}|F_n^{\sfE}(z)|\leq \limsup_{n\to \infty}\max_{z\in \sfE}|F_n^{\sfE}(z)|\leq 2.
		\label{eq:pritsker_inequality}
	\end{equation}
	A curve is Dini-smooth if any arc-length parametrization is continuously differentiable, having a modulus of continuity \(\omega(t)\) such that \(\omega(t)/t\) is integrable around the origin. In particular, any \(C^{1+\alpha}\) curve is also Dini-smooth. From \eqref{eq:pritsker_inequality} we gather that if \(\sfE\) is a piecewise Dini-smooth Jordan curve, then
	\[\limsup_{n\to\infty}t_n(\sfE)\Phi'(\infty)^n\leq 2,\]
	however, this bound is not sharp. Instead the following holds.
	
	\begin{theorem}[\cite{minadiaz-rubin-wennman25}]
		Let \(\sfE\) denote a piecewise Dini-smooth Jordan curve, then
		\[\lim_{n\to\infty}t_n(\sfE)\Phi'(\infty)^n=1.\]
		\label{thm:wennman}
	\end{theorem}
	\vspace{-.5cm}
	Since the Faber polynomials prove insufficient in saturating the lower bound in this case the proof instead relies on summing together Faber polynomials of different degrees, essentially weighted Faber polynomials.

It is natural to ask which regularity assumptions on a Jordan curve \(\sfE\) are sufficient to ensure that \eqref{eq:conformal_map_chebyshev_lower_bound} is asymptotically sharp, in the sense that
\[
\lim_{n\to\infty} t_n(\sfE)\Phi'(\infty)^n = 1.
\]
That some degree of boundary regularity is indeed necessary follows from the existence of quadratic Julia sets \(\sfE\) which are Jordan curves but for which
\[
\limsup_{n\to\infty} t_n(\sfE)\Phi'(\infty)^n > 1.
\]
This phenomenon can be verified by explicitly identifying the Chebyshev polynomials along a suitable subsequence of degrees; see, for example, \cite{barnsley-geronimo-harrington83, kamo-borodin94, stawiska96}.

It appears to remain open whether the stronger condition
\[
\liminf_{n\to\infty} t_n(\sfE)\Phi'(\infty)^n > 1
\]
holds for these examples, or whether there exists a Jordan curve for which this strict inequality is valid.

%	It is interesting to determine, which regularity conditions need to be imposed on a Jordan curve \(\sfE\) in order to ensure that \eqref{eq:conformal_map_chebyshev_lower_bound} is asymptotically saturated in the sense that
%	\[\lim_{n\to\infty}t_n(\sfE)\Phi'(\infty)^n=1.\]
%	That some regularity is required follows from the fact that there are examples of quadratic Julia sets \(\sfE\) which are Jordan curves for which
%	\[\limsup_{n\to\infty}t_n(\sfE)\Phi'(\infty)^n>1.\]
%	This is shown by explicitly representing the corresponding Chebyshev polynomials for a subsequence of degrees, see e.g. \cite{barnsley-geronimo-harrington83, kamo-borodin94, stawiska96}. It seems to be unknown whether the stronger statement
%	\[\liminf_{n\to\infty}t_n(\sfE)\Phi'(\infty)^n>1\]
%	holds in these cases or whether a Jordan curve with this property exists.
\subsection{Chebyshev polynomials relative to a Jordan arc}
\label{subs:cheb_arc}
After establishing what can be proved for unions of Jordan curves of class
$C^{2+\alpha}$ and their associated weights, Widom turned to the subtler case
of Jordan arcs. We describe what is known for Chebyshev polynomials corresponding to a single arc, whereas Widom's goal was to describe Chebyshev polynomials corresponding to a union of several disjoint arcs.

Let $\sfE$ be a Jordan arc of class $C^{2+\alpha}$, $0<\alpha<1$, with complement
$\Omega_{\sfE}$, and let
$\Phi:\Omega_{\sfE}\to\RS\setminus\overline{\D}$
be the exterior conformal map normalized by
\eqref{eq:canonical_conformal_map}.
For $\zeta\in\sfE$, the non-tangential limit
\[
\lim_{\substack{z\to\zeta\\ z\in\Omega_{\sfE}}}\Phi(z)
\]
depends on the side of the arc from which $z$ approaches $\zeta$.
We denote the two boundary values by $\Phi_+$ and $\Phi_-$.
In \cite{widom69}, Widom proved that the Faber polynomials satisfy
\begin{equation}
	F_n^{\sfE}(z)
	=
	\Phi_+(z)^n+\Phi_-(z)^n+o(1),
	\qquad n\to\infty,\quad z\in\sfE.
	\label{eq:faber_arc_asymptotics}
\end{equation}
Since $\Phi_\pm(z)\in\bbT$ for $z\in\sfE$, we obtain that
\[
\limsup_{n\to\infty}\|F_n^{\sfE}\|_{\sfE}\leq 2.
\]
Together with \eqref{eq:faber_upper_bnd_cheb}, this yields
\begin{equation}
	\limsup_{n\to\infty} t_n(\sfE)\,\Phi'(\infty)^n \le 2.
	\label{eq:chebyshev_upper_bnd_arc}
\end{equation}

If $\sfE=[-1,1]$, then the exterior conformal map is
\[
\Phi(z)=z+\sqrt{z^2-1},
\qquad
\Phi'(\infty)=2.
\]
Using \eqref{eq:chebyshev_first_kind}, we obtain
\begin{align*}
	\Phi(z)^n+\Phi(z)^{-n}
	&=
	\bigl(z+\sqrt{z^2-1}\bigr)^n
	+
	\bigl(z-\sqrt{z^2-1}\bigr)^n \\
	&=
	\sum_{k=0}^{n}\binom{n}{k}(1+(-1)^k)
	z^{n-k}(z^2-1)^{k/2} \\
	&=
	2\sum_{k=0}^{\lfloor n/2\rfloor}
	\binom{n}{2k}z^{n-2k}(z^2-1)^k
	=
	2^n T_n^{[-1,1]}(z).
\end{align*}
Since $F_n^{[-1,1]}$ coincides with the polynomial part of $\Phi(z)^n$, and
$\Phi(z)^{-n}=O(z^{-n})$ as $z\to\infty$, we gather that
\begin{equation}
	T_n^{[-1,1]}(x)
	=
	2^{-n}
	\left(
	\bigl(x+\sqrt{x^2-1}\bigr)^n
	+
	\bigl(x-\sqrt{x^2-1}\bigr)^n
	\right)
	=
	2^{-n}F_n^{[-1,1]}(x).
	\label{eq:faber_cheb_equal}
\end{equation}
Consequently,
\begin{equation}
	\Phi'(\infty)^n t_n([-1,1])
	=
	2
	=
	\|F_n^{[-1,1]}\|_{[-1,1]},
	\label{eq:straight_line_widom}
\end{equation}
and the bound \eqref{eq:chebyshev_upper_bnd_arc} is attained for an interval.

Motivated by this example, Widom conjectured that
\[
\lim_{n\to\infty} t_n(\sfE)\,\Phi'(\infty)^n = 2
\]
for every $C^{2+\alpha}$ Jordan arc $\sfE$.
More generally, he suggested that the asymptotic formula
\eqref{eq:widom_convergence_equation} should remain valid for arcs, provided
the right-hand side is multiplied by $2$.
This conjecture holds when $\sfE$ is a finite union of closed intervals,
see \cite[Theorem~11.5]{widom69}, but fails in general. A counterexample was given by Thiran and Detaille \cite{thiran-detaille91}.
Let
\[
\sfE(\alpha)
:=
\{z\in\C:|\arg z|\le\alpha\},
\qquad 0<\alpha<\pi,
\]
and let $\Phi$ be the associated exterior conformal map.
Then
\begin{equation}
	\lim_{n\to\infty}
	t_n(\sfE(\alpha))\,\Phi'(\infty)^n
	=
	2\cos^2(\alpha/4).
	\label{eq:thiran_detaille}
\end{equation}
Apart from straight line segments, this is the only case where the asymptotic
behavior of Chebyshev polynomials on a Jordan arc is fully understood.
For further results on circular arcs, see
\cite{eichinger17, schiefermayr19, schiefermayr-zinchenko21,
schiefermayr-zinchenko22}.
Connections with orthogonal polynomials were studied in
\cite{alpan-zinchenko20}, and the weighted analogue of
\eqref{eq:thiran_detaille} appears in
\cite{christiansen-eichinger-rubin-zinchenko26}, where residual polynomials are also described.
The asymptotical formula \eqref{eq:thiran_detaille} can be derived from Theorem~\ref{thm:achieser_disjoint_intervals} as is shown in
\cite{schiefermayr19}.

From \eqref{eq:thiran_detaille} we observe that any number in $(1,2]$ can arise as a limit point of
\[
t_n(\sfE)\,\Phi'(\infty)^n
\]
for $C^{2+\alpha}$ Jordan arcs.
Using weighted Faber polynomials and an analogue of
\eqref{eq:faber_arc_asymptotics}, Alpan \cite{alpan22} constructed a sequence
$\{Q_n\}$ with leading coefficient $\Phi'(\infty)^n$ such that
\[
\limsup_{n\to\infty}\|Q_n\|_{\sfE}<2,
\]
unless $\sfE$ satisfies the potential-theoretic \emph{S-property}.
For a single arc, this property holds only for straight line segments;
see \cite{christiansen-eichinger-rubin23} and
\cite{stahl85-1, stahl85-2, stahl12}.
Combining these observations with Totik's result \cite{totik14} yields the
following.

\begin{theorem}
	Let $\sfE$ be a $C^{2+\alpha}$ Jordan arc with exterior conformal map $\Phi$
	normalized by \eqref{eq:canonical_conformal_map}.
	Then
	\[
	1
	<
	\liminf_{n\to\infty} t_n(\sfE)\,\Phi'(\infty)^n
	\le
	\limsup_{n\to\infty} t_n(\sfE)\,\Phi'(\infty)^n
	\le
	2.
	\]
	The upper bound is strict unless $\sfE$ is a straight line segment.
	\label{thm:cheb_arc}
\end{theorem}

Thus Widom's conjecture fails for Jordan arcs except in the case of an
interval; see also~\cite{totik-yuditskii15}.
Moreover, Totik's theorem \cite{totik14} applies more generally to compact
sets whose outer boundary contains arc components.

The exact limiting behavior of $t_n(\sfE)\,\Phi'(\infty)^n$ is unknown in general.
A conjectural value proposed in \cite{christiansen-simon-zinchenko22}
agrees with the known limits for circular arcs and straight line segments.
	
\subsection{Potential theory and Widom factors}
\label{subseq:pot_theory}
We will now show that the lower bound \eqref{eq:conformal_map_chebyshev_lower_bound} extends to any compact set. This requires a generalization of the quantity $\Phi'(\infty)$ since the unbounded complement need not be simply connected. We will use several concepts from potential theory and refer the reader to \cite{ransford95} for details, see also \cite{ahlfors73, garnett-marschall05}.

Throughout this section let $\sfE$ denote a compact subset of $\C$ and recall the notation $\Omega_\sfE$ for the unbounded component of $\overline{\C}\setminus \sfE$ and $\cM(\sfE)$ for the space of probability measures which have support contained in $\sfE$. Given $\mu\in \cM(\sfE)$ we define the potential function $U^\mu$ via the formula
\[U^\mu(z) = \int_{\sfE}\log \frac{1}{|z-\zeta|}d\mu(\zeta)\] 
and the related energy functional as
\[\cE(\mu) = \int_{\sfE}U^\mu(z)d\mu(z).\]
A set $\sfE$ is called \emph{polar} if
\[
	\inf_{\mu\in \cM(\sfE)}\cE(\mu) = \infty.
\]
If a relation holds outside a polar set then it is said to hold \emph{quasi-everywhere} or q.e. for short. For any non-polar set, there exists a unique measure $\mu_{\sfE}\in \cM(\sfE)$ such that
\[\cE(\mu_{\sfE})=\inf_{\mu\in \cM(\sfE)}\cE(\mu),\]
see \cite[Theorems 3.3.2, 3.7.6]{ransford95}. This measure is called the \emph{equilibrium measure} relative to $\sfE$ and is supported on the outer boundary of $\sfE$. Two important examples are
\begin{equation}
	d\mu_{\overline\D} = \left.\frac{d\theta}{2\pi}\right\vert_{|z| = 1},\quad \text{and}\quad d\mu_{[-1,1]}(x) = \frac{1}{\pi}\frac{1}{\sqrt{1-x^2}}dx.
\label{eq:equilibrium_measure_representation}
\end{equation}
	
	\emph{The logarithmic capacity} of a set is defined through the formula
	\[\capacity(\sfE) = \sup_{\mu\in \cM(\sfE)} e^{-\cE(\mu)}.\]
	We see that sets of capacity \(0\) are precisely those that are polar. If $\sfE_1$ and $\sfE_2$ denote two compact sets with associated unbounded components $\Omega_{\sfE_1}$ and $\Omega_{\sfE_2}$ such that there exists a conformal map 
	\[\Phi:\Omega_{\sfE_1}\rightarrow \Omega_{\sfE_2}\]
	with $\Phi(z) = \alpha z+O(1)$ as $z\rightarrow\infty$, $\alpha\neq 0$, then \cite[Theorem 5.2.3]{ransford95}
	\begin{equation}
		|\alpha|\capacity(\sfE_1) = \capacity(\sfE_2).
		\label{eq:capacity_conformal_map}
	\end{equation}
	
One way to determine capacities is given in the following construction. Let \(\Omega\) denote an open subset of \(\RS\),  such that \(\partial \Omega\) is compact in \(\C\) and non-polar, then there exists a unique function \(G_{\Omega}:\Omega\times \Omega\to (0,\infty]\) with the properties
	\begin{itemize}
		\item $z\mapsto G_{\Omega}(z,w)$ is harmonic and bounded on closed subsets of $\Omega\setminus \{w\}$,
		\item As $z\rightarrow w\in \Omega$
	\begin{equation}
		G_{\Omega}(z,w) = \begin{cases}
		\log |z|+O(1), & w = \infty,\\
		-\log|z-w|+O(1), & w\neq \infty;
	\end{cases}
	\label{eq:greens_function_definition}
	\end{equation}
	\item $G_{\Omega}(z,w)\rightarrow 0$ as $z\rightarrow\zeta$ for quasi-every \(\zeta\in \partial \Omega\).
	\end{itemize}

 	The function $z\mapsto G_{\Omega}(z,w)$ is called \emph{Green's function} of $\Omega$ with a pole at $w$. A point $\zeta$ on the boundary $\partial\Omega$ is said to be a regular boundary point if
 	\[\lim_{z\rightarrow \zeta}G_{\Omega}(z,w)=0\]
 	for $w\in \Omega$. If the entire boundary consists of regular points then we call the boundary a regular set. Throughout this section we will mostly be concerned with Green's function of~\(\Omega_{\sfE}\) with a pole at $\infty$ which we simply denote with $G_{\sfE}(z)\coloneqq G_{\Omega_{\sfE}}(z,\infty)$. The asymptotical behavior at $\infty$ is explicit in terms of capacity. Indeed \cite[Theorem 5.2.1]{ransford95} implies that
	\begin{equation}
		G_{\sfE}(z) = \log|z|-\log \capacity(\sfE)+o(1),\quad \text{ as }z\rightarrow \infty.
		\label{eq:greens_function_at_infinity}
	\end{equation}
	If $Q$ is a polynomial with leading term \(az^m\) and $\sfE$ is a compact set with regular outer boundary, then
	\begin{equation}
		G_{Q^{-1}(\sfE)}(z) = \frac{1}{m}G_{\sfE}(Q(z))
		\label{eq:green_potential_polynomial_preimage}
	\end{equation}
	It follows from \eqref{eq:green_potential_polynomial_preimage} that if \(\sfE\) is any compact set, then  
	\begin{equation}
		\capacity(Q^{-1}(\sfE)) = \left(\frac{\capacity(\sfE)}{|a|}\right)^{1/m},
		\label{eq:capacity_preimage}
	\end{equation}
	see also \cite[Theorem 5.2.3]{ransford95}.
		We further have from \cite[Theorem 4.4.2]{ransford95} that
	\begin{equation}
		G_{\sfE}(z) = \cE(\mu_{\sfE})-U^{\mu_{\sfE}}(z).
		\label{eq:greens_function_from_potential}
	\end{equation}
	By combining this with \eqref{eq:equilibrium_measure_representation} we can prove Lemma \ref{lem:log_integral_equilibrium_measure} which we have used in the proof of Theorem \ref{thm:markov_weighted_chebyshev}.
	\begin{proof}[Proof of Lemma \ref{lem:log_integral_equilibrium_measure}]
			Using \eqref{eq:equilibrium_measure_representation} we recognize the relation
			\[-U^{\mu_{[-1,1]}}(z) = \frac{1}{\pi}\int_{-1}^{1}\frac{\log|x-z|}{\sqrt{1-x^2}}dx.\]
			It is easy to verify that 
			\[G_{[-1,1]}(z) = \log|z+\sqrt{z^2-1}|\]
			from the characterizing properties of Green's function. 
			Since \[G_{[-1,1]}(z) = \log |z|+\log 2+o(1),\quad z\to\infty\] we conclude from \eqref{eq:greens_function_at_infinity} that $\cE([-1,1]) = \log 2$ and if $z\notin [-1,1]$ we gather from \eqref{eq:greens_function_from_potential} that
			\[\frac{1}{\pi}\int_{-1}^{1}\frac{\log|x-z|}{\sqrt{1-x^2}}dx = \log|z+\sqrt{z^2-1}|-\log 2.\]
			This proves the relation if $z\notin [-1,1]$.
			
			Since both sides of \eqref{eq:greens_function_from_potential} are lower semicontinuous on $\C$ and agree outside a set with area measure zero they coincide on all of $\C$, see \cite[Theorem 2.7.5]{ransford95}.
		\end{proof}

	The relation in \eqref{eq:capacity_preimage} can be used to get lower bounds for Chebyshev polynomials. Extending Faber's lower bound \eqref{eq:conformal_map_chebyshev_lower_bound} is the so-called Szeg\H{o} inequality from \cite{szego24}.
%\begin{figure}
%\centering
%\includegraphics[width=0.3\textwidth]{Figures/szego.jpg}
%\caption{G\'{a}bor Szeg\H{o} (1895-1985).}
%\label{fig:szego}
%\end{figure}
	\begin{theorem}[Szeg\H{o} (1924) \cite{szego24}]
		Let $\sfE\subset \C$ denote a compact set then
		\begin{equation}
			\capacity(\sfE)^n\leq t_n(\sfE).
			\label{eq:szego_inequality}
		\end{equation}
		\label{thm:szego_inequality}
	\end{theorem}
	\begin{proof}
		It is clear that the lemniscatic set
		\[\sfE_n \coloneqq \{z: |T_n^{\sfE}(z)|\leq t_n(\sfE)\}\]
		contains $\sfE$. Since $\capacity$ increases under set inclusions we gather that
		\[\capacity(\sfE)\leq \capacity(\sfE_n).\]
		Recalling that the capacity and radius of a disk coincide the result follows from \eqref{eq:capacity_preimage} since 
		\[\capacity(\sfE_n) = t_n(\sfE)^{1/n}.\qedhere\]
	\end{proof}

	Szeg\H{o}'s inequality, Theorem \ref{thm:szego_inequality}, can be related to \eqref{eq:conformal_map_chebyshev_lower_bound}. In the case that $\Omega_{\sfE}$ is simply connected and $\Phi:\Omega_{\sfE}\rightarrow\overline{\C}\setminus \overline{\D}$ is the conformal map satisfying \eqref{eq:canonical_conformal_map}, then
	\begin{equation}
		\Phi'(\infty)\capacity(\sfE) = 1.
		\label{eq:conformal_der_capacity}
	\end{equation}
	To see this, note that 
	\begin{equation*}
		G_{\sfE}(z) = \log|\Phi(z)| = \log|z|+\log|\Phi'(\infty)|+o(1), \quad z\to \infty
	\end{equation*}
	and by referring to the defining properties of the Green's function we find that
	\begin{equation}
		G_{\sfE}(z) = \log |\Phi(z)|,
		\label{eq:green_func_conformal_map}
	\end{equation} which together with \eqref{eq:greens_function_at_infinity} additionally proves \eqref{eq:conformal_der_capacity}. As a further consequence we see that the sets \(\sfE(r)\) defined previously, have the representation
	\[\sfE(r) = \{z: |\Phi(z)| = r\} = \{z: G_{\sfE}(z) = \log r\},\quad r>1.\]
	The curves $\sfE(r)$ are therefore called the Green lines or equipotential curves corresponding to the set $\sfE$.
	
	Szeg\H{o} actually proved another connection between $\capacity(\sfE)$ and $t_n(\sfE)$. Related ideas had previously appeared in \cite{faber20, fekete23}.
	
%	\begin{theorem}[Faber, Fekete and Szeg\H{o} \cite{faber20,fekete23,szego24}]
%		Let $\sfE\subset \C$ be a compact set, then
%		\[\lim_{n\rightarrow \infty}t_n(\sfE)^{1/n} = \capacity(\sfE).\]
%		\label{thm:ffs-theorem}
%	\end{theorem}
%	This limits how fast $t_n(\sfE)$ can grow and a central point in the understanding of Chebyshev polynomials concerns which bounds can be placed upon the quantity
%	\begin{equation}
%		\cW_n(\sfE) = \frac{t_n(\sfE)}{\capacity(\sfE)^n},
%		\label{eq:definition_widom_factor}
%	\end{equation}
%	the so-called Widom factor of degree $n$ corresponding to $\sfE$. 
\begin{theorem}[Faber, Fekete, \& Szeg\H{o} {\cite{faber20,fekete23,szego24}}]
Let $\sfE \subset \C$ be a compact set. Then
\[
\lim_{n\to\infty} t_n(\sfE)^{1/n} = \capacity(\sfE).
\]
\label{thm:ffs-theorem}
\end{theorem}
\vspace{-.5cm}
Theorem~\ref{thm:ffs-theorem} implies that the \emph{Widom factor} of degree $n$ associated with $\sfE$ defined by
%the Chebyshev constants $t_n(\sfE)$ can grow at most subexponentially. 
%A central problem in the theory is therefore to understand the finer asymptotic behavior of
%$t_n(\sfE)$, which is conveniently encoded in the normalized quantity
\begin{equation}
	\cW_n(\sfE) := \frac{t_n(\sfE)}{\capacity(\sfE)^n},
	\label{eq:definition_widom_factor}
\end{equation}
can grow at most subexponentially. On the other hand, any subexponential growth rate of
$\{\cW_n(\sfE)\}$ is possible; see~\cite{goncharov-hatinoglu15}, where the terminology \emph{Widom factor} was introduced. It should be emphasized that the sets considered there are of Cantor type.

An easy example, showing that there is no universal upper bound to the Widom factors is furnished by the following. Consider the union of two disjoint intervals
\[
	\sfE(a) = [-1,-a]\cup[a,1], \qquad 0<a<1.
\]
By Theorem~\ref{thm:achieser_disjoint_intervals} together with
Theorem~\ref{thm:ffs-theorem}, we have
\[
	\capacity(\sfE(a)) = \frac{1}{2}\sqrt{1-a^2}.
\]
Consequently, using \eqref{eq:achieser_limit_pts}, we obtain
\[
	\cW_{2n}(\sfE(a)) = 2,
	\qquad
	\lim_{n\to\infty}\cW_{2n+1}(\sfE(a))
	= 2\cdot \sqrt{\frac{1+a}{1-a}}.
\]
As $a\to 1$, the latter quantity diverges, showing that Widom factors may be arbitrarily large. In this example the underlying set $\sfE(a)$ varies with the parameter $a$.

Theorem \ref{thm:szego_inequality} and Theorem \ref{thm:ffs-theorem} implies that
\begin{equation}
	\cW_n(\sfE)\geq 1,\qquad \cW_n(\sfE)^{1/n}\to 1,\quad n\to \infty.\label{eq:szego_and_ffs}
\end{equation}

	It is an important observation to note that Widom factors are invariant under dilations and translations. Therefore, the Widom factors do not depend on the size of the set, but rather on its topological, geometric, and potential-theoretic properties.

	\begin{theorem}
		Let $\sfE\subset \C$ denote a compact non-polar set. If $\alpha\in \C\setminus \{0\}$ and $b\in \C$ then
		\[\cW_n(\alpha \sfE+b) = \cW_n(\sfE).\]
		\label{thm:widom_translation_rotation_invariant}
	\end{theorem}
	\vspace{-1cm}
	\begin{proof}
		First of all \eqref{eq:definition_widom_factor} is well-defined since $\capacity{\sfE}>0$ by assumption. Furthermore from \eqref{eq:capacity_conformal_map} we gather that
		\[\capacity(\alpha\sfE+b) = |\alpha|\capacity(\sfE).\]
		On the other hand, by the uniqueness of the Chebyshev polynomial,
		\[\alpha^nT_n^{\alpha \sfE+b}\left(\frac{z-b}{\alpha}\right) = T_n^{\sfE}(z).\]
		Therefore $t_n(\alpha\sfE+b) = |\alpha|^nt_n(\sfE)$ and so we see that
		\[\cW_n(\alpha\sfE+b) = \frac{t_n(\alpha\sfE+b)}{\capacity(\alpha\sfE+b)^n} = \frac{|\alpha|^nt_n(\sfE)}{|\alpha|^n\capacity(\sfE)} = \cW_n(\sfE).\qedhere\]
	\end{proof}
	
If \(\Omega_{\sfE}\) is simply connected, then
\begin{equation}
	t_n(\sfE)\Phi'(\infty)^n = \cW_n(\sfE)
	\label{eq:conformal_map_widom_factor}
\end{equation}
and we saw previously in Theorems \ref{thm:faber_asymptotics}, \ref{thm:suetin}, and \ref{thm:wennman} that certain regularity conditions on the boundary of $\sfE$ guarantees that $\cW_n(\sfE)\to 1$ as \(n\to\infty\), which by \eqref{eq:szego_and_ffs} is the minimal possible limit. 

We saw in Theorem \ref{thm:chebyshev_on_lemnsicate} examples of sets where, at least for a subsequence, $\cW_n(\sfE) = 1$ holds. Indeed, assuming that $P$ is a monic polynomial of degree $m$ and $r>0$ then \eqref{eq:capacity_preimage} implies that with $\sfE(r) = \{z:|P(z)| = r^m\}$
\[\capacity(\sfE(r)) = \capacity(P^{-1}\{z:|z|\leq r^m\}) = r.\]
On the other hand, by Theorem \ref{thm:chebyshev_on_lemnsicate}, $T_{nm}^{\sfE(r)} = P(z)^n$ and therefore
\[\cW_{nm}(\sfE(r)) = \frac{t_{nm}(\sfE(r))}{\capacity(\sfE(r))^{nm}} = \frac{r^{nm}}{r^{nm}} = 1.\]
A natural question is whether there are other examples of sets where Szeg\H{o}'s lower bound is saturated at least for a subsequence. The answer turns out to be no. 

\begin{theorem}[Christiansen, Simon \& Zinchenko  (2020) \cite{christiansen-simon-zinchenko-III}]
	Let $\sfE\subset\C$ be a compact set with positive logarithmic capacity and unbounded complement $\Omega_{\sfE}$. Then $\cW_{n}(\sfE) = 1$ if and only if there is a polynomial $P$ of degree $n$ such that
	\[\partial \Omega_{\sfE} = \{z: |P(z)| = 1\}.\]
	\label{thm:csz_szego_lower_bound_equality}
\end{theorem}
\vspace{-1cm}
\begin{proof}
	Without loss of generality we may assume that \(\RS\setminus \Omega_{\sfE} = \sfE\) and by Theorem \ref{thm:widom_translation_rotation_invariant}  we may assume that \(\capacity(\sfE) =1\). Then \(|T_n^{\sfE}(z)|\leq 1\) for \(z\in \sfE\) and consequently
	\[\lim_{z\to \zeta}\left(\frac{1}{n}\log |T_n^{\sfE}(z)|-G_{\sfE}(z)\right)\leq 0,\qquad \text{q.e. }\zeta\in \partial \Omega_{\sfE}\]
	The maximum principle \cite[Theorem 3.6.9]{ransford95} implies that
	\[\frac{1}{n}\log |T_n^{\sfE}(z)|-G_{\sfE}(z)\leq 0,\quad z\in \Omega_{\sfE}.\]
		Since	
		\[\frac{1}{n}\log |T_n^{\sfE}(z)|-G_{\sfE}(z)\to 0,\quad z\to \infty.\]
	this maximum is attained at \(\infty\) and therefore \cite[Theorem 2.3.1]{ransford95} implies that
	\[\frac{1}{n}\log |T_n^{\sfE}(z)| = G_{\sfE}(z)>0,\quad z\in \Omega_{\sfE}.\]
	As a consequence
	\[\{z: |T_n^{\sfE}(z)|\leq 1\}\subset \sfE.\]
	The reverse inclusion is trivial. The result now follows from the fact that 
	\[ \{z: |T_n^{\sfE}(z)| = 1\}=\partial\{z: |T_n^{\sfE}(z)|\leq 1\}= \partial\sfE = \partial\Omega_{\sfE}.\qedhere\]
\end{proof}
As we see, sets saturating Szeg\H{o}'s lower bound are precisely lemniscates. If we again let~$\sfE(r) = \{z: |P(z)| = r^m\}$, where $\deg P =m$, then $\sfE(r)$ will be an analytic Jordan curve if $r>0$ is large enough. For such values of $r$, Theorem \ref{thm:faber_asymptotics} implies that $\cW_{n}(\sfE(r))\rightarrow 1$ as $n\rightarrow \infty$. The critical case occurs when $r = r_0$ is the smallest value for which $\sfE(r)$ is connected. In this case,~$\sfE(r_0)$ will no longer be a Jordan curve as it will contain a point of self intersection. Classical theory is insufficient to determine the limit points of $\cW_n(\sfE(r_0))$. This critical case is studied in relation to the family of polynomials $P(z) = z^m-1$ in \cite{bergman-rubin24}.

As we saw before, it is an interesting question to determine when Szeg\H{o}'s lower bound can be asymptotically saturated. In the case of several components or if the set contains a proper arc component then this cannot happen.
\begin{theorem}[Totik (2012) \cite{totik12}]
	Let \(\sfE\) consist of \(m\geq 2\) analytic Jordan curves lying exterior to one another.
	\begin{itemize}
		\item There exists a \(\beta>0\) and an infinite set \(I_1\subset \N\) such that
		\[\cW_n(\sfE)\geq 1+\beta,\qquad n\in I_1.\]
		\item There exists a \(C>0\) and an infinite set \(I_2\subset \N\) such that
		\[\cW_n(\sfE)\leq 1+Cn^{-1/(m-1)},\qquad n\in I_2.\]
	\end{itemize}
\end{theorem}
\begin{theorem}[Totik (2014) \cite{totik14}]
	Let \(\sfE\subset \C\) be compact with unbounded complement \(\Omega_{\sfE}\). If there is an open ball \(B\) such that \(B\cap \sfE\) is a \(C^{1+\alpha}\) Jordan arc, and \(B\setminus \sfE\subset \Omega_{\sfE}\), then there is a \(\beta>0\) such that
	\[\cW_n(\sfE)\geq 1+\beta, \qquad n\in \N.\]
\end{theorem}

We saw previously that $t_n([-1,1]) = 2\capacity([-1,1])^n$ so $\cW_n([-1,1])=2$ and this is in fact the optimal lower bound in the real setting. This was shown by Schiefermayr in \cite{schiefermayr08} using different means than presented here. Our proof is based on the same idea as the proof of Theorem \ref{thm:szego_inequality}.
	\begin{theorem}[Schiefermayr (2008) \cite{schiefermayr08}]
		Let $\sfE\subset \R$ denote a compact set, then
		\begin{equation}
			2\capacity(\sfE)^n\leq t_n(\sfE).
			\label{eq:schiefermayr_inequality}
		\end{equation}
		\label{thm:schiefermayr}
	\end{theorem}
	\vspace{-1cm}
	\begin{proof}
		The polynomial $T_n^{\sfE}$ is real and therefore
		\[\sfE\subset \sfE_n\coloneqq\{x: -t_n(\sfE)\leq T_n^{\sfE}(x)\leq t_n(\sfE)\}.\]
		From \eqref{eq:capacity_preimage} we gather that
		\[\capacity(\sfE_n) = \capacity([-t_n(\sfE),t_n(\sfE)])^{1/n} = \left(\frac{t_n(\sfE)}{2}\right)^{1/n}.\]
		The result follows by monotonicity of capacity with respect to set inclusion.
	\end{proof}
The sets $\sfE_n$ appearing in the proof of Theorem \ref{thm:schiefermayr} turn out to be important to the study of Widom factors in relation to real sets. It is easy to see from Theorem \ref{thm:alternation_theorem} that $T_n^{\sfE} = T_n^{\sfE_n}$ in this case. We will return to this later.
Totik has also proven an analogue of Theorem \ref{thm:csz_szego_lower_bound_equality} for the saturation of Schiefermayr's lower bound for real subsets. A new proof of the first part of the following theorem is provided in \cite{christiansen-simon-zinchenko-III}. 
\begin{theorem}[Totik (2011), (2014) \cite{totik11,totik14}]
	Let $\sfE\subset \R$ and fix $n\in \N$. Then
	\[\cW_n(\sfE) = 2\]
	if and only if there exists some polynomial $P$ of degree $n$ such that
	\[P^{-1}([-1,1]) = \sfE.\]
	Furthermore
	\[\lim_{n\rightarrow \infty}\cW_n(\sfE) = 2\]
	if and only if $\sfE$ is an interval.
	\label{thm:totik_schiefermayr_saturation}
\end{theorem}

%One of the main points of Paper \I is to investigate sets in the complex plane where $\cW_n(\sfE)\rightarrow 2$ as $n\rightarrow \infty$. 
\subsection{Totik--Widom bounds}

%A central question within the study of Chebyshev polynomials concerns placing upper bounds on Widom factors. If there exists a constant \(C\geq 1\) such that
%\[\cW_n(\sfE)\leq C,\qquad n\in \N\]
%then we say, in accordance with \cite{christiansen-simon-zinchenko-I}, the sequence of Widom factors \(\cW_n(\sfE)\) satisfies a \emph{Totik--Widom bound}.
A central problem in the theory of Chebyshev polynomials is to obtain uniform upper bounds for the associated Widom factors. 
We say that a compact set \(\sfE\) obeys \emph{Totik--Widom bounds} if there exists a constant \(C \ge 1\) such that
\[
\cW_n(\sfE) \le C,
\qquad n \in \N,
\]
in accordance with the terminology of \cite{christiansen-simon-zinchenko-I}.
If \(\sfE\) consists of a disjoint union of \(C^{1+\alpha}\) Jordan curves and arcs then it follows from \cite[Theorem 1.3]{totik14_2} (and implicitly in \cite{widom69} under the condition of \(C^{2+\alpha}\) smoothness) that \(\sfE\) satisfies a Totik--Widom bound. These bounds are implicit, an explicit Totik--Widom bound is provided in the case of lemniscates.
\begin{theorem}[Christiansen, Simon \& Zinchenko (2020) \cite{christiansen-simon-zinchenko-III}]
	Let $P$ be a monic polynomial of degree $m$ and 
	\[\sfE(r) = \{z: |P(z)| = r^m\}.\]
	For any $n$
	\[\cW_n(\sfE(r))\leq \max_{1\leq j \leq m} \cW_j(\sfE(r)).\]
\end{theorem}
\begin{proof}
	We have already seen that $\capacity(\sfE(r)) = r$. Any natural number can be expressed as~$nm+l$ where $l\in \{0,1,\dotsc,m-1\}$ and $n\in \N$. Therefore by making use of Theorem \ref{thm:chebyshev_on_lemnsicate} we see that
	\[t_{nm+l}(\sfE(r))\leq t_{l}(\sfE(r))t_{nm}(\sfE(r)) = t_l(\sfE(r))\capacity(\sfE(r))^{nm} = \cW_l(\sfE(r))\capacity(\sfE(r))^{nm+l}.\]
	In conclusion,
	\[\cW_{nm+l}(\sfE(r))\leq \cW_l(\sfE(r))\leq \max_{1\leq j \leq m}\cW_j(\sfE(r)).\qedhere\]
\end{proof}
It follows from the proof that for a fixed $l$ the mapping $n\mapsto \cW_{nm+l}(\sfE(r))$ is decreasing. If $r$ is large enough so that $\sfE(r)$ is an analytic curve then Theorem \ref{thm:faber_asymptotics} implies that the limit is $1$.

It is shown in~\cite{christiansen-simon-zinchenko-I} that so-called \emph{Parreau--Widom} sets satisfy Totik--Widom bounds, generalizing previous results in \cite{totik09,widom69}. These are sets $\sfE\subset \C$ for which
\[\parreauwidom(\sfE)\coloneqq\sum_{\bigl\{z: \nabla G_{\sfE}(z) = 0\bigr\}}G_{\sfE}(z)<\infty,\]
see e.g. \cite{christiansen-simon-zinchenko-I, sodin-yuditskii97}.
In words, this quantity is equal to the sum of the critical values of the corresponding Green's function with a pole at infinity. It is clear that any finite union of compact intervals are Parreau--Widom sets. As we will see, a Parreau--Widom set \(\sfE\subset \R\) always satisfies a Totik--Widom bound. 
	
\begin{theorem}[Christiansen, Simon \& Zinchenko (2017) \cite{christiansen-simon-zinchenko-I}]
\label{thm:christiansen-simon-zinchenko}
	Let $\sfE\subset \R$ be a regular Parreau-Widom set. Then
	\begin{equation}
		\cW_n(\sfE)\leq 2\exp(\parreauwidom(\sfE)).
		\label{eq:parreau_widom_bound}
	\end{equation}
	\label{thm:jacob}
\end{theorem}
Proving this requires an explicit representation of the equilibrium measure on the set \[\sfE_n \coloneqq T_n^{\sfE}([-t_n(\sfE),t_n(\sfE)]).\] See also \cite{christiansen-simon-zinchenko-I, totik01}.
	
	\begin{lemma}
		Let \( P \) be a real polynomial of degree \( n \) such that there exist real points \( x_0 < x_1 < \cdots < x_n \) satisfying \( P(x_k) = (-1)^{n-k} \). If \(\sfE = P^{-1}([-1,1])\), then the equilibrium measure on \(\sfE\) is given by
		\begin{equation}
			d\mu_{\sfE}(x) = \frac{1}{n\pi}\frac{|P'(x)|}{\sqrt{1-P^{2}(x)}}dx.
			\label{eq:equilibrium_on_e_n}
		\end{equation}
		\label{thm:equilibrium_measure_e_n}
	\end{lemma}
	\begin{proof}
		We begin with noting that summing over the roots of $P(x) = t$ counting multiplicity we have
		\[\sum_{k=1}^{n}\log|P^{-1}(t)-z| = \log |t-P(z)|.\]
		
		We can partition $\sfE$ into $n$ intervals where $P$ is monotonic and which $P$ maps to $[-1,1]$ (these are so-called bands, see the discussion below). Applying the change of variables $t = P(x)$ on each of the intervals and recalling \eqref{eq:green_potential_polynomial_preimage} and \eqref{eq:greens_function_from_potential} gives us
		\begin{align*}
			\frac{1}{n\pi}&\int_{\sfE}\log|x-z|\frac{|P'(x)|}{\sqrt{1-P^{2}(x)}}dx = \frac{1}{n\pi}\sum_{k=1}^{n}\int_{-1}^{1}\log|P^{-1}(t)-z|\frac{1}{\sqrt{1-t^2}}dt \\
			& = \frac{1}{n\pi}\int_{-1}^{1}\log |t-P(z)|\frac{1}{\sqrt{1-t^2}}dt= \frac{1}{n}G_{[-1,1]}(P(z))-\frac{1}{n}\log 2 = G_{\sfE}(z)+\log \capacity(\sfE).
		\end{align*}
		Through another application of \eqref{eq:greens_function_from_potential} we conclude that
		\[G_{\sfE}(z)+\log \capacity(\sfE) = \int_{\sfE}\log |x-z|d\mu_{\sfE}(x)\]
		and therefore
		\[\frac{1}{n\pi}\int_{\sfE}\log|x-z|\frac{|P'(x)|}{\sqrt{1-P^{2}(x)}}dx = \int_{\sfE}\log |x-z|d\mu_{\sfE}(x).\]
		Since this equality persists for all $z\notin \sfE$ it follows from \cite[Theorem 2.7.5]{ransford95} that equality holds everywhere in $\C$. Through \cite[Theorem 3.7.4]{ransford95} we find that \eqref{eq:equilibrium_on_e_n} holds.
	\end{proof}A simple consequence of Lemma \ref{thm:equilibrium_measure_e_n} is that any band of $P$ weighs the same under the associated equilibrium measure. A band is an interval $[a,b]$ between adjacent alternating points of $P$. In other words, $P([a,b]) = [-1,1]$ and $P'(x)\neq 0$ for $x\in (a,b)$, \cite{christiansen-simon-zinchenko-I}. For such a band~$[a,b]$
	\begin{equation}
		\mu_{\sfE}([a,b]) = \frac{1}{n\pi}\int_{[a,b]}\frac{|P'(x)|}{\sqrt{1-P^{2}(x)}}dx = \frac{1}{n\pi}\int_{-1}^{1}\frac{dx}{\sqrt{1-x^2}}=\frac{1}{n}.
		\label{eq:weight_of_band}
	\end{equation}
	By considering $P\coloneqq T_n^{\sfE}/t_n(\sfE)$ we can easily translate these results to the Chebyshev polynomial setting. Indeed, by letting $\sfE_n = P^{-1}([-1,1]) = (T_n^{\sfE})^{-1}([-t_n,t_n])$ we find that each band of~$T_n^{\sfE}/t_n(\sfE)$ has the same measure $1/n$ under $\mu_{\sfE_n}$.
\begin{proof}[Proof of Theorem \ref{thm:jacob}]
	We reproduce the proof following \cite{christiansen-simon-zinchenko-I}. Let $\sfE_n \coloneqq (T_n^{\sfE})^{-1}([-t_n(\sfE),t_n(\sfE)])$ then \eqref{eq:capacity_preimage} implies that \begin{equation}
 t_n(\sfE) = 2\capacity(\sfE_n)^n.
 \label{eq:cheb_norm_en_set}	
 \end{equation}
On the other hand since $G_{\sfE}-G_{\sfE_n}$ is harmonic away from $\sfE_n$ we find that
	\[\int_{\sfE_n}G_{\sfE}(x)d\mu_{\sfE_n}(x) = \lim_{z\rightarrow \infty}\left(G_{\sfE}(z)-G_{\sfE_n}(z)\right) = \log \frac{\capacity(\sfE_n)}{\capacity(\sfE)}.\]
	Since \(\sfE\) is regular, we obtain that $G_{\sfE}(x) = 0$ for every $x\in \sfE$. Furthermore $\R\setminus \sfE$ consists of a union of open intervals, so called gaps. If $K\coloneqq(a,b)$ is such a gap then a consequence of alternation (see Theorem \ref{thm:alternation_theorem}) is that $K$ cannot intersect two different bands, where we recall the definition of a band, being the closure of a component of $(T_n^{\sfE})^{-1}((-t_n(\sfE),t_n(\sfE)))$, see also~\cite{christiansen-simon-zinchenko-I,peherstorfer92,peherstorfer93}. Lemma \ref{thm:equilibrium_measure_e_n} and in particular \eqref{eq:weight_of_band} implies that any band has weight $1/n$ with respect to $\mu_{\sfE_n}$. We therefore conclude that any gap $K$ satisfies
	\[\mu_{\sfE_n}(K) \leq \frac{1}{n}.\]
	Also, $G_{\sfE}(x)>0$ if $x\notin \sfE$ and since $G_{\sfE}$ vanishes at the endpoints of $K$ we find that $G_{\sfE}$ attains a maximum on $K$ at a point where $\partial_xG_{\sfE}$ vanishes. On the other hand $\partial_yG_{\sfE} = 0$ on $\R\setminus \sfE$ by symmetry. Therefore $G_{\sfE}$ is maximized on $(a,b)$ at a point where $\nabla G_{\sfE} = 0$. By letting~$\{K_j\}_{j\in I}$ denote the bounded gaps of $\R\setminus \sfE$ we find that
	\[\log \frac{\capacity(\sfE_n)}{\capacity(\sfE)} = \int_{\sfE_n}G_{\sfE}(x)d\mu_{\sfE_n}(x) \leq \sum_{j\in I}\mu_{\sfE_n}(K_j)\max_{x\in K_j}G_{\sfE}(x)\leq\frac{1}{n}\sum_{\{ z:\nabla G_{\sfE}(z) = 0\}}G_{\sfE}(z).\] 
	Exponentiating and recalling \eqref{eq:cheb_norm_en_set} gives us that
	\[t_n(\sfE) = 2\capacity(\sfE_n)^n \leq 2\capacity(\sfE)^n\exp (\parreauwidom(\sfE))\]
	which completes the proof.
\end{proof}
For further consequences of Theorem \ref{thm:jacob}, see \cite{christiansen-simon-zinchenko-II}. Extending these considerations to the full generality of the complex plane remains an area of research in its early stages and it is an open question whether Parreau--Widom subsets of the complex plane satisfy Totik--Widom bounds. In \cite{schiefermayr-zinchenko22} these concepts were investigated for subsets of the unit circle. A natural first step in generalizing \eqref{eq:parreau_widom_bound} to the complex plane is to demonstrate that compact connected sets have bounded Widom factors. This question was originally posed as an interesting problem in~\cite[Problem 4.4]{pommerenke72}, and it was initially claimed that D. Wrase had provided an example of a compact connected set with unbounded Widom factors. However, recent findings in~\cite{andrievskii17} have cast doubt on this claim. After nearly 50 years of being considered settled, this question now appears to be open once again. 

One specific family of sets which provide large Widom factors in the context of connected compact sets are the symmetric star graphs
\[\sfE_m = \{z^m \in [-2,2]\}\]
studied in \cite{christiansen-eichinger-rubin23}. It is shown there that
\[\cW_n(\sfE_m)\to 2\]
as \(n\to \infty\) but also that
\begin{equation}
	\cW_{m-1}(\sfE_m) = 2^{2-1/m} \to 4,\qquad m\to \infty.
	\label{eq:4}
\end{equation}
Consequently, if Parreau--Widom type bounds extend to the complex setting, then the \(2\) in \eqref{eq:parreau_widom_bound} needs to be replaced by some number which is at least \(4\).

So far, we have considered Faber polynomials as a natural family of trial polynomials with asymptotically small norms. 
An alternative approach is based on constructing polynomials via discretizations of the corresponding equilibrium measure. 
The following results are obtained through such constructions.

\begin{theorem}[Totik \& Varga (2015) \cite{totik-varga15}]
Let $\sfE\subset \C$ be a compact set with unbounded complement $\Omega_\sfE$. If $\partial \Omega_{\sfE}$ is a finite union of Dini-smooth Jordan arcs, disjoint except possibly at their endpoints and such that $\partial \Omega_{\sfE}$ does not contain a point with vanishing external angle, then~$\sfE$ satisfies a Totik--Widom bound.
	\label{thm:totik_varga_dini_smooth_boundary}
\end{theorem}

In \cite{andrievskii16, andrievskii17,andrievskii-nazarov19} Totik--Widom bounds for sets with reduced boundary regularity were studied. To better understand these results we consider quasicircles and quasiconformal arcs. A quasicircle $\sfE$ is a Jordan curve such that any three points on the boundary satisfies the so-called \textit{Ahlfors condition}: there exists some $A>0$, such that if $z_1,$ $z_2$ both belong to~$\sfE$ then
\[|z_1-z|+|z-z_2|\leq A|z_1-z_2|\]
whenever $z$ lies on that subarc of $\sfE$, with smallest diameter connecting $z_1$ and $z_2$, see e.g.~\cite{ahlfors66,garnett-marschall05}. A quasidisk is the interior region of a quasicircle. Examples of quasicircles include boundaries whose parametrization satisfies Lipschitz conditions but also fractal sets like the von Koch snowflake. A quasiconformal arc is any proper subarc of a quasicircle.
\begin{theorem}[Andrievskii (2017) \cite{andrievskii17,andrievskii-nazarov19}]
\label{thm:totik-widom-quasidisk}
If\/ $\sfE=\bigcup_{j=1}^{m}\sfE_j$ where the sets $\sfE_j$ are mutually disjoint closed quasidisks and quasiconformal arcs then $\sfE$ satisfies a Totik--Widom bound.\end{theorem}

It is not at all clear what the least upper bound is. Andrievskii also considered the case where no regularity is present and concluded the following.
\begin{theorem}[Andrievskii (2017) \cite{andrievskii17}]
	Let $\sfE = \bigcup_{j=1}^{m}\sfE_j$ where the sets $\sfE_j$ are mutually distjoint compact and connected sets that all satisfy $\operatorname{diam}(\sfE_j)>0$. Then as $n\rightarrow \infty$
	\[\cW_n(\sfE) = O(\log n).\] 
	\label{thm:andrievskii_general_totik_widom}
\end{theorem}
\vspace{-.5cm}
These results highlight a significant distinction between Chebyshev and Faber polynomials for sets lacking boundary regularity. In fact, Gaier \cite{gaier99}, building on an example by Clunie~\cite{clunie59}, demonstrated the existence of a quasicircle $\sfE$ such that there is a positive constant $\alpha$ for which the associated sequence of Faber polynomials $\{F_n^{\sfE}\}$ satisfies
\[\|F_{n_k}^{\sfE}\|_{\sfE}>n_k^\alpha\]
for an increasing sequence $n_k$. On the other hand, Theorem \ref{thm:totik-widom-quasidisk} demonstrates that $\sfE$ satisfies a Totik--Widom bound, while Theorem \ref{thm:andrievskii_general_totik_widom} shows that the growth rate $n^\alpha$ is not possible for Widom factors of single-component sets. We emphasize that the purported counterexample to connected compact sets satisfying Totik--Widom bounds, as claimed in \cite{pommerenke72}, is based on the very construction referenced in \cite{clunie59}. Andrievskii's result shows that this example does not provide the necessary counterexample.

\subsection{Chebyshev and Faber polynomials relative to equipotential curves}	
	\label{subs:equipotential}
	We have previously seen in Section \ref{subs:cheb_curve} that if \(\sfE\) is a \(C^{1+\alpha}\) Jordan curve then the Faber polynomials,~\(F_n^{\sfE}\) will have asymptotically small norms as \(n\to\infty\). Thus providing good substitutes for Chebyshev polynomials of large degrees. Furthermore, under sufficient smoothness assumptions on the Jordan curve, one can actually establish
	\[\lim_{n\to\infty}\max_{z\in \sfE}|T_n^{\sfE}(z)-F_n^{\sfE}(z)| = 0\]
	see e.g. \cite[Theorem 8.3]{widom69}. Thus Chebyshev polynomials and Faber polynomials get measurably close for increasing degrees.
	
	There is another relation between Chebyshev polynomials and Faber polynomials that was originally studied by Faber in \cite{faber20}, where the degree is fixed but where the set instead varies.

	To explain this let \(\sfE\subset \C\) be a compact set with positive logarithmic capacity and denote
	\[\sfE(r) = \{z:G_{\sfE}(z) = \log r\}.\]
	For sufficiently large values of \(r\), the set \(\sfE(r)\) will form an analytic Jordan curve and consequently its unbounded complement \(\Omega_{\sfE(r)}\) will be simply connected with respect to the Riemann sphere. The Green function corresponding to \(\sfE(r)\) is given by
	\[G_{\sfE(r)}(z) = G_{\sfE}(z)-\log r.\]
	For large enough \(r>1\) the associated exterior conformal map \(\Phi_r\) fulfilling \eqref{eq:canonical_conformal_map} satisfies 
	\[\log |\Phi_r(z)| = G_{\sfE(r)}(z)=G_{\sfE}(z)-\log r\]
	by~\eqref{eq:green_func_conformal_map} and consequently, any two such maps, corresponding to different values of \(r\), are related by a multiplicative factor. It follows that the monic Faber polynomial, denoted
	\[\mathcal{F}_n^{\sfE}(z)\coloneqq \Phi_r'(\infty)^{-n}F_n^{\sfE(r)}(z)\]
	is independent of \(r\).
	
	If \(P\) is a monic polynomial of degree \(m\) and \(\sfE(r) = \{z:|P(z)| = r^m\}\),  where \(\sfE = \sfE(1)\), then we already saw in Theorem \ref{thm:chebyshev_on_lemnsicate} that
	\[T_{nm}^{\sfE(r)}(z) = P(z)^{n}.\]
	It can be shown that the exterior conformal map satisfying \eqref{eq:canonical_conformal_map}, is given by \[\Phi(z) = \frac{P(z)^{1/m}}{r},\]
	and consequently \(\mathcal{F}_{nm}^{\sfE}(z) = P(z)^{n}.\)
	This shows that the relation
	\begin{equation}
		T_{nm}^{\sfE(r)}(z) = \mathcal{F}_{nm}^{\sfE}
		\label{eq:cheb_equal_faber_lemniscate}
	\end{equation}
	holds for any value of \(r>0\). The equality in \eqref{eq:cheb_equal_faber_lemniscate} served as an example for Faber in \cite{faber20} to motivate that these families of polynomials may be related on equipotential curves. In fact, Faber could demonstrate that this remarkable relation holds true on an interval.
	From \eqref{eq:faber_cheb_equal}, we see that
	\[T_n^{[-1,1]} = \mathcal{F}_n^{[-1,1]},\]
	and this equality persists on all equipotential curves of \([-1,1]\).
	\begin{theorem}[Faber (1920) \cite{faber20}]
		Let \(\sfE = [-1,1]\). Then, for any \(r>1\),
		\[T_n^{\sfE} = T_n^{\sfE(r)} = \mathcal{F}_n^{\sfE}.\]
	\end{theorem}
	
	The fact that Chebyshev polynomials may remain constant on equipotential curves has subsequently been studied in a number of special cases.
	\begin{theorem}[Peherstorfer (1996) \cite{peherstorfer96}]
		Let \(P\) be a polynomial of degree \(m\), and let \(\sfE = P^{-1}([-1,1])\). Then
		\[T_{nm}^{\sfE} = T_{nm}^{\sfE(r)}\]
		for every \(n\in \N\) and \(r>1\).
	\end{theorem}
	Actually, the first extension of Faber's original result was established by Fischer in \cite{fischer92} in the case where \(P^{-1}([-1,1])\) consists of exactly two intervals, see also \cite{christiansen-simon-zinchenko-IV}. Subsequent research shows that the interval can be replaced by any compact set which satisfies the property that the Chebyshev polynomial of degree \(n\) is minimal for all equipotential curves, see \cite{bloom-calvi00, peherstorfer-steinbauer01}. Another example is provided by Julia sets of quadratic polynomials.
	
	\begin{theorem}[Stawiska (1996) \cite{stawiska96}]
		If \(\sfE = J_{P_{\lambda}}\) is the Julia set of the polynomial \(P_\lambda(z) = (z-\lambda)^2\) for \(0\leq \lambda \leq 2\), then
		\[T_{2^n}^{\sfE} = T_{2^{n}}^{\sfE(r)}\]
		holds for all \(n\in \N\) and \(r>1\).
	\end{theorem} 
	
	It turns out that it is always the case that \(T_n^{\sfE(r)}\) stabilizes for large \(r\) and approaches the corresponding Faber polynomial as the following shows.
	\begin{theorem}[\cite{minadiaz-rubin25}]
	Let \(\sfE\subset \C\) be a compact set with positive capacity, and \(n\in \N\). Then
	\[\lim_{r\to \infty}T_n^{\sfE(r)} = \mathcal{F}_n^{\sfE}.\]
	\label{thm:minadiaz}
	\end{theorem}
	\vspace{-.5cm}
	Furthermore, the proof establishes that the convergence is rapid in the sense that
	\[\max_{z\in \sfE(r)}|T_n^{\sfE(r)}(z)-\mathcal{F}_n^{\sfE}(z)| = O(r^{-1}), \qquad t\to \infty.\]
	This result highlights yet another close relation between Faber polynomials and Chebyshev polynomials. One advantage of establishing such relations is that Faber polynomials can be analyzed through the well-developed theory of conformal maps, see e.g.~\cite{pommerenke72}, and are therefore comparatively well understood.

\subsection{The zeros of Chebyshev polynomials}
To conclude this section, we turn our interest to the behavior of the zeros of Chebyshev polynomials. For a fixed degree $n$, almost nothing is known about the precise location of the zeros of $T_n^{\sfE}$. Many times, the asymptotical zero distribution is the interesting object to study. The following result constitutes an exception.

\begin{theorem}[Fej\'{e}r (1922)	 \cite{fejer22}]
	Let $\sfE\subset \C$ be a compact set and $w:\sfE\rightarrow [0,\infty)$ a non-negative weight function which is non-zero on at least $n$ points. All zeros of $T_n^{\sfE,w}$ lie in $\operatorname{cvh}(\supp (w))$ -- the convex hull of the support of the weight function.
	\label{thm:fejer_zeros_convex_hull}
\end{theorem}
\begin{proof}
	If $w$ has precisely $n$ points in its support then the Chebyshev polynomial is uniquely determined to be the polynomial with all its zeros coinciding with the supporting set. 
	
	We consider the case where $w$ has at least $n+1$ points of support. In order to obtain a contradiction, assume that $T_n^{\sfE,w}(z) = \prod_{k=1}^{n}(z-a_k)$ and that $a_1\notin \operatorname{cvh}(\supp(w))$. The separating hyperplane theorem, see e.g. \cite{boyd-vandenberghe04}, tells us that there exists a line \(L\) which decomposes~$\C\setminus L$ into two connected components, one containing $\supp(w)$ and one containing $a_1$. If $a_1^\ast$ denotes the orthogonal projection of $a_1$ onto $L$ then 
	\[|z-a_1^\ast|<|z-a_1|\]
	holds for every $z\in \supp(w)$. Consequently 
	\[w(z)|z-a_1^\ast|\prod_{k=2}^{n}|z-a_k|\leq w(z)\prod_{k=1}^{n}|z-a_k|= w(z)|T_n^{\sfE,w}(z)|\]
	with strict inequality on $\supp(w)\setminus\{a_k\mid k=2,\dotsc,n\}$. Since $w$ contains at least $n+1$ points in its support we conclude that $(z-a_1^\ast)\prod_{k=2}^{n}(z-a_k)$ is a monic polynomial of degree $n$ with smaller weighted supremum norm than $T_n^{\sfE,w}$ which is a contradiction.
\end{proof}

The remaining results we consider in this section provide information on the asymptotical behavior of the zeros of the Chebyshev polynomials as the degree goes to infinity. The first study with this flavor was performed on partial sums of Taylor series by Jentzsch in \cite{jentzsch16} and later substantially extended by Szeg\H{o} in \cite{szego22}. They were both interested in describing the zero distribution of partial sums of power series of analytic functions. 

If $P$ is a polynomial of exact degree $n$ with zeros at $z_1,\dotsc,z_n$, counting multiplicity, then we define the normalized zero counting measure of $P$ as the probability measure
\begin{equation}
	\nu(P) = \frac{1}{n}\sum_{j=1}^{n}\delta_{z_j}
	\label{eq:zero_counting_measure}
\end{equation}
where $\delta_{z}$ denotes the Dirac measure at $z$.

\begin{theorem}[Jentzsch {\cite{jentzsch16}}, Szeg\H{o} {\cite{szego22}}]
Let
\begin{equation}
f(z)=\sum_{k=0}^{\infty} a_k z^k
\label{eq:power_series}
\end{equation}
be a power series with radius of convergence $0<r<\infty$, and let
\[
P_n(z)=\sum_{k=0}^{n} a_k z^k
\]
denote its $n$th partial sum.  
Then there exists a subsequence of degrees $\{n_m\}_{m}$ such that
\[
\nu(P_{n_m}) \xrightarrow{*} \frac{d\theta}{2\pi}\Big|_{\{\,|z|=r\,\}}
\]
as $m\to\infty$.
\label{thm:jentzsch-szego}
\end{theorem}
To be precise concerning the accreditation of this result, Jentzsch showed in \cite{jentzsch16} that every point of $\{z: |z| = r\}$ was a limit point of the zeros of the partial sums. Szeg\H{o} extended Jentzsch' result in \cite{szego22} by showing that for a subsequence of $\{P_n\}$ the corresponding zeros distribute in an equidistributed fashion in any circular sector with respect to the corresponding angle and that these zeros approach the circle determined by the radius of convergence~$|z| = r$. See also \cite[\S 2.1]{andrievskii-blatt01} for a potential theoretic proof of Theorem \ref{thm:jentzsch-szego}. The proof stated there is based on \cite[Theorem 2.1.1]{andrievskii-blatt01} a simplification dealing with Chebyshev polynomials that we now formulate.

\begin{theorem}[Blatt, Saff \& Simkani (1988) \cite{blatt-saff-simkani88}]
	Let $\sfE\subset \C$ be a compact set with $\capacity(\sfE)>0$ such that the unbounded component $\Omega_{\sfE}$ of $\C\setminus \sfE$ has a boundary which is regular (in the sense of potential theory).	
%	 and $\{P_n\}$ a sequence of monic polynomials such that $\deg P_n = n$. If 
%	\begin{equation}
%		\limsup_{n\rightarrow \infty}\|P_n\|^{1/n}\leq \capacity(\sfE)
%		\label{eq:blatt-saff-simkani-asymptotically-extremal}
%	\end{equation}
%	and
	If
	\begin{equation}
		\nu(T_{n}^{\sfE})(A)\rightarrow 0
		\label{eq:interior_zero_condition}
	\end{equation}
	as $n\rightarrow \infty$ for any closed set $A$ contained in the union of the bounded components of $\C\setminus \overline{\Omega}_{\sfE}$ then
	\begin{equation}
		\nu(T_n^{\sfE})\xrightarrow{\ast}\mu_{\sfE}
	\end{equation}
	as $n\rightarrow \infty$.
	\label{thm:blatt-saff-simkani-zeros-inside}
\end{theorem}

In the special case where the bounded components of the complement are void, the following result trivially follows from Theorem \ref{thm:blatt-saff-simkani-zeros-inside}.
\begin{corollary}[Blatt, Saff \& Simkani (1988) \cite{blatt-saff-simkani88}]
	Suppose that $\sfE\subset \C$ is a compact set with $\capacity(\sfE)>0$, connected complement, regular boundary and empty interior. Then
	\begin{equation}
		\nu(T_n^{\sfE})\xrightarrow{\ast}\mu_{\sfE},\quad \text{ as} \quad n\rightarrow \infty.
	\end{equation}
\end{corollary}

We intend to prove Theorem \ref{thm:blatt-saff-simkani-zeros-inside} by combining recent simplified proofs. As a first step, we show that the support of any limit measure of $\nu(T_n^{\sfE})$ is contained in $\partial \Omega_{\sfE}$. It is always so that most zeros of $T_n^{\sfE}$ approach the so-called \textit{polynomially convex hull} of $\sfE$, with this we mean the set \(\C\setminus\Omega_{\sfE}\).

\begin{theorem}[Blatt, Saff \& Simkani (1988) \cite{blatt-saff-simkani88}]
	Let $\sfE\subset \C$ be a compact set with $\capacity(\sfE)>0$ then
	\[\nu(T_n^{\sfE})(A)\rightarrow 0\]
	for any closed subset $A$ in the unbounded component $\Omega_{\sfE}$ of $\C\setminus\sfE$.
	\label{thm:nbr_of_zeros_complement}
\end{theorem}
\begin{proof}
	We define the sequence of functions $\{h_n\}$ via
	\begin{equation}
		h_n(z) = \frac{1}{n}\log|T_n^{\sfE}(z)|+U^{\mu_{\sfE}}(z)+\frac{1}{n}\sum_{k=1}^{m_n}G_{\Omega_{\sfE}}(z,z_{k,n})
	\end{equation}
	where $z_{1,n},\dotsc,z_{m_n,n}$ is an enumeration of the zeros of $T_n^{\sfE}$ that reside in $\Omega_{\sfE}$, counting multiplicity. From the properties of $G_{\Omega_{\sfE}}$ detailed in \eqref{eq:greens_function_definition} we gather that the function $h_n$ is harmonic in $\Omega_{\sfE}$. This harmonicity extends to the point at $\infty$. Furthermore, for quasi-every~$\zeta\in \partial \Omega_{\sfE}$, 
	\[\limsup_{z\rightarrow \zeta}h_n(z) = \frac{1}{n}\log |T_n^{\sfE}(\zeta)|-\capacity(\sfE) \leq \frac{1}{n}\log \|T_n^{\sfE}\|_{\sfE}-\capacity(\sfE) \eqqcolon\varepsilon_n\]
	where $\varepsilon_n\rightarrow 0$, as a result of Theorem \ref{thm:ffs-theorem}. From the maximum principle, see e.g. \cite[Theorem 3.6.9]{ransford95}, we conclude that $h_n(z)\leq \varepsilon_n$ for every $z\in \Omega_{\sfE}$. Since 
	\[\lim_{z\rightarrow \infty}\left(\frac{1}{n}\log |T_n^{\sfE}(z)|+U^{\mu_{\sfE}}(z)\right) = 0\]
	the symmetry of $G_{\Omega_{\sfE}}$ provides us with the estimate
	\[\varepsilon_n\geq h_n(\infty) = \frac{1}{n}\sum_{k=1}^{m_n}G_{\Omega_{\sfE}}(\infty,z_{n,k})=\frac{1}{n}\sum_{k=1}^{m_n}G_{\sfE}(z_{n,k}) = \int_{\Omega_{\sfE}}G_{\sfE}(z)d\nu(T_n^{\sfE})(z).\]
	In particular,
	\begin{equation}
		\lim_{n\rightarrow \infty}\int_{\Omega_{\sfE}}G_{\sfE}(z)d\nu(T_n^{\sfE})(z) = 0.
		\label{eq:green_integral_zero_counting_outside}
	\end{equation}
	Let $A$ denote any closed subset of $\Omega_{\sfE}$. Since $A$ is compact, there exists some $c>0$ such that~$G_{\sfE}(z)\geq c$ for all $z\in A$. As a consequence
	\[\nu(T_n^{\sfE})(A)\leq \frac{1}{c}\int_{A}G_{\sfE}(z)d\nu(T_n^{\sfE})\leq \frac{1}{c}\int_{\Omega_{\sfE}}G_{\sfE}(z)d\nu(T_n^{\sfE})(z).\]
	From \eqref{eq:green_integral_zero_counting_outside} we gather that $\nu(T_n^{\sfE})(A)\rightarrow 0$ as $n\rightarrow \infty$. 
\end{proof}

If \eqref{eq:interior_zero_condition} holds then any limit point measure of $\nu(T_n^{\sfE})$ must be supported on the boundary of the unbounded complement $\Omega_{\sfE}$. The equilibrium measure $\mu_{\sfE}$ has this very property and so this is what we expect based on Theorem \ref{thm:blatt-saff-simkani-zeros-inside}. To conclude the proof we need a certain minimality condition which is a consequence of the Szeg\H{o} inequality and weak asymptotics of Chebyshev polynomials outside of the convex hull of $\sfE$.

\begin{theorem}
	Let $\sfE$ be a compact set with positive capacity. Uniformly on compact subsets of $\overline{\C}\setminus \operatorname{cvh}(\sfE)$ it holds that
	\begin{equation}
		\lim_{n\to\infty}\frac{|T_n^{\sfE}(z)|^{1/n}}{\capacity(\sfE)\exp{G_{\sfE}(z)}}= 1.
	\end{equation}
	\label{thm:saff_totik_chebyshev_asymptotics_unbounded_complement}
\end{theorem}
\begin{proof}
	As a consequence of Theorem \ref{thm:fejer_zeros_convex_hull}, we deduce that the family of functions $\{h_n\}$ defined by
	\[h_n(z)\coloneqq\frac{1}{n}\log\|T_n^{\sfE}\|_{\sfE}+G_{\sfE}(z)-\frac{1}{n}\log|T_n^{\sfE}(z)|\]
	is a family of harmonic functions on $\C\setminus \operatorname{cvh}(\sfE)$. From \eqref{eq:greens_function_at_infinity} together with \cite[Corollary 3.6.2]{ransford95} we gather that $h_n$ extends harmonically at infinity with the value
	\[h_n(\infty) = \frac{1}{n}\log \|T_n^{\sfE}\|_{\sfE}-\log \capacity(\sfE) = \frac{1}{n}\log\cW_{n}(\sfE) \geq 0.\]
	The inequality is a consequence of Theorem \ref{thm:szego_inequality}. Since $h_n$ extends superharmonically to the unbounded component $\Omega_{\sfE}$ of $\C\setminus \sfE$ and 
	\[\lim_{z\rightarrow \zeta}h_n(z) = \frac{1}{n}\log \|T_n^{\sfE}\|_{\sfE}-\frac{1}{n}\log|T_n^{\sfE}(\zeta)|\geq 0\]
	for q.e. $\zeta\in \partial\Omega_{\sfE}$ we conclude from the extended minimality principle \cite[Theorem 3.6.9]{ransford95} that~$h_n(z)\geq 0$ for every $z\in \Omega_{\sfE}$. On the other hand $h_n(\infty)\rightarrow 0$ due to Theorem \ref{thm:ffs-theorem}. The final ingredient is supplied by a variant of Harnack's theorem \cite[Theorem 1.3.10]{ransford95} which entails that these conditions are enough to guarantee that $h_n\rightarrow 0$ locally uniformly. Combining these considerations gives us that
	\[\exp(h_n(z)) = \frac{\|T_n^{\sfE}\|_{\sfE}^{1/n}\exp(G_{\sfE}(z))}{|T_n^{\sfE}(z)|^{1/n}} = \frac{\|T_n^{\sfE}\|_{\sfE}^{1/n}}{\capacity(\sfE)}\frac{\capacity(\sfE)\exp(G_{\sfE}(z))}{|T_n^{\sfE}(z)|^{1/n}}\rightarrow 1\]
	uniformly on compact subsets of $\overline{\C}\setminus \operatorname{cvh}(\sfE)$.
\end{proof}
While it is not true that $\nu(T_n^{\sfE})\xrightarrow{\ast}\mu_{\sfE}$ as $n\rightarrow \infty$ for any compact set $\sfE$ with $\capacity(\sfE)>0$, as exemplified by $\bbT = \sfE$, there is a general ``sweeping'' procedure which relates any limit point of $\nu(T_n^{\sfE})$ with $\mu_{\sfE}$. We recall the notation $\cM(\sfE)$ which denotes the family of probability measures on $\sfE$. Again, $\Omega_{\sfE}$ denotes the unbounded complement of $\sfE$. If $\mu\in \cM(\sfE)$ then we say that $\mu^b$ is the ``balayage'' of $\mu$ to $\partial \Omega_{\sfE}$ if $\mu^b\in \cM(\partial \Omega_{\sfE})$ and
\begin{equation}
	U^{\mu}(z) = U^{\mu^{b}}(z)
	\label{eq:balayage_outside}
\end{equation}
for quasi-every $z\in \overline{\Omega}_{\sfE}$. The measure $\mu^b$ defined this way is the unique measure in $\cM(\partial \Omega_{\sfE})$ satisfying \eqref{eq:balayage_outside} with finite energy $\cE(\cdot)<\infty$, see \cite[Theorem 2.2]{mhaskar-saff91}. Recall that $\mu_{\sfE}$ always has its support contained in $\partial \Omega_{\sfE}$.

Returning to the example of the unit circle we know that $U^{\mu_{\bbT}}(z) = -\log |z|$ for $|z|\geq 1$ and since $U^{\delta_0}(z) = -\log |z|$ for $|z|>0$ we conclude that $\mu_{\bbT}$ is the balayage of $\delta_0 = \nu(T_n^{\bbT})$. This observation can be significantly generalized as we now show, see also \cite[Theorem 2.1]{christiansen-simon-zinchenko-IV}.

\begin{theorem}[Mhaskar \& Saff (1991) \cite{mhaskar-saff91}]
	Let $\sfE$ denote a polynomially convex compact set with $\capacity(\sfE)>0$. If $\mu_\infty$ denotes any limit point of $\{\nu(T_n^{\sfE})\}$ then $\mu_\infty$ is supported on $\sfE$ and for all $z\in \Omega_{\sfE}$, the complement of $\C\setminus \sfE$, it holds that
	\begin{equation}
		U^{\mu_\infty}(z) = U^{\mu_{\sfE}}(z).
  	\end{equation} 
  	\label{thm:balayge-equality}
\end{theorem}
\vspace{-1cm}
\begin{proof}
	From Theorem \ref{thm:nbr_of_zeros_complement} it immediately follows that any limit point $\mu_\infty$ of the sequence~$\{\nu(T_n^{\sfE})\}$ is supported on the set $\sfE$. As a consequence $U^{\mu_\infty}$ is harmonic on $\Omega_{\sfE}$. Pick a subsequence $n_k$ such that $\nu(T_{n_k}^{\sfE})\xrightarrow{\ast}\mu_\infty$. Then
	\[U^{\mu_\infty}(z) = \lim_{n_k\rightarrow \infty}\int_{\C}\log \frac{1}{|z-\zeta|}d\nu(T_{n_k}^{\sfE})(\zeta) = \lim_{n_k\rightarrow \infty}-\frac{1}{n_k}\log |T_{n_k}^{\sfE}(z)|.\]
	Theorem \ref{thm:saff_totik_chebyshev_asymptotics_unbounded_complement} therefore implies that in a neighborhood of infinity
	\[U^{\mu_\infty}(z) = -G_{\sfE}(z)-\log \capacity(\sfE) = U^{\mu_{\sfE}}(z)\]
	and therefore the identity principle for harmonic functions \cite[Theorem 1.1.7]{ransford95} implies that this equality persists on $\Omega_{\sfE}$. 
\end{proof}

A consequence of Theorem \ref{thm:balayge-equality} is that the balayage of any limit point of $\{\nu(T_n^{\sfE})\}$ equals~$\mu_{\sfE}$. With this result at hand we are now in a position to finally prove Theorem \ref{thm:blatt-saff-simkani-zeros-inside}.

\begin{proof}[Proof of Theorem \ref{thm:blatt-saff-simkani-zeros-inside}]
	Let $\mu_\infty$ be any limit point of $\nu(T_n^{\sfE})$. Then Theorem \ref{thm:balayge-equality} implies that~$\mu_\infty$ is supported on the polynomially convex hull of $\sfE$. On the other hand the condition that~$\nu(T_n^{\sfE})(A)\rightarrow 0$ for any closed subset on the bounded components of $\C\setminus \partial \Omega_{\sfE}$ implies that the support of any limit point measure $\mu_\infty$ is contained in $\partial \Omega_{\sfE}$. The regularity (in the sense of potential theory) of $\partial\Omega_{\sfE}$ together with Theorem \ref{thm:balayge-equality} and the lower-semicontinuity of potentials gives us that
	\[U^{\mu_\infty}(\zeta)\leq \liminf_{z\rightarrow \zeta}U^{\mu_\infty}(z) = \liminf_{z\rightarrow \zeta}U^{\mu_{\sfE}}(z) \leq  -\log\capacity(\sfE).\]
	Since the support of $\mu_\infty$ lies on $\partial \Omega_{\sfE}$ and $\mu_\infty$ necessarily is a probability measure we conclude that
	\[\cE(\mu_\infty)=\int_{\C}U^{\mu_\infty}(\zeta)d\mu_\infty(\zeta)\leq -\log \capacity(\sfE).\]
	The uniqueness of the minimizer $\mu_{\sfE}$ for the energy potential now implies that $\mu_{\sfE} = \mu_{\infty}$. Since $\mu_\infty$ was an arbitrary limit point and $\nu(T_n^{\sfE})$ is limit point compact as a consequence of the Banach--Alaoglu Theorem, see e.g. \cite[Theorem V.3.1]{conway90}, we finally conclude that 
	\[\nu(T_n^{\sfE})\xrightarrow{\ast}\mu_{\sfE}\]
	as $n\rightarrow \infty.$
\end{proof}
In a completely analogous fashion to how we showed Theorem \ref{thm:nbr_of_zeros_complement} we can prove the following simplification of \cite[Theorem III.4.1]{saff-totik97} using Theorem \ref{thm:blatt-saff-simkani-zeros-inside}. 

\begin{theorem}[Saff \& Totik (1997) \cite{saff-totik97}]
	\label{thm:saff-totik-point-evaluation-zeros}
	Let $\sfE\subset \C$ be a compact set of positive capacity such that the unbounded component $\Omega_{\sfE}$ of $\C\setminus \sfE$ has a boundary which is regular (in the sense of potential theory). If every bounded component of $\C\setminus \partial\Omega_\sfE$ contains a point $z_0$ such that
	\begin{equation}
		\liminf_{n\rightarrow \infty}|T_n^{\sfE}(z_0)|^{1/n}= \capacity(\sfE)
		\label{eq:saff-totik-point-evaluation-inside}
	\end{equation}
	then $\nu(T_n^{\sfE})\xrightarrow{\ast}\mu_{\sfE}$ as $n\rightarrow \infty$.
\end{theorem}
\begin{proof}
	With the intent of applying Theorem \ref{thm:blatt-saff-simkani-zeros-inside} we show that the number of zeros on any closed subset of the bounded components of $\C\setminus \sfE$ is at most $o(n)$ as $n\rightarrow \infty$. Let $\Omega$ be a bounded component of $\C\setminus \sfE$. We define
	\[h_n(z) = \frac{1}{n}\log|T_n^{\sfE}(z)|-\capacity(\sfE)+\frac{1}{n}\sum_{k=1}^{m_n}G_{\Omega}(z,z_{k,n})\]
	where $z_{1,n},\dotsc,z_{m_n,n}$ is an enumeration of the zeros of $T_n^{\sfE}$ contained in $\Omega$. We find that for all~$\zeta\in \partial \Omega$ it holds that
	\[\limsup_{z\rightarrow \zeta}h_n(z) = \frac{1}{n}\log|T_n^{\sfE}(\zeta)|-\capacity(\sfE)\leq \frac{1}{n}\log \|T_n^{\sfE}\|_{\sfE}-\capacity(\sfE)=:\varepsilon_n\]
	where, as a consequence of Theorem \ref{thm:ffs-theorem}, $\varepsilon_n\rightarrow 0$ as $n\rightarrow \infty$. From the maximum principle, we gather that
	\[h_n(z)\leq \varepsilon_n\]
	for all $z\in \Omega$. By taking $z = z_0$ we find that
	\[\frac{1}{n}\sum_{k=1}^{m_n}G_{\Omega}(z_0,z_{k,n})\leq \varepsilon_n+\capacity(\sfE)-\frac{1}{n}\log |T_n^{\sfE}(z_0)| = o(1)\]
	as $n\rightarrow \infty$. On the other hand
	\[\frac{1}{n}\sum_{k=1}^{m_n}G_{\Omega}(z_0,z_{k,n}) = \frac{1}{n}\sum_{k=1}^{m_n}G_{\Omega}(z_{k,n},z_0) = \int_{\Omega}G_{\Omega}(z,z_0)d\nu(T_n^{\sfE})(z)\rightarrow 0\]
	as $n\rightarrow \infty$. Any closed subset $A$ of $\Omega$ is compact and hence, given such a set, there exists some~$c>0$ such that $G_{\Omega}(z,z_0)\geq c$. It now follows that
	\begin{align*}
		\limsup_{n\rightarrow \infty}\nu(T_n^{\sfE})(A)&\leq \limsup_{n\rightarrow \infty}\frac{1}{c}\int_{A}G_{\Omega}(z,z_0)d\nu(T_n^{\sfE})(z)\\&\leq \limsup_{n\rightarrow \infty}\frac{1}{c}\int_{\Omega}G_{\Omega}(z,z_0)d\nu(T_n^{\sfE})(z)=0.
	\end{align*}
	Through this chain of reasoning we have verified that the hypothesis of Theorem \ref{thm:blatt-saff-simkani-zeros-inside} is satisfied and therefore $\nu(T_n^{\sfE})\xrightarrow{\ast} \mu_{\sfE}$ as $n\rightarrow \infty$.
\end{proof}

\begin{remark}
	If a sequence of monic polynomials $\{P_n\}$, where $\deg P_n = n$, satisfies \[\limsup_{n\rightarrow \infty}\|P_n\|^{1/n}_{\sfE} = \capacity(\sfE)\] then we say that $P_n$ is \textit{asymptotically extremal} on $\sfE$, a terminology originating from \cite[p. 169]{saff-totik97}. By Theorem \ref{thm:ffs-theorem}, the sequence $\{T_n^{\sfE}\}$ is asymptotically extremal for any compact set~$\sfE$. Many results on weak-star limits of zero counting measures as in \eqref{eq:zero_counting_measure} can be phrased in terms of asymptotic extremality and therefore implicitly hold for Chebyshev polynomials. As an example, Theorems \ref{thm:blatt-saff-simkani-zeros-inside}, \ref{thm:saff_totik_chebyshev_asymptotics_unbounded_complement}, and \ref{thm:saff-totik-point-evaluation-zeros} all hold in this extended setting. A generalization of Theorem \ref{thm:blatt-saff-simkani-zeros-inside} was shown by Grothmann in \cite{grothmann90} and can be found in \cite[Theorem 2.1.1]{andrievskii-blatt01}. 
\end{remark}

It is not always the case that the zeros of $T_n^{\sfE}$ approach the outer boundary of $\sfE$. In particular this can never happen for closures of analytic Jordan domains.

\begin{theorem}[Saff \& Totik (1990) \cite{saff-totik90}]
	Let $\sfE\subset \C$ be a compact set with connected interior and connected complement. There exists a neighborhood $U$ of $\partial \sfE$ and $N\in \N$ such that 
	\[\nu(T_n^{\sfE})(U) = 0,\quad n\geq N\]
	if and only if $\partial\sfE$ is an analytic Jordan curve. 
	\label{thm:saff-totik-zeros}
	\end{theorem}

This entails that the zeros of $T_n^{\sfE}$ stay strictly away from a neighborhood of $\partial \sfE$ for large~$n$ precisely when $\partial\sfE$ is analytic. The proof in \cite{saff-totik90} is performed using a comparison with the Faber polynomials which exhibit this very property. A local result on asymptotic zero distributions of Chebyshev polynomials is given in the following.

\begin{theorem}[Christiansen, Simon \& Zinchenko (2020) \cite{christiansen-simon-zinchenko-IV}]
	Let $\sfE$ be a polynomially convex compact set and $U$ an open connected set with connected complement so that $U\cap \partial \sfE$ is a continuous arc that divides $U$ into two pieces: one contained in the interior of $\sfE$ and one contained in $\C\setminus \sfE$. If 
	\[\liminf_{n\rightarrow \infty}\nu(T_n^{\sfE})(U)=0\]
	then $U\cap \partial \sfE$ is an analytic arc.
\end{theorem}

A geometric condition which is related to the zero distribution of $T_n^{\sfE}$ can be found in~\cite{saff-stylianopoulos15}. There they introduce the notion of a \textit{non-convex type singularity}. We state a simplified form that follows from their result. 
\begin{theorem}[Saff \& Stylianopoulos (2015) \cite{saff-stylianopoulos15}]
	\label{thm:saff_stylianopoulos}
	Let \(\sfE\subset \C\) denote a Jordan curve with a non-convex type singularity. Then
	\[\nu(T_n^{\sfE})\xrightarrow{\ast}\mu_{\sfE},\qquad n\to\infty.\]
\end{theorem}
One simple example of a Jordan curve with a non-convex type singularity is a non-convex polygon. If \(\sfE\) is such a curve then $\nu(T_n^{\sfE})\xrightarrow{\ast} \mu_{\sfE}$ as $n\rightarrow \infty$. We refer the reader to \cite[Corollary 2.2]{saff-stylianopoulos15} and the subsequent discussion for an in depth discussion of the implications of their result.

\section{Computing Chebyshev polynomials}
While explicit formulas only exist for Chebyshev polynomials in very special cases there are algorithms which allow for their numerical determination. Classically developed, \textit{the Remez algorithm} allows for the computation of the Chebyshev polynomials corresponding to real sets, see \cite{iske18, remez34-1,remez34-2}. This algorithm, however, relies on the alternating properties of Chebyshev polynomials and is therefore not applicable in the complex setting. An algorithm from \cite{foucartLasserre19} provides a different approach to computing weighted Chebyshev polynomials on unions of intervals. 

To facilitate the computation of the complex counterparts of Chebyshev polynomials one can use an algorithm developed by Tang in \cite{tang87,tang88} and later refined in \cite{fischer-modersitzki-93}. The aim of this algorithm, adapted to the setting of Chebyshev polynomials, is to recover the data from Theorem \ref{thm:rivlin_shapiro}, namely 
\begin{align*}
	&\begin{pmatrix}
		\lambda_1& \lambda_2 &\dotsc & \lambda_m
	\end{pmatrix}\in (0,1)^{m}, \\
	&\begin{pmatrix}
		z_1& z_2 &\dotsc & z_m
	\end{pmatrix}\in \operatorname{Ext}(wT_n^{\sfE,w},\sfE)^{m}, \\
	& \begin{pmatrix}
		\arg T_n^{\sfE,w}(z_1)&\dotsc & \arg T_n^{\sfE,w}(z_m)
	\end{pmatrix}\in [0,2\pi)^{m},
\end{align*}
and construct the corresponding Chebyshev polynomial from these data. In \cite{hubner-rubin25} this algorithm is applied to compute Chebyshev polynomials corresponding to a plethora of different sets. These numerical experiments clearly suggest certain behavior of Chebyshev polynomials both concerning the asymptotic distribution of their zeros as well as the asymptotic behavior of their norms. 
To give an example of the effectiveness of this numerical approach, we emphasize that 
computational experiments based on Tang's algorithm have played a central role 
in the formulation and refinement of several recent hypotheses and results. 
In particular, such experiments proved useful in the development of 
\cite[Theorem~1]{bergman-rubin24}, 
\cite[Theorem~1.2]{christiansen-eichinger-rubin23}, 
\cite[Theorem~1.4]{christiansen-rubin25}, 
\cite[Theorem~2.1]{minadiaz-rubin25}, and 
\cite{minadiaz-rubin-wennman25}.
An experimental approach to the study of Chebyshev polynomials, including a detailed implementation of Tang's algorithm for their numerical computation, is presented in \cite{hubner-rubin25}. We illustrate zeros of Chebyshev polynomials computed using this algorithm in Figures \ref{fig:hyp} and \ref{fig:lem}.

\begin{figure}[h!]\centering
	\includegraphics[width = 0.8\textwidth]{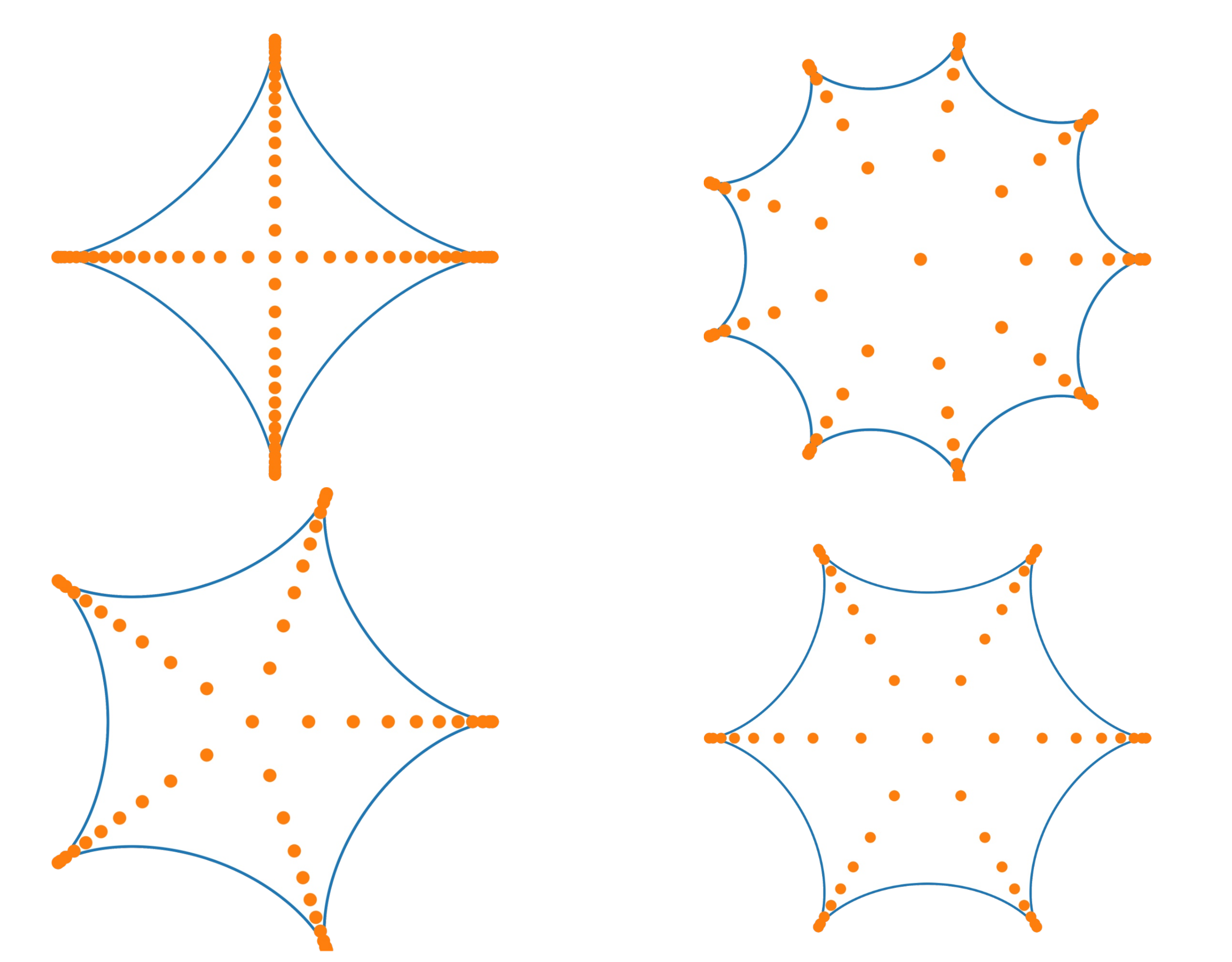}
	\caption{Zeros of Chebyshev polynomials corresponding to $m$-cusped Hypoycloids}
	\label{fig:hyp}
\end{figure}
\begin{figure}[h!]\centering
	\includegraphics[width = 0.5\textwidth]{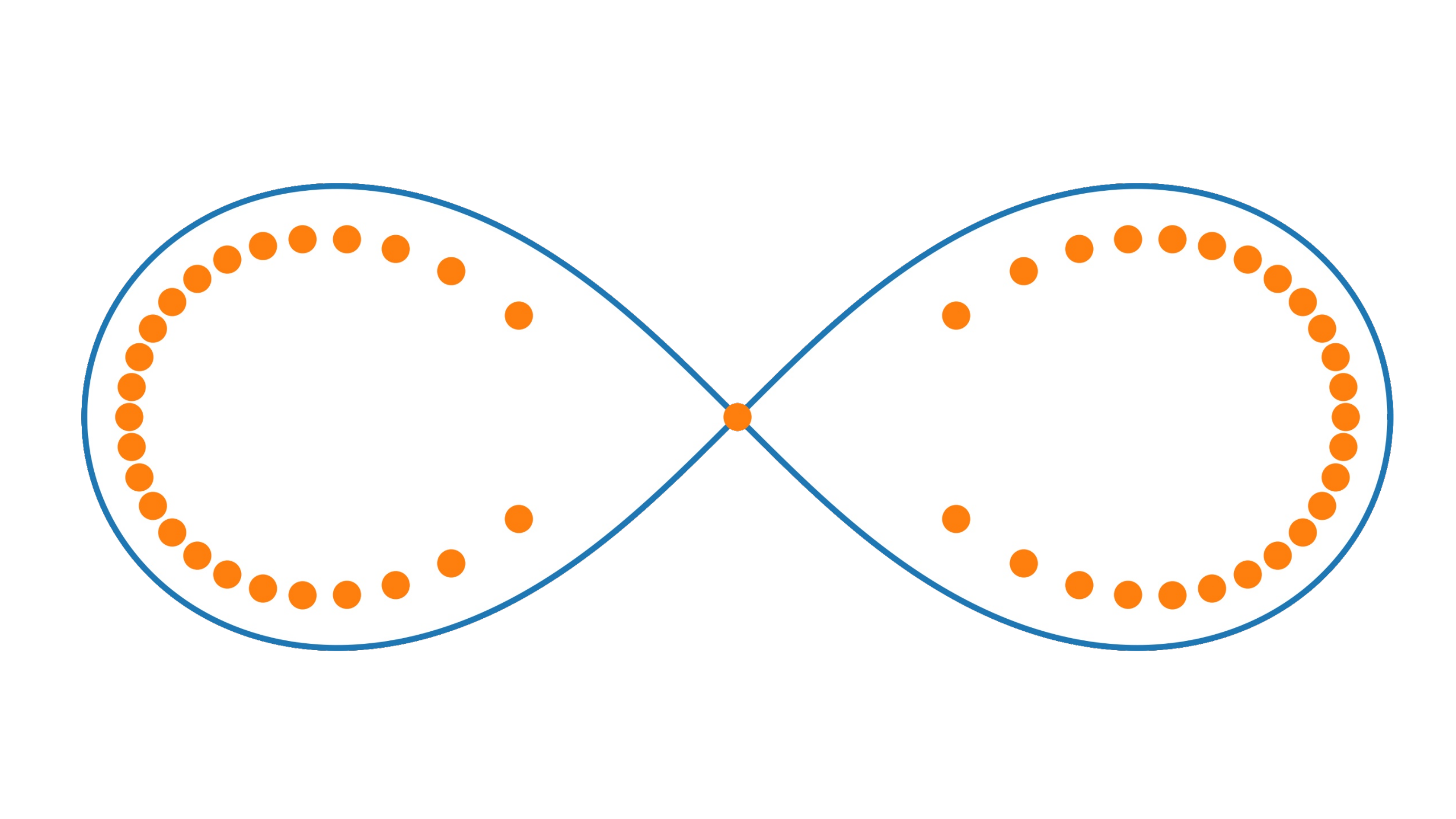}
	\caption{Zeros of Chebyshev polynomials corresponding to the Bernoulli lemniscate $\{z:|z^2-1|=1\}$}
	\label{fig:lem}
\end{figure}

\section{Open problems}

We collect several seemingly open problems concerning Chebyshev polynomials in the complex plane.
Our starting point is a classical question posed by Pommerenke \cite[Problem~4.4]{pommerenke72}.
Although a negative answer has been claimed in the literature, no complete proof appears to have been
published. Moreover, recent work of Andrievskii \cite{andrievskii17} casts doubt on the validity of the
proposed counterexample. We believe that resolving this problem is fundamental for understanding the
growth of Chebyshev polynomials on general compact sets, see also the remark after \cite[Theorem 1]{andrievskii-nazarov19}.

\begin{problem}[Pommerenke {\cite[Problem~4.4]{pommerenke72}}]\label{prob:bnd_Widom}
	Let $\sfE \subset \C$ be a compact connected set with positive logarithmic capacity.
	Is the sequence $\{\cW_n(\sfE)\}_{n=1}^\infty$ bounded?
\end{problem}

An affirmative answer would demonstrate that Chebyshev and Faber polynomials may exhibit
substantially different behavior once boundary regularity assumptions are relaxed; see also
\cite{minadiaz-rubin-wennman25}. This question was asked in the case of finitely many components in \cite[Open Problem 3.1]{christiansen-simon-zinchenko22}.

Problem~\ref{prob:bnd_Widom} also serves as a natural entry point for investigating whether
Parreau--Widom type bounds extend beyond the real line, recall \eqref{eq:parreau_widom_bound}. It is natural to ask whether Theorem~\ref{thm:christiansen-simon-zinchenko} admits an analogue in the
complex plane. As we saw in \eqref{eq:4} the constant \(2\) in \eqref{eq:parreau_widom_bound} needs to be replaced by at least \(4\).

\begin{problem}
	Does there exist a constant \(C\geq 4\) such that if $\sfE \subset \C$ is a compact, regular Parreau--Widom set, then
	\[
		\cW_n(\sfE) \le C\exp\bigl(\parreauwidom(\sfE)\bigr),
		\qquad n \in \N?
	\]
\end{problem}

Closely related to Problem~\ref{prob:bnd_Widom} is the task of identifying compact sets that exhibit
large Widom factors. In particular, one may ask whether extremal sets exist at fixed degree.

\begin{problem}
	Fix $n \in \N$.
	Does there exist a compact set $\sfE(n)$ of positive logarithmic capacity that maximizes
	\[
		\sfE \longmapsto \cW_n(\sfE)
	\]
	among all such sets? If so, characterize such a set.
\end{problem}

For piecewise Dini-smooth Jordan curves, Theorem~\ref{thm:wennman} shows that
\[
	\cW_n(\sfE) \to 1,
	\qquad n \to \infty.
\]
Thus, in order to produce large Widom factors, one must necessarily consider classes of sets with
weaker boundary regularity.

As demonstrated by Theorem~\ref{thm:cheb_arc}, Jordan arcs already lead to larger limiting values of~$\cW_n(\sfE)$. However, even in this setting, the precise asymptotic behavior remains unknown. Recently, a potential limit value was conjectured in \cite{christiansen-simon-zinchenko22}. Let us state this conjecture in the unweighted setting as the following open problem.

Let $\sfE$ be a $C^{2+\alpha}$ Jordan arc, and denote its complement by
\[
	\Omega_{\sfE} = \RS \setminus \sfE.
\]
Let $\Phi : \Omega_{\sfE} \to \RS \setminus \overline{\D}$ be the canonical conformal map satisfying \eqref{eq:canonical_conformal_map}, and let
$\Phi_\pm$ denote its boundary values taken from the two sides of the arc.
By \cite[Eq.~(A.3)]{totik14_2}, the equilibrium measure of $\sfE$ is given by
\[
	\mu_{\sfE}
	=
	\frac{1}{2\pi}\Bigl(|\Phi_+'(z)| + |\Phi_-'(z)|\Bigr)\,|dz|.
\]
Following \cite{widom69}, define
\[
	R(\infty)
	\coloneqq
	\exp\!\left(
		\int_{\sfE}
		\log\!\left(
			\frac{1}{2\pi}\bigl(|\Phi_+'(\zeta)| + |\Phi_-'(\zeta)|\bigr)
		\right)
		d\mu_{\sfE}(\zeta)
	\right).
\]

\begin{problem}[Christiansen, Simon \& Zinchenko {\cite[Conjecture 3.4]{christiansen-simon-zinchenko22}}]
	Let $\sfE$ be a $C^{2+\alpha}$ Jordan arc.
	Is it true that
	\[
		\lim_{n \to \infty} \cW_n(\sfE)
		=
		2\pi\,R(\infty)\,\capacity(\sfE)?
	\]
\end{problem}

In the case where $\sfE$ is a compact interval or a circular arc then this is the correct limit.

Another natural question concerns the sharpness of Szeg\H{o}'s lower bound and the role played by
boundary smoothness.
As previously discussed, there are known examples of quasicircles for which \[\limsup_{n\rightarrow \infty}\cW_n(\sfE)>1,\] see \cite{barnsley-geronimo-harrington83, kamo-borodin94, stawiska96}. These examples are Julia sets of quadratic polynomials whose Chebyshev polynomials can be explicitly determined for subsequences of degrees. Consequently, sufficient conditions to ensure that Widom factors corresponding to closed Jordan domain converge to $1$ require some regularity of the bounding curve. 
\begin{problem}
	Does there exist a Jordan curve $\sfE$ such that
	\[
		\liminf_{n \to \infty} \cW_n(\sfE) > 1?
	\]
\end{problem}

Turning to the relationship between Chebyshev and Faber polynomials, let $\mathcal{F}_n^{\sfE}$
denote the monic Faber polynomial associated with a compact set $\sfE$.
By Theorem~\ref{thm:minadiaz}, one always has
\[
	\lim_{r \to \infty} T_n^{\sfE(r)} = \mathcal{F}_n^{\sfE},
\]
where $\sfE(r)$ denotes the level curve of the Green function.
In certain special cases, such as intervals, circles, and lemniscates (for suitable degrees), it is
known that
\[
	T_n^{\sfE} = \mathcal{F}_n^{\sfE}.
\]
This leads to the question of whether the two families necessarily coincide on all larger level curves if they coincide on one.

\begin{problem}[{\cite[Problem 2.2]{minadiaz-rubin25}}]
	Let $\sfE \subset \C$ be a compact set such that
	\[
		T_n^{\sfE} = \mathcal{F}_n^{\sfE}.
	\]
	Is it necessarily true that
	\[
		T_n^{\sfE(r)} = \mathcal{F}_n^{\sfE},
		\qquad r > 1?
	\]
\end{problem}

Finally, while the asymptotic zero distribution of Faber polynomials is relatively well understood,
much less is known for Chebyshev polynomials. Indeed in the context of Chebyshev polynomials relative to a compact set \(K\) we have the following quote.
\begin{quote}
\emph{It seems to be a very difficult problem to determine the distribution 
of the zeros (if it exists at all) for general~$K$.}
\hfill --- Saff \& Totik~\cite{saff-totik90}
\end{quote}
 infinite compact subset of $\C$ with
connected interior and complement, then the zeros of $T_n^{\sfE}$ stay away from $\partial \sfE$ for
large $n$ if and only if $\partial \sfE$ is an analytic Jordan curve.
However, even in this case, the limiting zero distribution remains unclear.

\begin{problem}
	Let $\sfE$ be an analytic Jordan curve.
	What are the weak-$*$ limit points of $\nu(T_n^{\sfE})$ as $n \to \infty$?
\end{problem}

For Faber polynomials, it is known from \cite{kuijlaars-saff95} and later \cite{minadiaz09} that if
$\sfE$ is a piecewise analytic Jordan curve with corners, then
\[
	\nu(F_{n_k}^{\sfE}) \xrightarrow{\ast} \mu_{\sfE},
	\qquad k \to \infty
\]
for some subsequence of degrees \(n_k\).
We know from Theorem \ref{thm:saff_stylianopoulos} that if \(\sfE\) is a non-convex polygon, then \(\nu(T_n^{\sfE})\xrightarrow{\ast}\mu_{\sfE}\) as \(n\to\infty\). It would be interesting to know the following.	
\begin{problem}[{\cite[Open Problem 3.10]{christiansen-simon-zinchenko22}}]
	Let $\sfE$ denote a convex polygon.
	What are the weak-$*$ limit points of $\nu(T_n^{\sfE})$ as $n \to \infty$?
\end{problem}
This question could be rephrased in the greater generality of piecewise analytic Jordan curves with corners.

A potentially simpler model problem for studying zeros of Chebyshev polynomials but where nothing seems to be known is the following, see also Figure \ref{fig:lem}.

\begin{problem}
	Let $P$ be a monic polynomial of degree $m$, and define
	\[
		\sfE(r) = \{ z \in \C : |P(z)| = r^m \}.
	\]
	What are the weak-$*$ limit points of $\nu(T_n^{\sfE(r)})$ as $n \to \infty$
	along subsequences with $n \not\equiv 0 \pmod m$?
\end{problem}

Note that $T_{nm}^{\sfE(r)} = P(z)^n$, and hence this case is already understood.

\section*{Acknowledgement} The author is grateful for support from \textit{Odysseus Grant G0DDD23N from the Research Foundation Flanders} and \textit{The G\"{o}ran Gustafsson Foundation for Research in Natural Sciences and Medicine}. This text is a refined version of the Doctoral Thesis \cite{rubin24-2}. Constructive comments and suggestion supplied to the author by his former supervisor Jacob S. Christiansen as well as Bernhard Beckermann, Arno Kuijlaars, Andrei Mart\'{i}nez-Finkelshtein, Klaus Schiefermayr, Margaret Stawiska Friedland, and the anonymous referees are greatly appreciated.

\end{document}